\RequirePackage{snapshot}
\documentclass[english,3p]{elsarticle}
\usepackage[T1]{fontenc}
\usepackage[utf8]{inputenc}
\usepackage{float}
\usepackage{units}
\usepackage{amsmath}
\usepackage{amsthm}
\usepackage{amssymb}
\usepackage{stmaryrd}
\usepackage{graphicx}
\usepackage{wasysym}

\makeatletter

\providecommand{\tabularnewline}{\\}

\numberwithin{equation}{section}
\numberwithin{figure}{section}
\theoremstyle{plain}
\newtheorem{thm}{\protect\theoremname}
\theoremstyle{plain}
\newtheorem{lem}[thm]{\protect\lemmaname}
\theoremstyle{remark}
\newtheorem{rem}[thm]{\protect\remarkname}

\usepackage{algpseudocode,algorithm,algorithmicx}
\usepackage{tikz}
\usetikzlibrary{shapes,arrows,calc,decorations.pathreplacing}

\@ifundefined{showcaptionsetup}{}{%
 \PassOptionsToPackage{caption=false}{subfig}}
\usepackage{subfig}
\makeatother

\usepackage{babel}
\providecommand{\lemmaname}{Lemma}
\providecommand{\remarkname}{Remark}
\providecommand{\theoremname}{Theorem}

\begin{document}

\title{Positivity-preserving and entropy-bounded discontinuous Galerkin
method for the chemically reacting, compressible Navier-Stokes equations}

\author[NRL]{Eric J. Ching}
\author[NRL]{Ryan F. Johnson}
\author[UFL]{Sarah Burrows}
\author[UCF]{Jacklyn Higgs}
\author[NRL]{Andrew D. Kercher}

\address[NRL]{Laboratories for Computational Physics and Fluid Dynamics, U.S. Naval Research Laboratory, 4555 Overlook Ave SW, Washington, DC 20375}
\address[UFL]{Department of Mechanical \& Aerospace Engineering, University of Florida, Gainesville, FL 32611}
\address[UCF]{Department of Mechanical and Aerospace Engineering, University of Central Florida, Orlando, FL 32816}

\begin{abstract}
This article concerns the development of a fully conservative, positivity-preserving,
and entropy-bounded discontinuous Galerkin scheme for the multicomponent,
chemically reacting, compressible Navier-Stokes equations with complex
thermodynamics. In particular, we extend to viscous flows the fully
conservative, positivity-preserving, and entropy-bounded discontinuous
Galerkin method for the chemically reacting Euler equations that we
previously introduced. An important component of the formulation is
the positivity-preserving Lax-Friedrichs-type viscous flux function
devised by Zhang {[}\emph{J. Comput. Phys.}, 328 (2017), pp. 301-343{]},
which was adapted to multicomponent flows by Du and Yang {[}\emph{J.
Comput. Phys.}, 469 (2022), pp. 111548{]} in a manner that treats
the inviscid and viscous fluxes as a single flux. Here, we similarly
extend the aforementioned flux function to multicomponent flows but
separate the inviscid and viscous fluxes, resulting in a different
dissipation coefficient. This separation of the fluxes allows for
use of other inviscid flux functions, as well as enforcement of entropy
boundedness on only the convective contribution to the evolved state,
as motivated by physical and mathematical principles. We also detail
how to account for boundary conditions and incorporate previously
developed techniques to reduce spurious pressure oscillations into
the positivity-preserving framework. Furthermore, potential issues
associated with the Lax-Friedrichs-type viscous flux function in the
case of zero species concentrations are discussed and addressed. Comparisons
between the Lax-Friedrichs-type viscous flux function and more conventional
flux functions are provided, the results of which motivate an adaptive
solution procedure that employs the former only when the element-local
solution average has negative species concentrations, nonpositive
density, or nonpositive pressure. The resulting formulation is compatible
with curved, multidimensional elements and general quadrature rules
with positive weights. A variety of multicomponent, viscous flows
is computed, ranging from a one-dimensional shock tube problem to
multidimensional detonation waves and shock/mixing-layer interaction.

\end{abstract}
\maketitle
\global\long\def\middlebar{\,\middle|\,}%
\global\long\def\average#1{\left\{  \!\!\left\{  #1\right\}  \!\!\right\}  }%
\global\long\def\expnumber#1#2{{#1}\mathrm{e}{#2}}%
 \newcommand*{\horzbar}{\rule[.5ex]{2.5ex}{0.5pt}}

\global\long\def\revisionmath#1{\textcolor{red}{#1}}%

\makeatletter \def\ps@pprintTitle{  \let\@oddhead\@empty  \let\@evenhead\@empty  \def\@oddfoot{\centerline{\thepage}}  \let\@evenfoot\@oddfoot} \makeatother

\let\svthefootnote\thefootnote\let\thefootnote\relax\footnotetext{\\ \hspace*{65pt}DISTRIBUTION STATEMENT A. Approved for public release. Distribution is unlimited.}\addtocounter{footnote}{-1}\let\thefootnote\svthefootnote

\section{Introduction\label{sec:Introduction}}

In the past two decades, interest in the discontinuous Galerkin (DG)
method for fluid flow simulations has surged dramatically. This method
benefits from arbitrarily high order of accuracy on unstructured grids,
as well as a compact stencil and high arithmetic intensity suited
for modern computing systems. However, one of the primary obstacles
to widespread use of this numerical scheme is its susceptibility to
nonlinear instabilities in underresolved regions and near non-smooth
features. Robustness is an even greater concern when mixtures of thermally
perfect gases and chemical reactions are considered~\citep{Ban20,Joh20_2}.
For instance, it is well-known that fully conservative schemes fail
to maintain pressure equilibrium~\citep{Abg88,Kar94,Abg96}, leading
to the generation of spurious pressure oscillations that can cause
solver divergence. A number of quasi-conservative methods, such as
the double-flux technique~\citep{Abg01,Lv15,Bil11}, have been proposed
to circumvent this issue, typically at the expense of energy conservation.
Recently, Johnson and Kercher introduced a fully conservative DG scheme
that more effectively maintains pressure equilibrium in smooth regions
of the flow~\citep{Joh20_2}. Strang splitting was employed to decouple
the temporal integration of the convective and diffusive operators
from that of stiff chemical source terms. Artificial viscosity was
used to stabilize the solution near shocks and other non-smooth features.
However, artificial viscosity alone is often not sufficient to guarantee
stability. To further increase robustness, we made key advancements
to this fully conservative DG method, focusing on the inviscid case,
to ensure satisfaction of the positivity property (i.e., nonnegative
species concentrations, positive density, and positive pressure) and
an entropy bound based on the minimum entropy principle for the multicomponent
Euler equations~\citep{Gou20,Chi22}, which states that the spatial
minimum of specific thermodynamic entropy of entropy solutions is
a nondecreasing function of time. The main ingredients of this DG
formulation~\citep{Chi22,Chi22_2} are (a) an invariant-region-preserving
inviscid flux function~\citep{Jia18}, (b) a simple linear-scaling
limiter~\citep{Zha10}, (c) satisfaction of a time-step-size constraint
for the transport step with a strong-stability-preserving explicit
time integrator, (d) incorporation of the techniques introduced by
Johnson and Kercher~\citep{Joh20_2} that reduce spurious pressure
oscillations in smooth flow regions, and (e) an entropy-stable DG
discretization in time based on diagonal-norm summation-by-parts operators
for the reaction step. It was found that the formulation was capable
of robustly and accurately computing complex inviscid, reacting flows
using high-order polynomials and relatively coarse meshes. Enforcement
of entropy boundedness was critical for stability in simulations of
multidimensional detonation waves.

The consideration of viscous flows brings about additional complications.
Using conventional viscous flux functions, such as the second form
of Bassi and Rebay (BR2)~\citep{Bas00} and the symmetric interior
penalty method (SIPG)~\citep{Har08}, positivity is not guaranteed
to be maintained. Specifically, it is possible for the constraint
on the time step size to be arbitrarily small for the solution to
satisfy said property. To remedy this problem, Zhang~\citep{Zha17}
introduced a Lax-Friedrichs-type viscous flux function for the monocomponent
Navier-Stokes equations, accompanied by a strictly positive upper
bound on the time step size to guarantee satisfaction of the positivity
property. Although it may be surprising that the viscous flux can
be a primary source of negative concentrations, nonpositive density,
and/or nonpositive pressure, we call attention to certain numerical
challenges specific to multicomponent-flow simulations: not only are
nonlinear instabilities more likely to occur, but also species concentrations
are typically close or equal to zero, such that small numerical errors
can easily lead to negative concentrations. Furthermore, in the monocomponent
case, the mass conservation equation is identical between the Euler
system and the Navier-Stokes system; therefore, the diffusive operator
does not directly contribute to negative densities. However, this
is not true in the multicomponent case, which means that the viscous
flux can indeed be largely responsible for negative concentrations.
Note that many multicomponent-flow codes simply ``clip'' negative
species concentrations, but such an intrusive strategy violates mass
conservation and pollutes the solution with low-order errors.

Du and Yang~\citep{Du22} recently extended the aforementioned Lax-Friedrichs-type
viscous flux function to multicomponent flows. Specifically, they
combined the inviscid and viscous fluxes into a single flux, such
that the resulting dissipation coefficient accounts for both fluxes
simultaneously. Entropy boundedness was not considered. Instead of
operator splitting, they employed an exponential multistage/multistep,
explicit time integration scheme~\citep{Hua18,Du19,Du19_2} that
can handle stiff source terms. Although in the present study we use
operator splitting since it has proven successful to date and its
accuracy is less reliant on ``well-prepared'' initial conditions~\citep{Hua18,Du19,Du19_2},
exponential multistage/multistep time integrators are indeed worthy
of future investigation.

In this work, we develop a fully conservative, positivity-preserving,
and entropy-bounded DG method for the compressible, multicomponent,
chemically reacting Navier-Stokes equations. We focus on the transport
step since the treatment of stiff chemical source terms is identical
to that in the inviscid case. Enforcement of a lower bound on the
specific thermodynamic entropy is performed on only the convective
contribution to the evolved state since the viscous flux function
is not fully compatible with said entropy bound. This was also done
by Dzanic and Witherden~\citep{Dza22}, albeit in a different manner,
in their entropy-based filtering framework. Furthermore, at least
in the monocomponent, calorically perfect setting, the minimum entropy
principle does not hold for the Navier-Stokes equations unless the
thermal diffusivity is zero~\citep{Tad86,Gue14}. Although such analysis
has not yet been performed for the multicomponent Navier-Stokes equations
with the thermally perfect gas model, we do not expect the conclusion
to change. Our primary contributions are as follows:
\begin{itemize}
\item We extend the aforementioned positivity-preserving Lax-Friedrichs-type
viscous flux function~\citep{Zha17} to multicomponent flows. Specifically,
unlike in~\citep{Du22}, we treat the inviscid and viscous fluxes
separately, resulting in a different dissipation coefficient. The
rationale for separating the fluxes is twofold. First, enforcing a
bounded entropy on the convective contribution necessitates isolating
the fluxes. Secondly, in our experience, we have found the HLLC inviscid
flux function to perform more favorably than the Lax-Friedrichs inviscid
flux function. We also discuss the treatment of boundary conditions
in more detail.
\item Entropy boundedness is enforced on only the convective contribution
in a rigorous manner that maintains full compatibility with the positivity
property.
\item We discuss potential issues associated with the Lax-Friedrichs-type
viscous flux function if any of the species concentrations is zero.
This is true regardless of whether the inviscid and viscous fluxes
are treated simultaneously or separately. A remedy for this pathological
case is proposed.
\item We incorporate the techniques by Johnson and Kercher~\citep{Joh20_2}
that reduce spurious pressure oscillations into the positivity-preserving
framework, which imposes an additional constraint on the time step
size.
\item The performance of the Lax-Friedrichs-type viscous flux function is
assessed. Optimal convergence for smooth flows is observed. However,
comparisons with the BR2 scheme indicate that when possible, the latter
is generally still preferred. As such, we employ an adaptive solution
procedure that only employs the Lax-Friedrichs-type viscous flux function
when necessary.
\item The proposed formulation is compatible with curved, multidimensional
elements of arbitrary shape and general quadrature rules with positive
weights. We first apply it to a series of one-dimensional viscous
flows: advection-diffusion of a thermal bubble, a premixed flame,
and shock-tube flow. More complex viscous flow problems are then considered,
namely a two-dimensional detonation wave enclosed by adiabatic walls
and three-dimensional shock/mixing-layer interaction.  Just as in
the inviscid case, enforcement of the entropy bound significantly
improves the stability of the solution.
\end{itemize}
The remainder of this article is organized as follows. The governing
equations, transport properties, and thermodynamic relations are summarized
in Section~\ref{sec:governing_equations}, followed by a review of
the basic DG discretization in Section~\ref{sec:DG-discretization}.
Section~\ref{sec:positivity-preserving-entropy-bounded-DG} presents
the positivity-preserving and entropy-bounded DG method for the transport
step. Results for a variety of test cases are given in the next section.
The paper concludes with some final remarks.

\section{Governing equations\label{sec:governing_equations}}

The compressible, multicomponent, chemically reacting Navier-Stokes
equations in $d$ spatial dimensions are given by
\begin{equation}
\frac{\partial y}{\partial t}+\nabla\cdot\mathcal{F}\left(y,\nabla y\right)-\mathcal{S}\left(y\right)=0,\label{eq:conservation-law-strong-form}
\end{equation}
where $y$ is the state vector, $\nabla y$ is its spatial gradient,
$t$ is time, $\mathcal{F}$ is the flux, and $\mathcal{S}=\left(0,\ldots,0,0,\omega_{1},\ldots,\omega_{n_{s}}\right)^{T}$
is the chemical source term, with $\omega_{i}$ corresponding to the
production rate of the $i$th species. The physical coordinates are
denoted by $x=(x_{1},\ldots,x_{d})$. The vector of state variables
is expanded as

\begin{equation}
y=\left(\rho v_{1},\ldots,\rho v_{d},\rho e_{t},C_{1},\ldots,C_{n_{s}}\right)^{T},\label{eq:reacting-navier-stokes-state}
\end{equation}
where $\rho$ is density, $v=\left(v_{1},\ldots,v_{d}\right)$ is
the velocity vector, $e_{t}$ is the specific total energy, $C=\left(C_{1},\ldots,C_{n_{s}}\right)$
is the vector of molar concentrations, and $n_{s}$ is the number
of species. The partial density of the $i$th species is defined as
\[
\rho_{i}=W_{i}C_{i},
\]
where $W_{i}$ is the molecular weight of the $i$th species, from
which the density can be computed as

\[
\rho=\sum_{i=1}^{n_{s}}\rho_{i}.
\]

\noindent The mole and mass fractions of the $i$th species are given
by
\[
X_{i}=\frac{C_{i}}{\sum_{i=1}^{n_{s}}C_{i}},\quad Y_{i}=\frac{\rho_{i}}{\rho}.
\]
The equation of state for the mixture is written as
\begin{equation}
P=R^{0}T\sum_{i=1}^{n_{s}}C_{i},\label{eq:EOS}
\end{equation}
where $P$ is the pressure, $T$ is the temperature, and $R^{0}$
is the universal gas constant. The specific total energy is the sum
of the mixture-averaged specific internal energy, $u$, and the specific
kinetic energy, written as

\[
e_{t}=u+\frac{1}{2}\sum_{k=1}^{d}v_{k}v_{k},
\]
where the former is the mass-weighted sum of the specific internal
energies of each species, given by
\[
u=\sum_{i=1}^{n_{s}}Y_{i}u_{i}.
\]
With the thermally perfect gas model, $u_{i}$ is defined as
\[
u_{i}=h_{i}-R_{i}T=h_{\mathrm{ref},i}+\int_{T_{\mathrm{ref}}}^{T}c_{p,i}(\tau)d\tau-R_{i}T,
\]
where $h_{i}$ is the specific enthalpy of the $i$th species, $R_{i}=R^{0}/W_{i}$,
$T_{\mathrm{ref}}$ is the reference temperature of 298.15 K, $h_{\mathrm{ref},i}$
is the reference-state species formation enthalpy, and $c_{p,i}$
is the specific heat at constant pressure of the $i$th species, which
is approximated with a polynomial as a function of temperature based
on the NASA coefficients~\citep{Mcb93,Mcb02}, i.e.,
\begin{equation}
c_{p,i}=\sum_{k=0}^{n_{p}}a_{ik}T^{k}.\label{eq:specific_heat_polynomial}
\end{equation}

\noindent The mixture-averaged specific thermodynamic entropy is obtained
via a mass-weighted sum of the specific entropies of each species
as
\[
s=\sum_{i=1}^{n_{s}}Y_{i}s_{i},
\]
with $s_{i}$ defined as
\[
s_{i}=s_{\mathrm{ref},i}^{o}+\int_{T_{\mathrm{ref}}}^{T}\frac{c_{v,i}(\tau)}{\tau}d\tau-R_{i}\log\frac{C_{i}}{C_{\mathrm{ref}}},
\]
where $s_{\mathrm{ref},i}^{o}$ is the species formation entropy at
the reference temperature and reference pressure, $P_{\mathrm{ref}}$,
of 1 atm, $c_{v,i}=c_{p,i}-R_{i}$ is the specific heat at constant
volume of the $i$th species, and $C_{\mathrm{ref}}=P_{\mathrm{ref}}/R^{0}T_{\mathrm{ref}}$
is the reference concentration.

The flux can be expressed as the difference between the convective
flux, $\mathcal{F}^{c}$, and the viscous flux, $\mathcal{F}^{v}$,
i.e.,

\[
\mathcal{F}\left(y,\nabla y\right)=\left(\mathcal{F}^{c}\left(y\right)-\mathcal{F}^{v}\left(y,\nabla y\right)\right),
\]
where the $k$th spatial components are defined as
\begin{equation}
\mathcal{F}_{k}^{c}\left(y\right)=\left(\rho v_{k}v_{1}+P\delta_{k1},\ldots,\rho v_{k}v_{d}+P\delta_{kd},v_{k}\left(\rho e_{t}+P\right),v_{k}C_{1},\ldots,v_{k}C_{n_{s}}\right)^{T}\label{eq:reacting-navier-stokes-spatial-convective-flux-component}
\end{equation}
 and
\begin{equation}
\mathcal{F}_{k}^{v}\left(y,\nabla y\right)=\left(\tau_{1k},\ldots,\tau_{dk},\sum_{j=1}^{d}\tau_{kj}v_{j}+\sum_{i=1}^{n_{s}}W_{i}C_{i}h_{i}V_{ik}-q_{k},C_{1}V_{1k},\ldots,C_{n_{s}}V_{n_{s}k}\right)^{T},\label{eq:navier-stokes-viscous-flux-spatial-component}
\end{equation}
respectively. $\tau$ is the viscous stress tensor, $q$ is the heat
flux, and $V_{ik}$ is the $k$th spatial component of the diffusion
velocity of the $i$th species, defined as

\[
V_{ik}=\hat{V}_{ik}-\frac{\sum_{l=1}^{n_{s}}W_{l}C_{l}\hat{V}_{lk}}{\rho},\quad\hat{V}_{ik}=\frac{\bar{D}_{i}}{C_{i}}\frac{\partial C_{i}}{\partial x_{k}}-\frac{\bar{D}_{i}}{\rho}\frac{\partial\rho}{\partial x_{k}},
\]
which includes a standard correction to ensure mass conservation (i.e.,
$\sum_{i=1}^{n_{s}}W_{i}C_{i}V_{ik}=0$)~\citep{Cof81,Hou11}. $\bar{D}_{i}$
is the mixture-averaged diffusion coefficient of the $i$th species,
obtained as~\citep{Kee89}
\[
\bar{D}_{i}=\frac{1}{\bar{W}}\frac{\sum_{j=1,j\ne i}^{n_{s}}X_{j}W_{j}}{\sum_{j=1,j\ne i}^{n_{s}}X_{j}/D_{ij}},
\]
where $\bar{W}=\rho/\sum_{i}C_{i}$ is the mixture molecular weight
and $D_{ij}$ is the binary diffusion coefficient between the $i$th
and $j$th species, which is a positive function of temperature and
pressure~\citep{Kee05,cantera}. Note that $\bar{D}_{i}$ can be
nonzero for $C_{i}=0$. The $k$th spatial components of the viscous
stress tensor and the heat flux are written as

\noindent
\[
\tau_{k}\left(y,\nabla y\right)=\mu\left(\frac{\partial v_{1}}{\partial x_{k}}+\frac{\partial v_{k}}{\partial x_{1}}-\delta_{k1}\frac{2}{3}\sum_{j=1}^{d}\frac{\partial v_{j}}{\partial x_{j}},\ldots,\frac{\partial v_{d}}{\partial x_{k}}+\frac{\partial v_{k}}{\partial x_{d}}-\delta_{kd}\frac{2}{3}\sum_{j=1}^{d}\frac{\partial v_{j}}{\partial x_{j}}\right),
\]

\noindent where $\mu$ is the dynamic viscosity, calculated using
the Wilke model~\citep{Wil50}, and
\begin{eqnarray*}
q_{k}\left(y,\nabla y\right) & = & -\lambda_{T}\frac{\partial T}{\partial x_{k}},
\end{eqnarray*}
where $\lambda_{T}$ is the thermal conductivity, computed with the
Mathur model~\citep{Mat67}, respectively. The viscous flux can also
be written as
\begin{equation}
\mathcal{F}^{v}\left(y,\nabla y\right)=G\left(y\right):\nabla y\label{eq:viscous-flux-homogeneity-tensor}
\end{equation}
where $G\left(y\right)$ is the homogeneity tensor~\citep{Har13},
obtained by differentiating the viscous flux with respect to the gradient,
i.e., $G(y)=\partial\mathcal{F}^{v}/\partial\nabla y$. Additional
information on the thermodynamic relations, transport properties,
and chemical reaction rates can be found in~\citep{Joh20_2} and~\citep{Chi22}.

\section{Discontinuous Galerkin discretization\label{sec:DG-discretization}}

This section summarizes the DG discretization of Equation~\ref{eq:conservation-law-strong-form}
and the approach introduced by Johnson and Kercher~\citep{Joh20_2}
to suppress spurious pressure oscillations in smooth regions of the
flow.

Let the computational domain, $\Omega$, be partitioned by $\mathcal{T}$,
which consists of non-overlapping cells, $\kappa$, with boundaries
$\partial\kappa$. Let $\mathcal{E}=\mathcal{E_{I}}\cup\mathcal{E}_{\partial}$
be the set of interfaces, $\epsilon$, such that $\cup_{\epsilon\in\mathcal{E}}\epsilon=\cup_{\kappa\in\mathcal{T}}\partial\kappa$,
comprised of the interior interfaces,
\[
\epsilon_{\mathcal{I}}\in\mathcal{E_{I}}=\left\{ \epsilon_{\mathcal{I}}\in\mathcal{E}\middlebar\epsilon_{\mathcal{I}}\cap\partial\Omega=\emptyset\right\} ,
\]
and boundary interfaces,
\[
\epsilon_{\partial}\in\mathcal{E}_{\partial}=\left\{ \epsilon_{\partial}\in\mathcal{E}\middlebar\epsilon_{\partial}\subset\partial\Omega\right\} .
\]
At a given interior interface, there exists $\kappa^{+},\kappa^{-}\in\mathcal{T}$
such that $\epsilon_{\mathcal{I}}=\partial\kappa^{+}\cap\partial\kappa^{-}$.
$n^{+}$ is the outward-facing normal of $\kappa^{+}$, and $n^{+}=-n^{-}$.
The discrete subspace $V_{h}^{p}$ over $\mathcal{T}$ is defined
as
\begin{eqnarray}
V_{h}^{p} & = & \left\{ \mathfrak{v}\in\left[L^{2}\left(\Omega\right)\right]^{m}\middlebar\forall\kappa\in\mathcal{T},\left.\mathfrak{v}\right|_{\kappa}\in\left[\mathcal{P}_{p}(\kappa)\right]^{m}\right\} ,\label{eq:discrete-subspace}
\end{eqnarray}
where $m=n_{s}+d+1$ is the number of state variables and $\mathcal{P}_{p}(\kappa)$
in one spatial dimension is the space of polynomial functions of degree
no greater than $p$ in $\kappa$.  In multiple dimensions, the choice
of polynomial space often depends on the element type~\citep{Har13}.

To solve for the discrete solution, we require $y\in V_{h}^{p}$ to
satisfy

\begin{gather}
\sum_{\kappa\in\mathcal{T}}\left(\frac{\partial y}{\partial t},\mathfrak{v}\right)_{\kappa}-\sum_{\kappa\in\mathcal{T}}\left(\mathcal{F}^{c}\left(y,\nabla y\right),\nabla\mathfrak{v}\right)_{\kappa}+\sum_{\epsilon\in\mathcal{E}}\left(\mathcal{F}^{c\dagger}\left(y,n\right),\left\llbracket \mathfrak{v}\right\rrbracket \right)_{\epsilon}-\sum_{\epsilon\in\mathcal{E}}\left(\average{\mathcal{F}^{v}\left(y,\nabla y\right)}\cdot n-\delta^{v}\left(y,\nabla y,n\right),\left\llbracket \mathfrak{v}\right\rrbracket \right)_{\epsilon}\nonumber \\
+\sum_{\kappa\in\mathcal{T}}\left(G\left(y^{+}\right):\left(\average y-y^{+}\right)\otimes n,\nabla\mathfrak{v}\right)_{\partial\kappa}-\sum_{\kappa\in\mathcal{T}}\left(\mathcal{S}\left(y\right),\mathfrak{v}\right)_{\kappa}=0\qquad\forall\mathfrak{v}\in V_{h}^{p},\label{eq:semi-discrete-form}
\end{gather}
where $\left(\cdot,\cdot\right)$ denotes the inner product, $\mathcal{F}^{c\dagger}$
is the inviscid flux function, $\average{\cdot}$ is the average operator,
$\left\llbracket \cdot\right\rrbracket $ is the jump operator, and
$\delta^{v}$ is a viscous-flux penalty term that depends on the viscous
flux function. Note that Equation~(\ref{eq:semi-discrete-form})
corresponds to a primal formulation~\citep{Arn02,Har13}; in~\citep{Zha17},
a flux formulation is used. It is worth mentioning that the penalty
term for many conventional viscous flux functions is not a function
of the gradient, i.e., $\delta^{v}=\delta^{v}(y,n)$; however, as
will be seen in Section~\ref{subsec:Positivity-preserving-Lax-Friedrichs},
the penalty term for the proposed Lax-Friedrichs-type viscous flux
function indeed depends on the gradient. In this work, we employ the
HLLC inviscid numerical flux~\citep{Tor13}. To compute $\delta^{v}$,
we consider the BR2 scheme~\citep{Bas00} and the proposed Lax-Friedrichs-type
flux function. The jump operator, average operator, inviscid flux
function, and penalty term are defined as
\begin{align*}
\left\llbracket v\right\rrbracket =v^{+}-v^{-} & \textup{ on }\epsilon\qquad\forall\epsilon\in\mathcal{E_{I}},\\
\average y=\frac{1}{2}\left(y^{+}+y^{-}\right) & \textup{ on }\epsilon\qquad\forall\epsilon\in\mathcal{E_{I}},\\
\average{\mathcal{F}^{\nu}\left(y,\nabla y\right)}=\frac{1}{2}\left(\mathcal{F}^{\nu}\left(y^{+},\nabla y^{+}\right)+\mathcal{F}^{\nu}\left(y^{-},\nabla y^{-}\right)\right) & \textup{ on }\epsilon\qquad\forall\epsilon\in\mathcal{E_{I}},\\
\mathcal{F}^{c\dagger}\left(y,n\right)=\mathcal{F}^{c\dagger}\left(y^{+},y^{-},n\right) & \textup{ on }\epsilon\qquad\forall\epsilon\in\mathcal{E_{I}},\\
\delta^{v}\left(y,\nabla y,n\right)=\delta^{v}\left(y^{+},y^{-},\nabla y^{+},\nabla y^{-},n\right) & \textup{ on }\epsilon\qquad\forall\epsilon\in\mathcal{E_{I}},
\end{align*}
at interior interfaces and

\begin{align*}
\left\llbracket v\right\rrbracket =v^{+} & \textup{ on }\epsilon\qquad\forall\epsilon\in\mathcal{E}_{\partial},\\
\average y=y_{\partial}\left(y^{+},n^{+}\right) & \textup{ on }\epsilon\qquad\forall\epsilon\in\mathcal{E}_{\partial},\\
\average{\mathcal{F}^{\nu}\left(y,\nabla y\right)}=\mathcal{F}_{\partial}^{\nu}\left(y_{\partial}\left(y^{+},n^{+}\right),\nabla y^{+}\right) & \textup{ on }\epsilon\qquad\forall\epsilon\in\mathcal{E}_{\partial},\\
\mathcal{F}^{c\dagger}\left(y,n\right)=\mathcal{F}_{\partial}^{c\dagger}\left(y^{+},n^{+}\right) & \textup{ on }\epsilon\qquad\forall\epsilon\in\mathcal{E}_{\partial},\\
\delta^{v}\left(y,\nabla y,n\right)=\delta_{\partial}^{v}\left(y^{+},y_{\partial}\left(y^{+},n^{+}\right),\nabla y^{+},n^{+}\right) & \textup{ on }\epsilon\qquad\forall\epsilon\in\mathcal{E}_{\partial},
\end{align*}
at boundary interfaces, where $y_{\partial}\left(y^{+},n^{+}\right)$
is the boundary state, $\mathcal{F}_{\partial}^{c\dagger}\left(y^{+},n^{+}\right)$
is the inviscid boundary flux function, $\mathcal{F}_{\partial}^{\nu}\left(y_{\partial}\left(y^{+}\right),\nabla y^{+},n^{+}\right)$
is the viscous boundary flux, and $\delta_{\partial}^{v}\left(y^{+},y_{\partial}\left(y^{+},n^{+}\right),\nabla y^{+},n^{+}\right)$
is the boundary penalty term. \ref{sec:boundary-conditions} provides
a discussion of the prescription of various boundary conditions.

Strang splitting~\citep{Str68} is applied to decouple the temporal
integration of the transport operators from that of the stiff chemical
source term over a given interval $(t_{0},t_{0}+\Delta t]$ as
\begin{align}
\frac{\partial y}{\partial t}+\nabla\cdot\mathcal{F}\left(y\right)=0 & \textup{ in }\Omega\times\left(t_{0},t_{0}+\nicefrac{\Delta t}{2}\right],\label{eq:strang-splitting-1}\\
\frac{\partial y}{\partial t}-\mathcal{S}\left(y\right)=0 & \textup{ in }\left(t_{0},t_{0}+\Delta t\right],\label{eq:strang-splitting-2}\\
\frac{\partial y}{\partial t}+\nabla\cdot\mathcal{F}\left(y\right)=0 & \textup{ in }\Omega\times\left(t_{0}+\nicefrac{\Delta t}{2},t_{0}+\Delta t\right].\label{eq:strang-splitting-3}
\end{align}
Equations~(\ref{eq:strang-splitting-1}) and~(\ref{eq:strang-splitting-3})
are advanced in time using a strong-stability-preserving Runge-Kutta
method (SSPRK)~\citep{Got01,Spi02}, whereas Equation~(\ref{eq:strang-splitting-2})
is solved using a fully implicit, temporal DG discretization. Since
the reaction step is identical between the inviscid and viscous cases,
we refer the reader to~\citep{Chi22} and~\citep{Chi22_2} for more
details on the DG discretization in time for Equation~(\ref{eq:strang-splitting-2}).
Here, we focus on the transport step.

We assume a nodal basis, such that the local solution approximation
is given by
\begin{equation}
y_{\kappa}=\sum_{j=1}^{n_{b}}y_{\kappa}(x_{j})\phi_{j},\label{eq:solution-approximation}
\end{equation}
where $\phi_{j}$ is the $j$th basis function, $n_{b}$ is the number
of basis functions, and $x_{j}$ is the physical coordinate of the
$j$th node. The volume and surface integrals in Equation~(\ref{eq:semi-discrete-form})
are computed using a quadrature-free approach~\citep{Atk96,Atk98}.
Furthermore, the flux can be approximated as
\begin{equation}
\mathcal{F_{\kappa}}\left(y\right)\approx\sum_{k=1}^{n_{c}}\mathcal{F}\left(y_{\kappa}\left(x_{k}\right),\nabla y_{\kappa}\left(x_{k}\right)\right)\varphi_{k},\label{eq:flux-projection}
\end{equation}
where $n_{c}\geq n_{b}$ and $\left\{ \varphi_{1},\ldots,\varphi_{n_{c}}\right\} $
is a set of basis functions that may be different from those in Equation~(\ref{eq:solution-approximation}).
As discussed in~\citep{Joh20_2}, pressure equilibrium is (approximately)
maintained in smooth regions of the flow and at material interfaces
if $n_{c}=n_{b}$ and the integration points are in the solution nodal
set. However, if over-integration is desired (i.e., $n_{c}>n_{b}$),
the standard flux interpolation~(\ref{eq:flux-projection}) results
in the rapid generation of large spurious pressure oscillations. Therefore,
in the case of over-integration, Equation~(\ref{eq:flux-projection})
is replaced with
\begin{equation}
\mathcal{F_{\kappa}}\left(y\right)\approx\sum_{k=1}^{n_{c}}\mathcal{F}_{\kappa}\left(\widetilde{y}_{\kappa}\left(x_{k}\right),\nabla y_{\kappa}\left(x_{k}\right)\right)\varphi_{k}=\sum_{k=1}^{n_{c}}\mathcal{F}_{\kappa}^{c}\left(\widetilde{y}_{\kappa}\left(x_{k}\right)\right)\varphi_{k}-\sum_{k=1}^{n_{c}}\mathcal{F}_{\kappa}^{v}\left(\widetilde{y}_{\kappa}\left(x_{k}\right),\nabla y_{\kappa}\left(x_{k}\right)\right)\varphi_{k},\label{eq:modified-flux-projection}
\end{equation}
where
\[
\mathcal{F}_{\kappa}^{v}\left(\widetilde{y}_{\kappa}\left(x_{k}\right),\nabla y_{\kappa}\left(x_{k}\right)\right)=G\left(\widetilde{y}_{\kappa}\left(x_{k}\right)\right):\nabla y_{\kappa}\left(x_{k}\right)
\]
and $\widetilde{y}$ is a modified state given by
\begin{equation}
\widetilde{y}\left(y,\widetilde{P}\right)=\left(\rho v_{1},\ldots,\rho v_{d},\widetilde{\rho u}\left(C_{1},\ldots,C_{n_{s}},\widetilde{P}\right)+\frac{1}{2}\sum_{k=1}^{d}\rho v_{k}v_{k},C_{1},\ldots,C_{n_{s}}\right)^{T}.\label{eq:interpolated-state-modified}
\end{equation}
$\widetilde{P}$ in Equation~(\ref{eq:interpolated-state-modified})
is a polynomial in $\mathcal{P}_{p}(\kappa)$ that approximates the
pressure as

\[
\widetilde{P}_{\kappa}=\sum_{j=1}^{n_{b}}P\left(y_{\kappa}\left(x_{j}\right)\right)\phi_{j},
\]
from which the modified internal energy, $\widetilde{\rho u}$, is
calculated. Furthermore, in Equation~(\ref{eq:semi-discrete-form}),
$\delta^{v}\left(y,\nabla y,n\right)$ and $G\left(y^{+}\right):\left(\average y-y^{+}\right)\otimes n$
are replaced with $\delta^{v}\left(\widetilde{y},\nabla y,n\right)$
and $G\left(\widetilde{y}^{+}\right):\left(\average{\widetilde{y}}-\widetilde{y}^{+}\right)\otimes n$,
respectively.

Since the linear-scaling limiter used to enforce the positivity property
and entropy boundedness does not completely eliminate small-scale
nonphysical artifacts~\citep{Zha10,Jia18,Lv15_2,Wu21_2}, especially
near non-smooth features, we add the artificial dissipation term~\citep{Har13}
\begin{equation}
-\sum_{\kappa\in\mathcal{T}}\left(\nu_{\mathrm{AV}}\nabla y,\nabla\mathfrak{v}\right)_{\kappa}\label{eq:artificial-viscosity-integral}
\end{equation}
to the LHS of Equation~(\ref{eq:semi-discrete-form}), where $\nu_{\mathrm{AV}}$
is the artificial viscosity, calculated as~\citep{Joh20_2}
\[
\nu_{\mathrm{AV}}=\left(C_{\mathrm{AV}}+S_{\mathrm{AV}}\right)\left(\frac{h^{2}}{p+1}\left|\frac{\partial T}{\partial y}\cdot\frac{\mathcal{R}\left(y,\nabla y\right)}{T}\right|\right).
\]
$S_{\mathrm{AV}}$ is a shock sensor based on solution variations
inside a given element~\citep{Chi19}, $C_{\mathrm{AV}}$ is a user-defined
parameter, $h$ is a length scale associated with the element, and
$\mathcal{R}\left(y,\nabla y\right)$ is the strong form of the residual~(\ref{eq:conservation-law-strong-form}).
In our previous work, this artificial viscosity formulation effectively
mitigated small-scale nonlinear instabilities in various multicomponent-flow
problems~\citep{Joh20_2,Chi22}. However, other types of artificial
viscosity or limiters can be employed instead. Additional details
on the basic DG discretization and the issue of pressure equilibrium
preservation can be found in~\citep{Joh20_2}.

\section{Transport step: Positivity-preserving, entropy-bounded discontinuous
Galerkin method\label{sec:positivity-preserving-entropy-bounded-DG}}

Let $\mathcal{G}_{\sigma}$ denote the following set:
\begin{equation}
\mathcal{G}_{\sigma}=\left\{ y\mid\rho>0,\rho u^{*}>0,C_{1}\geq0,\ldots,C_{n_{s}}\geq0,\chi_{\sigma}\geq0\right\} ,
\end{equation}
where $\sigma\in\mathbb{R}$, $\chi_{\sigma}=\rho s-\rho\sigma$,
and $u^{*}$ is the ``shifted'' internal energy~\citep{Hua19},
calculated as
\begin{equation}
u^{*}=u-u_{0},\quad u_{0}=\left.u\right|_{T=0},\label{eq:shifted-internal-energy}
\end{equation}
such that $u^{*}>0$ if and only if $T>0$, provided $c_{v,i}>0,\:i=1,\ldots,n_{s}$~\citep{Gio99}.
Note that $\rho>0$ and $u^{*}>0$ imply $P(y)>0$. The $\chi_{\sigma}>0$
inequality is associated with entropy boundedness, which will be discussed
in more detail later in this section. Let $\mathcal{G}$ denote a
similar set, but without the entropy constraint, i.e.,

\begin{equation}
\mathcal{G}=\left\{ y\mid\rho>0,\rho u^{*}>0,C_{1}\geq0,\ldots,C_{n_{s}}\geq0\right\} .
\end{equation}
Since $\rho u^{*}(y)$ is a concave function of the state~\citep{Chi22},
$\mathcal{G}$ is a convex set. If all species concentrations are
strictly positive, then for a given $\sigma$, $\chi_{\sigma}$ is
concave~\citep{Jia18,Chi22,Gou20} and $\mathcal{G}_{\sigma}$ is
also a convex set. However, if any of the species concentrations is
zero, then $\chi_{\sigma}$ is no longer concave~\citep{Gou20,Gou20_2}.
For the remainder of this paper, in any discussion of entropy, $\mathcal{G}_{\sigma}$
is always assumed to be a convex set. Note that this assumption does
not seem to have any discernible negative effects on the solver~\citep{Chi22,Chi22_2}.
In addition, as will be made clear in the upcoming subsection and
Remark~\ref{rem:zero-species-concentrations}, positive species concentrations
are assumed until Section~\ref{subsec:zero-species-concentrations},
wherein this restriction is relaxed to allow for consideration of
zero concentrations.

\subsection{Positivity-preserving Lax-Friedrichs-type viscous flux function\label{subsec:Positivity-preserving-Lax-Friedrichs}}

In this subsection, we extend the local Lax-Friedrichs-type viscous
flux function by Zhang~\citep{Zha17} to multicomponent flows with
species diffusion. In particular, we consider the viscous flux separately
from the inviscid flux, unlike Du and Yang~\citep{Du22}, who adapted
said flux function to multicomponent flows in a manner that treats
both fluxes simultaneously. Unless otherwise specified, we assume
this flux function is employed for the remainder of the section. The
penalty term takes the form~\citep{Zha17}
\[
\delta^{v}(y^{+},y^{-},\nabla y^{+},\nabla y^{-},n)=\frac{\beta}{2}\left(y^{+}-y^{-}\right),
\]
where $\beta>0$ is the dissipation coefficient. The lemma below introduces
a constraint on the definition of $\beta$ that is essential for satisfaction
of the positivity property by the DG formulation, as will be discussed
in Section~\ref{subsec:1D-PPDG-high-order}. In Section~\ref{subsec:thermal-bubble},
we demonstrate that this viscous flux function achieves optimal convergence
for smooth flows. For compatibility with boundary conditions and the
aforementioned techniques to reduce spurious pressure oscillations,
the definition of $\beta$ is first presented in terms of the following
expansion of the viscous flux:
\[
\mathcal{F}^{v}=\left(\mathcal{F}_{\rho v}^{v},\mathcal{F}_{\rho e_{t}}^{v},\mathcal{F}_{C_{1}}^{v},\ldots,\mathcal{F}_{C_{n_{s}}}^{v}\right)^{T},
\]
where $\mathcal{F}_{\rho v}^{v}$ is the viscous momentum flux, $\mathcal{F}_{\rho e_{t}}^{v}$
is the viscous total-energy flux, and $\mathcal{F}_{C_{i}}^{v}$ is
the viscous molar flux of the $i$th species. Furthermore, until Section~\ref{subsec:zero-species-concentrations},
we assume that all species concentrations are strictly positive, unless
otherwise specified.
\begin{lem}
\label{lem:beta-constraints}Assume that $y=\left(\rho v,\rho e_{t},C\right)^{T}$
is in $\mathcal{G}$ and that $C_{i}>0,\:\forall i$. Then $y\pm\beta^{-1}\mathcal{F}^{v}\cdot n$,
where $n$ is a given unit vector, is also in $\mathcal{G}$ under
the following conditions:

\begin{equation}
\beta>\beta^{*}\left(y,\mathcal{F}^{v},n\right)=\left.\max\left\{ \max_{i=1,\ldots,n_{s}}\frac{\left|\mathcal{F}_{C_{i}}^{v}\cdot n\right|}{C_{i}},\beta_{T}\right\} \right|_{\left(y,\mathcal{F}^{v},n\right)},\label{eq:beta-constraint}
\end{equation}
where
\begin{equation}
\beta_{T}=\frac{\left|b\right|+\sqrt{b^{2}+2\rho^{2}u^{*}\left|\mathcal{F}_{\rho v}^{v}\cdot n\right|^{2}}}{2\rho^{2}u^{*}},\label{eq:beta_T}
\end{equation}
with $b=\rho\mathcal{F}_{\rho e_{t}}^{v}\cdot n-\rho v\cdot\mathcal{F}_{\rho v}^{v}\cdot n$.
\end{lem}

\begin{proof}
$y\pm\beta^{-1}\mathcal{F}^{v}\cdot n$ can be expanded as
\[
\begin{split}y\pm\beta^{-1}\mathcal{F}^{v}\cdot n= & \left(\rho v\pm\beta^{-1}\mathcal{F}_{\rho v}^{v}\cdot n,\rho e_{t}\pm\beta^{-1}\mathcal{F}_{\rho e_{t}}^{v}\cdot n,C_{1}\pm\beta^{-1}\mathcal{F}_{C_{1}}^{v}\cdot n,\ldots,C_{n_{s}}\pm\beta^{-1}\mathcal{F}_{C_{n_{s}}}^{v}\cdot n\right)^{T}.\end{split}
\]
First, we focus on positivity of density and species concentrations.
For the $i$th species, $C_{i}\pm\beta^{-1}\mathcal{F}_{C_{i}}^{v}\cdot n>0$
if and only if $\beta>\left|\mathcal{F}_{C_{i}}^{v}\cdot n\right|/C_{i}$.
Accounting for all species yields
\begin{equation}
\beta>\max_{i=1,\ldots,n_{s}}\frac{\left|\mathcal{F}_{C_{i}}^{v}\cdot n\right|}{C_{i}}.\label{eq:beta-constraint-concentration}
\end{equation}
Density is then also positive.

Next, we focus on positivity of temperature. For a given $y=\left(\rho v,\rho e_{t},C\right)^{T}$,
let $Z(y)$ be defined as
\begin{equation}
Z(y)=\rho^{2}u^{*}(y)=\rho(y)\rho e_{t}-\left|\rho v\right|^{2}/2-\rho^{2}u_{0}(y).\label{eq:Z-definition}
\end{equation}
Note that $Z(y)>0$ implies $T(y)>0$. $Z\left(y\pm\beta^{-1}\mathcal{F}^{v}\cdot n\right)$
can be expressed as
\begin{align*}
Z\left(y\pm\beta^{-1}\mathcal{F}^{v}\cdot n\right)= & \sum_{i=1}^{n_{s}}W_{i}\left(C_{i}\pm\beta^{-1}\mathcal{F}_{C_{i}}^{v}\cdot n\right)\left(\rho e_{t}\pm\beta^{-1}\mathcal{F}_{\rho e_{t}}^{v}\cdot n\right)\\
 & -\frac{1}{2}\left|\rho v\pm\beta^{-1}\mathcal{F}_{\rho v}^{v}\cdot n\right|^{2}-\left[\sum_{i=1}^{n_{s}}W_{i}\left(C_{i}\pm\beta^{-1}\mathcal{F}_{C_{i}}^{v}\cdot n\right)\right]^{2}u_{0},
\end{align*}
which, after multiplying both sides by $\beta^{2}$ and some algebraic
manipulation, can be rewritten as
\begin{align}
\beta^{2}Z\left(y\pm\beta^{-1}\mathcal{F}^{v}\cdot n\right)= & \rho^{2}u^{*}\beta^{2}\pm b\beta+g,\label{eq:beta-quadratic-form}
\end{align}
where $b=\rho e_{t}M+\rho\mathcal{F}_{\rho e_{t}}^{v}\cdot n-\rho v\cdot\mathcal{F}_{\rho v}^{v}\cdot n-2\rho u_{0}M$,
$g=M\mathcal{F}_{\rho e_{t}}^{v}\cdot n-\frac{1}{2}\left|\mathcal{F}_{\rho v}^{v}\cdot n\right|^{2}-u_{0}M^{2}$,
and $M=\sum_{i=1}^{n_{s}}W_{i}\mathcal{F}_{C_{i}}^{v}\cdot n$. By
mass conservation, $M=0$, such that $b=\rho\mathcal{F}_{\rho e_{t}}^{v}\cdot n-\rho v\cdot\mathcal{F}_{\rho v}^{v}\cdot n$
and $g=-\frac{1}{2}\left|\mathcal{F}_{\rho v}^{v}\cdot n\right|^{2}$.
Setting the RHS of Equation~(\ref{eq:beta-quadratic-form}) equal
to zero yields two quadratic equations with $\beta$ as the unknowns.
Since $\rho^{2}u^{*}$ is positive, the two quadratic equations are
convex. Furthermore, since $b^{2}+2\rho^{2}u^{*}\left|\mathcal{F}_{\rho v}^{v}\cdot n\right|^{2}\geq0$,
for each of the two quadratic equations, the roots are real and at
least one is nonnegative. A sufficient condition to ensure $Z\left(y\pm\beta^{-1}\mathcal{F}^{v}\cdot n\right)>0$
is $\beta>\beta_{T}\geq0$, where $\beta_{T}$, given by Equation~(\ref{eq:beta_T}),
is the largest of all roots of the quadratic equations. Combining
this with the inequality~(\ref{eq:beta-constraint-concentration})
yields~(\ref{eq:beta-constraint}).
\end{proof}
\begin{rem}
\label{rem:zero-species-concentrations}Lemma~\ref{lem:beta-constraints}
and the inequality~(\ref{eq:beta-constraint-concentration}) assume
that the species concentrations are positive. If $C_{i}=0$ and $\mathcal{F}_{C_{i}}^{v}\cdot n\neq0$,
then there exists no finite value of $\beta$ such that $C_{i}\pm\beta^{-1}\mathcal{F}_{C_{i}}^{v}\cdot n\geq0$
since $\mathcal{F}_{C_{i}}^{v}$ is not directly proportional to $C_{i}$.
Specifically, the $k$th spatial component of $\mathcal{F}_{C_{i}}^{v}(y,\nabla y)$
can be written as
\begin{align}
\mathcal{F}_{C_{i},k}^{v}(y,\nabla y) & =C_{i}V_{ik}\nonumber \\
 & =C_{i}\hat{V}_{ik}-\frac{C_{i}\sum_{l=1}^{n_{s}}W_{l}C_{l}\hat{V}_{lk}}{\rho}\nonumber \\
 & =\bar{D}_{i}\frac{\partial C_{i}}{\partial x_{k}}-\frac{C_{i}\bar{D}_{i}}{\rho}\frac{\partial\rho}{\partial x_{k}}-\frac{C_{i}}{\rho}\sum_{l=1}^{n_{s}}W_{l}\left(\bar{D}_{l}\frac{\partial C_{l}}{\partial x_{k}}-\frac{C_{l}\bar{D}_{l}}{\rho}\frac{\partial\rho}{\partial x_{k}}\right),\label{eq:molar_flux_definition}
\end{align}
such that $\mathcal{F}_{C_{i},k}^{v}$ can be nonzero even if $C_{i}=0$.
As previously mentioned, however, it is crucial to account for zero
concentrations. In Section~\ref{subsec:zero-species-concentrations},
we relax this restriction and discuss how to ensure nonnegative species
concentrations, which is done in a different manner from how positive
density and temperature are guaranteed.
\end{rem}

\begin{rem}
\label{rem:beta-not-abstract-form}The constraint on $\beta$ in~(\ref{eq:beta-constraint})
is left in abstract form, i.e., in terms of $\mathcal{F}^{v}=\left(\mathcal{F}_{\rho v}^{v},\mathcal{F}_{\rho e_{t}}^{v},\mathcal{F}_{C_{1}}^{v},\ldots,\mathcal{F}_{C_{n_{s}}}^{v}\right)^{T}$.
This is to allow for consideration of, for example, $\widetilde{y}\pm\mathcal{F}^{v}\left(\widetilde{y},\nabla y\right)\cdot n$,
where $y\neq\widetilde{y}$, which is necessary for the modified flux
interpolation~(\ref{eq:modified-flux-projection}) and for boundary
conditions. If we take $\mathcal{F}^{v}=\mathcal{F}^{v}\left(y,\nabla y\right)$
and substitute the definitions of each component of $\mathcal{F}^{v}$,
the constraint on $\beta$ reduces to

\begin{equation}
\beta>\max\left\{ \max_{i=1,\ldots,n_{s}}\left|V_{i}\cdot n\right|,\beta_{T}\right\} ,\label{eq:beta-constraint-reduced}
\end{equation}
where Equation~(\ref{eq:beta_T}) is now given by

\[
\beta_{T}=\frac{\left|b\right|+\sqrt{b^{2}+2\rho^{2}u^{*}\left|\tau\cdot n\right|^{2}}}{2\rho^{2}u^{*}},
\]
with
\[
b=\rho q\cdot n+\rho\sum_{i=1}^{n_{s}}W_{i}C_{i}h_{i}V_{i}\cdot n.
\]
Related to Remark~\ref{rem:zero-species-concentrations}, $V_{i}$
in~(\ref{eq:beta-constraint-reduced}) can blow up as $C_{i}\rightarrow0$
(unless $\nabla C_{i}=0$). If species diffusion is neglected, these
expressions recover those in~\citep{Zha17} for the monocomponent
case.
\end{rem}

\begin{rem}
$\sum_{i}W_{i}\left[C_{i}\pm\beta^{-1}\mathcal{F}_{C_{i}}^{v}(C,\nabla C)\cdot n\right]$
recovers $\rho$ since
\begin{align*}
\sum_{i=1}^{n_{s}}W_{i}\left[C_{i}\pm\beta^{-1}\mathcal{F}_{C_{i}}^{v}(C,\nabla C)\cdot n\right] & =\sum_{i=1}^{n_{s}}W_{i}C_{i}\pm\beta^{-1}\sum_{i=1}^{n_{s}}W_{i}C_{i}V_{i}\cdot n\\
 & =\sum_{i=1}^{n_{s}}W_{i}C_{i}\\
 & =\rho,
\end{align*}
where the second line is due to mass conservation, i.e., $\sum_{i=1}^{n_{s}}W_{i}C_{i}V_{ik}=0,\;k=1,\ldots,d$.
\end{rem}

\begin{rem}
Combining the convective and diffusive fluxes into a single flux,
as done by Zhang~\citep{Zha17} and Du and Yang~\citep{Du22}, results
in a different constraint on $\beta$. As discussed in Section~\ref{sec:Introduction},
in this work, we elect to use the HLLC inviscid flux function since
in our experience, it typically produces more accurate solutions than
the Lax-Friedrichs inviscid flux function. As such, the inviscid and
viscous fluxes are treated separately in our formulation.
\end{rem}

\subsection{One-dimensional case~\label{subsec:1D-PPDG}}

In this subsection, we consider the one-dimensional case. We first
focus on $p=0$ before proceeding to $p\geq1$. Without loss of generality,
we assume a uniform grid with element size $h$.

\subsubsection{First-order DG scheme in one dimension\label{subsec:1D-PPDG-first-order}}

Consider the following $p=0$, element-local DG discretization with
forward Euler time stepping:
\begin{equation}
\begin{aligned}y_{\kappa}^{j+1}= & y_{\kappa}^{j}-\frac{\Delta t}{h}\left[\mathcal{F}^{c\dagger}\left(y_{\kappa}^{j},y_{\kappa_{L}}^{j},-1\right)+\mathcal{F}^{c\dagger}\left(y_{\kappa}^{j},y_{\kappa_{R}}^{j},1\right)\right]\\
 & +\frac{\Delta t}{h}\left[-\mathcal{F}^{v}\left(y_{\kappa_{L}}^{j},\nabla y_{\kappa_{L}}^{j}\right)+\mathcal{F}^{v}\left(y_{\kappa_{R}}^{j},\nabla y_{\kappa_{R}}^{j}\right)-\delta^{v}\left(y_{\kappa}^{j},y_{\kappa_{L}}^{j},\nabla y_{\kappa}^{j},\nabla y_{\kappa_{L}}^{j},-1\right)-\delta^{v}\left(y_{\kappa}^{j},y_{\kappa_{R}}^{j},\nabla y_{\kappa}^{j},\nabla y_{\kappa_{R}}^{j},1\right)\right],
\end{aligned}
\label{eq:p0-DG-1D}
\end{equation}
where $\Delta t$ is the time step size, $j$ is the time step index,
and $\kappa_{L}$ and $\kappa_{R}$ are the elements to the left and
right of $\kappa$, respectively. Equation~(\ref{eq:p0-DG-1D}) can
be rearranged to split the convective and diffusive contributions
as~\citep{Zha17}
\begin{align}
y_{\kappa}^{j+1} & =\frac{1}{2}\left(y_{\kappa,c}^{j+1}+y_{\kappa,v}^{j+1}\right),\\
y_{\kappa,c}^{j+1} & =y_{\kappa}^{j}-\frac{\Delta t^{*}}{h}\left[\mathcal{F}^{c\dagger}\left(y_{\kappa}^{j},y_{\kappa_{L}}^{j},-1\right)+\mathcal{F}^{c\dagger}\left(y_{\kappa}^{j},y_{\kappa_{R}}^{j},1\right)\right],\\
y_{\kappa,v}^{j+1} & =y_{\kappa}^{j}+\frac{\Delta t^{*}}{h}\left[-\frac{1}{2}\mathcal{F}^{v}\left(y_{\kappa_{L}}^{j},\nabla y_{\kappa_{L}}^{j}\right)+\frac{1}{2}\mathcal{F}^{v}\left(y_{\kappa_{R}}^{j},\nabla y_{\kappa_{R}}^{j}\right)-\delta^{v}\left(y_{\kappa}^{j},y_{\kappa_{L}}^{j},\nabla y_{\kappa}^{j},\nabla y_{\kappa_{L}}^{j},-1\right)-\delta^{v}\left(y_{\kappa}^{j},y_{\kappa_{R}}^{j},\nabla y_{\kappa}^{j},\nabla y_{\kappa_{R}}^{j},1\right)\right],\label{eq:p0-1D-viscous-contribution}
\end{align}
where $\Delta t^{*}=2\Delta t$.

First taking into account the convective contribution, let $\lambda$
be an upper bound on the maximum wave speed of the system. $y_{\kappa}^{j},y_{\kappa_{L}}^{j},y_{\kappa_{R}}^{j}\in\mathcal{G}_{\sigma}$
implies $y_{\kappa,c}^{j+1}\in\mathcal{G}_{\sigma}$ if an \emph{invariant-region-preserving}
flux function~\citep{Jia18} is employed and the time step size satisfies
\begin{equation}
\frac{\Delta t^{*}\lambda}{h}\leq\frac{1}{2}.\label{eq:p0-time-step-constraint-inviscid}
\end{equation}
The Godunov, Lax-Friedrichs, HLL, and HLLC inviscid flux functions
are invariant-region-preserving~\citep{Jia18}. Since the focus of
this paper is the diffusive contribution, we refer the reader to~\citep{Chi22}
and the references therein for additional information on the convective
contribution.

For $p=0$, $\mathcal{F}^{v}\left(y_{\kappa},\nabla y_{\kappa}\right)=G(y_{\kappa}):\nabla y_{\kappa}=0$
since $\nabla y_{\kappa}=0$. As such, Equation~(\ref{eq:p0-1D-viscous-contribution})
reduces to
\begin{align}
y_{\kappa,v}^{j+1} & =y_{\kappa}^{j}+\frac{\Delta t^{*}}{h}\left[-\delta^{v}\left(y_{\kappa}^{j},y_{\kappa_{L}}^{j},\nabla y_{\kappa}^{j},\nabla y_{\kappa_{L}}^{j},-1\right)-\delta^{v}\left(y_{\kappa}^{j},y_{\kappa_{R}}^{j},\nabla y_{\kappa}^{j},\nabla y_{\kappa_{R}}^{j},1\right)\right]\nonumber \\
 & =\left[1-\frac{\Delta t^{*}}{2h}\left(\beta_{\kappa_{L}}+\beta_{\kappa_{R}}\right)\right]y_{\kappa}^{j}+\frac{\Delta t^{*}}{2h}\beta_{\kappa_{L}}y_{\kappa_{L}}^{j}+\frac{\Delta t^{*}}{2h}\beta_{\kappa_{R}}y_{\kappa_{R}}^{j}.\label{eq:p0-1D-viscous-contribution-convex-combination}
\end{align}
Under the time-step-size constraint
\[
\frac{\Delta t^{*}}{2h}\max_{\kappa}\left(\beta_{\kappa_{L}}+\beta_{\kappa_{R}}\right)\leq1,
\]
the RHS of Equation~(\ref{eq:p0-1D-viscous-contribution-convex-combination})
is a convex combination of $y_{\kappa}^{j}$, $y_{\kappa_{L}}^{j}$,
and $y_{\kappa_{R}}^{j}$ for any $\kappa$. $y_{\kappa}^{j},y_{\kappa_{L}}^{j},y_{\kappa_{R}}^{j}\in\mathcal{G}$
then implies $y_{\kappa,v}^{j+1}\in\mathcal{G}$. This holds even
for zero species concentrations. Finally, since $y_{\kappa}^{j+1}$
is a convex combination of $y_{\kappa,c}^{j+1}$ and $y_{\kappa,v}^{j+1}$,
$y_{\kappa}^{j},y_{\kappa_{L}}^{j},y_{\kappa_{R}}^{j}\in\mathcal{G}$
implies $y_{\kappa}^{j+1}\in\mathcal{G}$. Note that in principle,
this holds for any positive values of $\beta_{\kappa_{L}}$ and $\beta_{\kappa_{R}}$.

\subsubsection{High-order DG scheme in one dimension\label{subsec:1D-PPDG-high-order}}

Consider a quadrature rule with $n_{q}$ points and positive weights
denoted with $x_{q}$ and $w_{q}$, respectively, such that $x_{q}\in\kappa=\left[x_{L},x_{R}\right]$
, $\sum_{q=1}^{n_{q}}w_{q}=1$, and $n_{q}\geq n_{b}$, The endpoints,
$x_{L}$ and $x_{R}$, need not be included in the set of quadrature
points, and none of the volumetric integrals in Equation~(\ref{eq:semi-discrete-form})
need to be evaluated with this quadrature rule. The standard flux
interpolation~(\ref{eq:flux-projection}) is assumed here; the modified
flux interpolation~(\ref{eq:modified-flux-projection}) will be accounted
for in Section~\ref{subsec:modified-flux-interpolation-1d}. As in~\citep{Lv15_2}
the element-local solution average can be expanded as
\begin{align}
\overline{y}_{\kappa} & =\sum_{q=1}^{n_{q}}w_{q}y_{\kappa}(x_{q})\nonumber \\
 & =\sum_{q=1}^{n_{q}}\theta_{q}y_{\kappa}(x_{q})+\theta_{L}y_{\kappa}(x_{L})+\theta_{R}y_{\kappa}(x_{R}),\label{eq:element-average-convex-combination}
\end{align}
where, if the set of quadrature points includes the endpoints,
\[
\theta_{q}=\begin{cases}
w_{q} & x_{q}\neq x_{L},x_{q}\neq x_{R}\\
0 & \mathrm{otherwise}
\end{cases}
\]
and
\[
\theta_{L}=w_{L},\quad\theta_{R}=w_{R},
\]
with $w_{L}$ and $w_{R}$ denoting the quadrature weights at the
left and right endpoints, respectively. Otherwise, we take
\[
\theta_{q}=w_{q}-\theta_{L}\psi_{q}\left(x_{L}\right)-\theta_{R}\psi_{q}\left(x_{R}\right),
\]
where $\psi_{1},\ldots,\psi_{n_{d}}$ form a set of Lagrange basis
functions whose nodes are located at $n_{d}$ points of the set $\left\{ x_{q},q=1,\ldots,n_{q}\right\} $,
with $n_{b}\leq n_{d}\leq n_{q}$, and $\psi_{n_{d}+1},\ldots,\psi_{n_{q}}$
are equal to zero. As a result, $\sum_{q=1}^{n_{q}}\theta_{q}y_{\kappa}(x_{q})$
can be written as
\begin{align*}
\sum_{q=1}^{n_{q}}\theta_{q}y_{\kappa}(x_{q}) & =\sum_{q=1}^{n_{q}}\left[w_{q}-\theta_{L}\psi_{q}\left(x_{L}\right)-\theta_{R}\psi_{q}\left(x_{R}\right)\right]y_{\kappa}(x_{q})\\
 & =\sum_{q=1}^{n_{q}}w_{q}y_{\kappa}(x_{q})-\theta_{L}\sum_{q=1}^{n_{q}}y_{\kappa}(x_{q})\psi_{q}\left(x_{L}\right)-\theta_{R}\sum_{q=1}^{n_{q}}y_{\kappa}(x_{q})\psi_{q}\left(x_{R}\right)\\
 & =\sum_{q=1}^{n_{q}}w_{q}y_{\kappa}(x_{q})-\theta_{L}y_{\kappa}(x_{L})+\theta_{R}y_{\kappa}(x_{R}).
\end{align*}
$\theta_{L}$ and $\theta_{R}$ will be related to a time-step-size
constraint below (see~\citep{Lv15_2} and~\citep{Chi22} for additional
details). Note that $\sum_{q}\theta_{q}+\theta_{L}+\theta_{R}=1$
since $\sum_{q=1}^{n_{q}}\psi_{q}=1$. Due to the positivity of the
quadrature weights, there exist positive $\theta_{L}$ and $\theta_{R}$
that yield $\theta_{q}\geq0,\;q=1,\ldots,n_{q}$~\citep{Lv15_2}.
Define $\partial\mathcal{D}_{\kappa}=\left\{ x_{L,}x_{R}\right\} $,
and let $\mathcal{D}_{\kappa}$ denote the following set of points:
\[
\mathcal{D_{\kappa}}=\partial\mathcal{D}_{\kappa}\bigcup\left\{ x_{q},q=1,\ldots,n_{q}\right\} =\left\{ x_{L,}x_{R},x_{q},q=1,\ldots,n_{q}\right\} .
\]

Employing the forward Euler time-integration scheme and taking $\mathfrak{v}\in V_{h}^{0}$
yields the fully discrete scheme satisfied by the element averages,
\[
\overline{y}_{\kappa}^{j+1}=\frac{1}{2}\left(\overline{y}_{\kappa,c}^{j+1}+\overline{y}_{\kappa,v}^{j+1}\right),
\]
where
\begin{equation}
\begin{aligned}\overline{y}_{\kappa,c}^{j+1}= & \overline{y}_{\kappa}^{j}-\frac{\Delta t^{*}}{h}\left[\mathcal{F}^{c\dagger}\left(y_{\kappa}^{j}(x_{L}),y_{\kappa_{L}}^{j}(x_{L}),-1\right)+\mathcal{F}^{c\dagger}\left(y_{\kappa}^{j}(x_{R}),y_{\kappa_{R}}^{j}(x_{R}),1\right)\right]\\
= & \sum_{q=1}^{n_{q}}\theta_{q}y_{\kappa}^{j}(x_{q})+\theta_{L}y_{\kappa}^{j}(x_{L})-\frac{\Delta t^{*}}{h}\left[\mathcal{F}^{c\dagger}\left(y_{\kappa}^{j}(x_{L}),y_{\kappa_{L}}^{j}(x_{L}),-1\right)+\mathcal{F}^{\dagger}\left(y_{\kappa}^{j}(x_{L}),y_{\kappa}^{j}(x_{R}),1\right)\right]\\
 & +\theta_{R}y_{\kappa}^{j}(x_{R})-\frac{\Delta t^{*}}{h}\left[\mathcal{F}^{c\dagger}\left(y_{\kappa}^{j}(x_{R}),y_{\kappa}^{j}(x_{L}),-1\right)+\mathcal{F}^{\dagger}\left(y_{\kappa}^{j}(x_{R}),y_{\kappa_{R}}^{j}(x_{R}),1\right)\right],
\end{aligned}
\label{eq:fully-discrete-form-average-convective-1d}
\end{equation}
and
\begin{equation}
\begin{aligned}\overline{y}_{\kappa,v}^{j+1}= & \overline{y}_{\kappa}^{j}+\frac{\Delta t^{*}}{h}\left[-\frac{1}{2}\mathcal{F}^{v}\left(y_{\kappa}^{j}(x_{L}),\nabla y_{\kappa}^{j}(x_{L})\right)-\frac{1}{2}\mathcal{F}^{v}\left(y_{\kappa_{L}}^{j}(x_{L}),\nabla y_{\kappa_{L}}^{j}(x_{L})\right)\right.\\
 & +\frac{1}{2}\mathcal{F}^{v}\left(y_{\kappa}^{j}(x_{R}),\nabla y_{\kappa}^{j}(x_{R})\right)+\frac{1}{2}\mathcal{F}^{v}\left(y_{\kappa_{R}}^{j}(x_{R}),\nabla y_{\kappa_{R}}^{j}(x_{R})\right)\\
 & \left.-\frac{\beta_{\kappa_{L}}}{2}y_{\kappa}^{j}(x_{L})+\frac{\beta_{\kappa_{L}}}{2}y_{\kappa_{L}}^{j}(x_{L})-\frac{\beta_{\kappa_{R}}}{2}y_{\kappa}^{j}(x_{R})+\frac{\beta_{\kappa_{R}}}{2}y_{\kappa_{R}}^{j}(x_{R})\right]\\
= & \sum_{q=1}^{n_{q}}\theta_{q}y_{\kappa}^{j}(x_{q})+\frac{\Delta t^{*}}{2h}\beta_{\kappa_{L}}\left[y_{\kappa_{L}}^{j}(x_{L})-\beta_{\kappa_{L}}^{-1}\mathcal{F}^{v}\left(y_{\kappa_{L}}^{j}(x_{L}),\nabla y_{\kappa_{L}}^{j}(x_{L})\right)\right]\\
 & +\frac{\Delta t^{*}}{2h}\beta_{\kappa_{R}}\left[y_{\kappa_{R}}^{j}(x_{R})+\beta_{\kappa_{R}}^{-1}\mathcal{F}^{v}\left(y_{\kappa_{R}}^{j}(x_{R}),\nabla y_{\kappa_{R}}^{j}(x_{R})\right)\right]\\
 & +\left(\theta_{L}-\frac{\Delta t^{*}}{2h}\beta_{\kappa_{L}}\right)\left[y_{\kappa}^{j}(x_{L})-\frac{\Delta t^{*}}{2h\left(\theta_{L}-\frac{\Delta t^{*}}{2h}\beta_{\kappa_{L}}\right)}\mathcal{F}^{v}\left(y_{\kappa}^{j}(x_{L}),\nabla y_{\kappa}^{j}(x_{L})\right)\right]\\
 & +\left(\theta_{R}-\frac{\Delta t^{*}}{2h}\beta_{\kappa_{R}}\right)\left[y_{\kappa}^{j}(x_{R})+\frac{\Delta t^{*}}{2h\left(\theta_{R}-\frac{\Delta t^{*}}{2h}\beta_{\kappa_{R}}\right)}\mathcal{F}^{v}\left(y_{\kappa}^{j}(x_{R}),\nabla y_{\kappa}^{j}(x_{R})\right)\right].
\end{aligned}
\label{eq:fully-discrete-form-average-viscous-1d}
\end{equation}
The second equality in Equation~(\ref{eq:fully-discrete-form-average-convective-1d})
is due to the conservation property of the numerical flux:
\[
\mathcal{F}^{\dagger}\left(y_{\kappa}^{j}(x_{L}),y_{\kappa}^{j}(x_{R}),1\right)=-\mathcal{F}^{\dagger}\left(y_{\kappa}^{j}(x_{R}),y_{\kappa}^{j}(x_{L}),-1\right).
\]
Note that Equations~(\ref{eq:fully-discrete-form-average-convective-1d})
and~(\ref{eq:fully-discrete-form-average-viscous-1d}) hold regardless
of whether the integrals in Equation~(\ref{eq:semi-discrete-form})
are computed with conventional quadrature or a quadrature-free approach~\citep{Atk96,Atk98}.

The limiting strategy, which is described in Section~\ref{subsec:limiting-procedure},
requires that $\overline{y}_{\kappa,c}^{j+1}$ and $\overline{y}_{\kappa,v}^{j+1}(x)$
be in $\mathcal{G}_{s_{b}}$ and $\mathcal{G}$, respectively, where
$s_{b}$ is a lower bound on the specific thermodynamic entropy. As
discussed in~\citep{Chi22}, we employ a local entropy bound,
\begin{equation}
s_{b,\kappa}^{j+1}(y)=\min\left\{ s\left(y^{j}(x)\right)\vert x\in\mathcal{D}_{\kappa}\cup\mathcal{D}_{\kappa_{L}}\cup\mathcal{D}_{\kappa_{R}}\right\} ,\label{eq:local-entropy-bound}
\end{equation}
which is motivated by the minimum entropy principle satisfied by entropy
solutions to the multicomponent Euler equations~\citep{Gou20}. It
can be shown that if $y_{\kappa}^{j}(x)\in\mathcal{G}_{s_{b}},\;\forall x\in\mathcal{D}_{\kappa}$
, and $y_{\kappa}^{-,j}\in\mathcal{G}_{s_{b}},\;\forall x\in\partial\mathcal{D}_{\kappa}$,
where $y_{\kappa}^{-}$ denotes the exterior state along $\partial\kappa$,
then $\overline{y}_{\kappa,c}^{j+1}$ is in $\mathcal{G}_{s_{b}}$
under the time-step-size constraint
\begin{equation}
\frac{\Delta t^{*}\lambda}{h}\leq\frac{1}{2}\min\left\{ \theta_{L},\theta_{R}\right\} \label{eq:CFL-condition-convective-1d}
\end{equation}
and the conditions
\begin{equation}
\theta_{L}>0,\theta_{R}>0,\theta_{q}\geq0,q=1,\ldots,n_{q}.\label{eq:theta-conditions-1d}
\end{equation}
More information can be found in~\citep{Chi22}. The conditions under
which $\overline{y}_{\kappa,v}^{j+1}\in\mathcal{G}$ are analyzed
in the following theorem.

\begin{thm}
\label{thm:CFL-condition-1D-viscous}If $y_{\kappa}^{j}(x)\in\mathcal{G},\;\forall x\in\mathcal{D_{\kappa}}$,
and $y_{\kappa}^{-,j}\in\mathcal{G},\;\forall x\in\partial\mathcal{D}_{\kappa}$,
then $\overline{y}_{\kappa,v}^{j+1}$ is also in $\mathcal{G}$ under
the time-step-size constraint
\begin{equation}
\frac{\Delta t^{*}}{h}\leq\min\left\{ \frac{\theta_{L}}{\beta_{\kappa_{L}}},\frac{\theta_{R}}{\beta_{\kappa_{R}}}\right\} ,\label{eq:CFL-condition-1D-viscous}
\end{equation}
the constraints on $\beta$,
\begin{align}
\beta_{\kappa_{L}} & >\max\left\{ \beta^{*}\left(y_{\kappa}^{j}(x_{L}),\mathcal{F}^{v}\left(y_{\kappa}^{j}(x_{L}),\nabla y_{\kappa}^{j}(x_{L})\right),-1\right),\beta^{*}\left(y_{\kappa_{L}}^{j}(x_{L}),\mathcal{F}^{v}\left(y_{\kappa_{L}}^{j}(x_{L}),\nabla y_{\kappa_{L}}^{j}(x_{L})\right),-1\right)\right\} ,\label{eq:beta-constraint-1d-L}\\
\beta_{\kappa_{R}} & >\max\left\{ \beta^{*}\left(y_{\kappa}^{j}(x_{R}),\mathcal{F}^{v}\left(y_{\kappa}^{j}(x_{R}),\nabla y_{\kappa}^{j}(x_{R})\right),1\right),\beta^{*}\left(y_{\kappa_{R}}^{j}(x_{R}),\mathcal{F}^{v}\left(y_{\kappa_{R}}^{j}(x_{R}),\nabla y_{\kappa_{R}}^{j}(x_{R})\right),1\right)\right\} ,\label{eq:beta-constraint-1d-R}
\end{align}
and the conditions~(\ref{eq:theta-conditions-1d}).
\end{thm}

\begin{proof}
The inequality~(\ref{eq:beta-constraint-1d-L}) guarantees that $y_{\kappa_{L}}^{j}(x_{L})-\beta_{\kappa_{L}}^{-1}\mathcal{F}^{v}\left(y_{\kappa_{L}}^{j}(x_{L}),\nabla y_{\kappa_{L}}^{j}(x_{L})\right)\in\mathcal{G}$.
According to the time-step-size constraint~(\ref{eq:CFL-condition-1D-viscous}),
we have
\[
\frac{\Delta t^{*}}{h}\leq\frac{\theta_{L}}{\beta_{\kappa_{L}}},
\]
such that
\[
\theta_{L}-\frac{\Delta t^{*}}{2h}\beta_{\kappa_{L}}\geq\frac{\Delta t^{*}}{2h}\beta_{\kappa_{L}}.
\]
It follows that
\begin{align*}
\frac{\Delta t^{*}}{2h\left(\theta_{L}-\frac{\Delta t^{*}}{2h}\beta_{\kappa_{L}}\right)} & \leq\frac{\Delta t^{*}}{2h\left(\frac{\Delta t^{*}}{2h}\beta_{\kappa_{L}}\right)}\\
 & =\beta_{\kappa_{L}}^{-1},
\end{align*}
which means $y_{\kappa}^{j}(x_{L})-\frac{\Delta t^{*}}{2h\left(\theta_{L}-\frac{\Delta t^{*}}{2h}\beta_{\kappa_{L}}\right)}\mathcal{F}^{v}\left(y_{\kappa}^{j}(x_{L}),\nabla y_{\kappa}^{j}(x_{L})\right)\in\mathcal{G}$.
Moreover, we have $\frac{\Delta t^{*}}{2h}\beta_{\kappa_{L}}\leq\theta_{L}\leq1$.
The same arguments can be applied to show
\begin{align*}
 & y_{\kappa_{R}}^{j}(x_{R})+\beta_{\kappa_{R}}^{-1}\mathcal{F}^{v}\left(y_{\kappa_{R}}^{j}(x_{R}),\nabla y_{\kappa_{R}}^{j}(x_{R})\right)\in\mathcal{G},\\
 & y_{\kappa}^{j}(x_{R})+\frac{\Delta t^{*}}{2h\left(\theta_{R}-\frac{\Delta t^{*}}{2h}\beta_{\kappa_{R}}\right)}\mathcal{F}^{v}\left(y_{\kappa}^{j}(x_{R}),\nabla y_{\kappa}^{j}(x_{R})\right)\in\mathcal{G},\\
 & \frac{\Delta t^{*}}{2h}\beta_{\kappa_{R}}\leq\theta_{R}\leq1.
\end{align*}
Therefore, $\overline{y}_{\kappa,v}^{j+1}$ is a convex combination
of states in $\mathcal{G}$, such that $\overline{y}_{\kappa,v}^{j+1}\in\mathcal{G}$.
\end{proof}
\begin{rem}
Though forward Euler time stepping is employed for demonstration purposes,
any time integration scheme that can be expressed as a convex combination
of forward Euler steps, such as strong-stability-preserving Runge-Kutta
(SSPRK) methods, can be used.
\end{rem}

As previously mentioned, the final ingredient of the positivity-preserving,
entropy-bounded DG scheme is a limiting strategy (described in Section~\ref{subsec:limiting-procedure})
to ensure $y_{\kappa,c}^{j+1}(x)\in\mathcal{G}_{s_{b}}$ and $y_{\kappa,v}^{j+1}(x)\in\mathcal{G}$,
for all $x\in\mathcal{D}_{\kappa}$, where $y_{c}^{j+1}$ satisfies

\begin{gather*}
\sum_{\kappa\in\mathcal{T}}\left(\frac{y_{c}^{j+1}-y^{j}}{\Delta t^{*}},\mathfrak{v}\right)_{\kappa}-\sum_{\kappa\in\mathcal{T}}\left(\mathcal{F}^{c}\left(y^{j},\nabla y^{j}\right),\nabla\mathfrak{v}\right)_{\kappa}+\sum_{\epsilon\in\mathcal{E}}\left(\mathcal{F}^{c\dagger}\left(y^{j},n\right),\left\llbracket \mathfrak{v}\right\rrbracket \right)_{\epsilon}=0\qquad\forall\mathfrak{v}\in V_{h}^{p},
\end{gather*}
and $y_{v}^{j+1}$ satisfies

\begin{gather*}
\sum_{\kappa\in\mathcal{T}}\left(\frac{y_{v}^{j+1}-y^{j}}{\Delta t^{*}},\mathfrak{v}\right)_{\kappa}-\sum_{\epsilon\in\mathcal{E}}\left(\average{\mathcal{F}^{v}\left(y^{j},\nabla y^{j}\right)}\cdot n-\delta^{v}\left(y^{j},\nabla y^{j},n\right),\left\llbracket \mathfrak{v}\right\rrbracket \right)_{\epsilon}\\
+\sum_{\kappa\in\mathcal{T}}\left(G\left(y^{j,+}\right):\left(\average{y^{j}}-y^{j,+}\right)\otimes n,\nabla\mathfrak{v}\right)_{\partial\kappa}=0\qquad\forall\mathfrak{v}\in V_{h}^{p},
\end{gather*}
such that $y^{j+1}=\frac{1}{2}\left(y_{c}^{j+1}+y_{v}^{j+1}\right)$.
$y_{\kappa}^{j+1}(x)$ is then in $\mathcal{G}$, for all $x\in\mathcal{D}_{\kappa}$,
since $y_{\kappa}^{j+1}$ is a convex combination of $y_{\kappa,c}^{j+1}$
and $y_{\kappa,v}^{j+1}$. Entropy boundedness is enforced on only
the convective contribution since the viscous flux function is not
fully compatible with an entropy constraint and the minimum entropy
principle does not hold for the Navier-Stokes equations unless the
thermal diffusivity is zero~\citep{Tad86,Gue14}. The limiting strategy
here relies on a simple linear-scaling limiter that is conservative,
maintains stability, and in general preserves order of accuracy for
smooth solutions~\citep{Zha10,Zha17,Zha12_2,Lv15_2,Jia18}. However,
it is not expected to suppress all small-scale instabilities, which
is why artificial viscosity is employed in tandem.

\subsubsection{Modified flux interpolation\label{subsec:modified-flux-interpolation-1d}}

In this subsection, we discuss how to account for the modified flux
interpolation~(\ref{eq:modified-flux-projection}). In~\citep{Chi22_2},
we already discussed the inviscid case; therefore, we only consider
$y_{\kappa,v}$ here. The scheme satisfied by the element averages
becomes
\begin{equation}
\begin{aligned}\overline{y}_{\kappa,v}^{j+1}= & \overline{y}_{\kappa}^{j}+\frac{\Delta t^{*}}{h}\left[-\frac{1}{2}\mathcal{F}^{v}\left(\widetilde{y}_{\kappa}^{j}(x_{L}),\nabla y_{\kappa}^{j}(x_{L})\right)-\frac{1}{2}\mathcal{F}^{v}\left(\widetilde{y}_{\kappa_{L}}^{j}(x_{L}),\nabla y_{\kappa_{L}}^{j}(x_{L})\right)\right.\\
 & +\frac{1}{2}\mathcal{F}^{v}\left(\widetilde{y}_{\kappa}^{j}(x_{R}),\nabla y_{\kappa}^{j}(x_{R})\right)+\frac{1}{2}\mathcal{F}^{v}\left(\widetilde{y}_{\kappa_{R}}^{j}(x_{R}),\nabla y_{\kappa_{R}}^{j}(x_{R})\right)\\
 & \left.-\frac{\beta_{\kappa_{L}}}{2}\widetilde{y}_{\kappa}^{j}(x_{L})+\frac{\beta_{\kappa_{L}}}{2}\widetilde{y}_{\kappa_{L}}^{j}(x_{L})-\frac{\beta_{\kappa_{R}}}{2}\widetilde{y}_{\kappa}^{j}(x_{R})+\frac{\beta_{\kappa_{R}}}{2}\widetilde{y}_{\kappa_{R}}^{j}(x_{R})\right].
\end{aligned}
\label{eq:fully-discrete-form-average-viscous-1d-modified}
\end{equation}
If the nodal set includes the endpoints (e.g., equidistant or Gauss-Lobatto
points), then $y_{\kappa}^{j}(x_{L})=\widetilde{y}_{\kappa}^{j}(x_{L})$
and $y_{\kappa}^{j}(x_{R})=\widetilde{y}_{\kappa}^{j}(x_{R})$, in
which case both Equation~(\ref{eq:fully-discrete-form-average-viscous-1d})
and Theorem~\ref{thm:CFL-condition-1D-viscous} hold and the modified
flux interpolation does not require any additional modifications to
the formulation.

\subsection{Limiting procedure\label{subsec:limiting-procedure}}

Here, we describe the positivity-preserving and entropy limiters to
ensure $y_{\kappa,c}^{j+1}(x)\in\mathcal{G}_{s_{b}}$ and $y_{\kappa,v}^{j+1}(x)\in\mathcal{G}$,
respectively, for all $x\in\mathcal{D}_{\kappa}$. We assume that
$\overline{y}_{\kappa,c}^{j+1}(x)\in\mathcal{G}_{s_{b}}$ and $\overline{y}_{\kappa,v}^{j+1}(x)\in\mathcal{G}$.
The $j+1$ superscript and $\kappa$ subscript are dropped for brevity.
The limiting procedure is identical across one, two, and three dimensions.

\subsubsection*{Positivity-preserving limiter}

The positivity-preserving limiter enforces $\rho>0$, $C_{i}\geq0,\;\forall i$,
and $\rho u^{*}>0$ via the following steps:
\begin{itemize}
\item If $\rho(y\left(x\right))>\varepsilon,\:\forall x\in\mathcal{D}_{\kappa}$,
where $\varepsilon$ is a small positive number, such as $10^{-10}$,
then set $C_{i}^{(1)}=C_{i}=\sum_{j=1}^{n_{b}}C_{i}(x_{j})\phi_{j},i=1,\ldots,n_{s}$;
if not, set
\[
C_{i}^{(1)}=\overline{C}_{i}+\omega^{(1)}\left(C_{i}-\overline{C}_{i}\right),\quad\omega^{(1)}=\frac{\rho(\overline{y})-\epsilon}{\rho(\overline{y})-\underset{x\in\mathcal{D}}{\min}\rho(y(x))}.
\]
for $i=1,\ldots,n_{s}.$ Let $y^{(1)}=\left(\rho v_{1},\ldots,\rho v_{d},\rho e_{t},C_{1}^{(1)},\ldots,C_{n_{s}}^{(1)}\right)$.
This is referred to as the ``density limiter'' in Section~\ref{subsec:zero-species-concentrations}.
\item For $i=1,\ldots,n_{s}$, if $C_{i}^{(1)}(x)\geq0,\:\forall x\in\mathcal{D}_{\kappa}$,
then set $C_{i}^{(2)}=C_{i}^{(1)}$; if not, set
\[
C_{i}^{(2)}=\overline{C}_{i}+\omega^{(2)}\left(C_{i}^{(1)}-\overline{C}_{i}\right),\quad\omega^{(2)}=\frac{\overline{C}_{i}}{\overline{C}_{i}-\underset{x\in\mathcal{D}}{\min}C_{i}^{(1)}(x)}.
\]
Let $y^{(2)}=\left(\rho v_{1},\ldots,\rho v_{d},\rho e_{t},C_{1}^{(2)},\ldots,C_{n_{s}}^{(2)}\right)$.
\item If $\rho u^{*}\left(y^{(2)}(x)\right)>\epsilon,\:\forall x\in\mathcal{D}_{\kappa}$,
then set $y^{(3)}=y^{(2)}$; if not, set
\[
y^{(3)}=\overline{y}+\omega^{(3)}\left(y^{(2)}-\overline{y}\right),\quad\omega^{(3)}=\frac{\rho u^{*}(\overline{y})-\epsilon}{\rho u^{*}(\overline{y})-\underset{x\in\mathcal{D}}{\min}\rho u^{*}(y^{(2)}(x))}.
\]
Since $\rho u^{*}(y)$ is a concave function of $y$\citep{Chi22},
$\rho u^{*}(y^{(3)}(x))>0,\:\forall x\in\mathcal{D}_{\kappa}$~\citep{Wan12,Zha17}.
\end{itemize}
The positivity-preserving limiter is applied to both $y_{c}$ and
$y_{v}$.

\subsubsection*{Entropy limiter}

The entropy limiter, which is applied only to $y_{c}$, enforces $\chi\geq0$
as follows: if $\chi\left(y^{(3)}(x)\right)\geq0,\:\forall x\in\mathcal{D}_{\kappa}$,
then set $y^{(4)}=y^{(3)}$; if not, set
\[
y^{(4)}=\overline{y}+\omega^{(4)}\left(y^{(3)}-\overline{y}\right),\quad\omega^{(4)}=\frac{\chi(\overline{y})}{\chi(\overline{y})-\underset{x\in\mathcal{D}}{\min}\chi(y^{(3)}(x))}.
\]
Since $\chi(y)$ is a concave function of $y$~\citep{Jia18,Wu21_2},
$s\left(y^{(4)}(x)\right)\geq s_{b},\:\forall x\in\mathcal{D}_{\kappa}$.

The solution is then replaced as
\[
y\leftarrow\frac{1}{2}\left(y_{c}^{(4)}+y_{v}^{(3)}\right).
\]
This limiting procedure is applied at the end of every RK stage. Note
that if $y$ is split in a different manner as~\citep{Dza22}
\[
y\leftarrow y_{\ddagger}^{(3)},
\]
where $y_{\ddagger}$ satisfies
\begin{gather*}
\sum_{\kappa\in\mathcal{T}}\left(\frac{y_{\ddagger}^{j+1}-y_{\mathsf{c}}^{(4),j+1}}{\Delta t},\mathfrak{v}\right)_{\kappa}-\sum_{\epsilon\in\mathcal{E}}\left(\average{\mathcal{F}^{v}\left(y^{j},\nabla y^{j}\right)}\cdot n-\delta^{v}\left(y^{j},\nabla y^{j},n\right),\left\llbracket \mathfrak{v}\right\rrbracket \right)_{\epsilon}\\
+\sum_{\kappa\in\mathcal{T}}\left(G\left(y^{j,+}\right):\left(\average{y^{j}}-y^{j,+}\right)\otimes n,\nabla\mathfrak{v}\right)_{\partial\kappa}=0\qquad\forall\mathfrak{v}\in V_{h}^{p},
\end{gather*}
and $y_{\mathsf{c}}$ satisfies
\begin{gather*}
\sum_{\kappa\in\mathcal{T}}\left(\frac{y_{\mathsf{c}}^{j+1}-y^{j}}{\Delta t},\mathfrak{v}\right)_{\kappa}-\sum_{\kappa\in\mathcal{T}}\left(\mathcal{F}^{c}\left(y^{j},\nabla y^{j}\right),\nabla\mathfrak{v}\right)_{\kappa}+\sum_{\epsilon\in\mathcal{E}}\left(\mathcal{F}^{c\dagger}\left(y^{j},n\right),\left\llbracket \mathfrak{v}\right\rrbracket \right)_{\epsilon}=0\qquad\forall\mathfrak{v}\in V_{h}^{p},
\end{gather*}
then $\overline{y}_{\kappa,\ddagger}^{j+1}$ may not be in $\mathcal{G}$
in the case that $y_{\kappa,\mathsf{c}}^{j+1}(x)\notin\mathcal{G}_{s_{b}},\forall x\in\mathcal{D}_{\kappa}$
(i.e., $y_{\kappa,\mathsf{c}}^{(4),j+1}\neq y_{\kappa,\mathsf{c}}^{j+1}$).

\subsection{Multidimensional case}

In this section, the one-dimensional positivity-preserving, entropy-bounded
DG method presented in the previous subsection is extended to two
and three dimensions. Before doing so, we first review the geometric
mapping, as well as volume and surface quadrature rules. For conciseness,
any key ideas already presented in Section~\ref{subsec:1D-PPDG}
are only briefly mentioned here.

\subsubsection{Preliminaries\label{subsec:preliminaries-EBDG-2D}}

\paragraph{Geometric mapping}

Let $\xi=\left(\xi_{1},\ldots,\xi_{d}\right)$ denote the reference
coordinates and $\widehat{\kappa}$ denote the reference element.
The mapping $x(\xi):\widehat{\kappa}\rightarrow\kappa$ is defined
as
\[
x(\xi)=\sum_{m=1}^{n_{g}}x_{\kappa,m}\Phi_{m}(\xi),
\]
where $\left\{ x_{\kappa,1},\ldots,x_{\kappa,n_{g}}\right\} $ is
the set of geometric nodes of $\kappa,$ $\left\{ \Phi_{1},\ldots,\Phi_{n_{g}}\right\} $
is the set of geometric basis functions, and $n_{g}$ is the number
of basis functions. Let $J_{\kappa}$ denote the geometric Jacobian
and $\left|J_{\kappa}\right|$ denote its determinant, which is allowed
to vary with $\xi$. $y_{\kappa}$ can be expressed as
\[
y_{\kappa}=\sum_{j=1}^{n_{b}}y_{\kappa}(x_{j})\phi(\xi),\quad x=x(\xi)\in\kappa,\;\forall\xi\in\widehat{\kappa}.
\]
Let $\kappa^{(f)}$ be the $f$th neighbor of $\kappa$ and $\partial\kappa^{(f)}$
be the $f$th face of $\kappa$, such that $\partial\kappa=\bigcup_{f=1}^{n_{f}}\partial\kappa^{(f)}$,
where $n_{f}$ is the number of faces. Note that $n_{f}$ can vary
across elements, but we slightly abuse notation for brevity. $\widehat{\epsilon}$
denotes the reference face. We define $x\left(\zeta^{(f)}\right):\widehat{\epsilon}\rightarrow\partial\kappa^{(f)}$,
with $\zeta^{(f)}=\left(\zeta_{1}^{(f)},\ldots,\zeta_{d-1}^{(f)}\right)$
denoting the reference coordinates, as
\[
x\left(\zeta^{(f)}\right)=\sum_{m=1}^{n_{g,f}^{\partial}}x_{\kappa,m}^{(f)}\Phi_{m}^{(f)}\left(\zeta^{(f)}\right),
\]
where $\left\{ x_{\kappa,1}^{(f)},\ldots x_{\kappa,n_{g,f}^{\partial}}^{(f)}\right\} $
is the set of geometric nodes of $\partial\kappa^{(f)}$, $\left\{ \Phi_{1}^{(f)},\ldots,\Phi_{n_{g,f}^{\partial}}^{(f)}\right\} $
is the set of basis functions, and $n_{g,f}^{\partial}$ is the number
of basis functions. $\xi\left(\zeta^{(f)}\right):\widehat{\epsilon}\rightarrow\widehat{\kappa}$
is the mapping from the reference face to the reference element. The
surface Jacobian is denoted $J_{\partial\kappa}^{(f)}$, which can
vary with $\zeta^{(f)}$.

\paragraph{Quadrature rules}

Consider a volume quadrature rule with $n_{q}$ points and positive
weights, denoted $\xi_{q}$ and $w_{q}$, $q=1,\ldots,n_{q}$, respectively,
with $n_{q}\geq n_{b}$. The weights are appropriately scaled such
that $\sum_{q=1}^{n_{q}}w_{q}=\left|\widehat{\kappa}\right|$, where
$\left|\widehat{\kappa}\right|$ is the volume of $\widehat{\kappa}$.
The quadrature rule can be used to evaluate the volume integral over
$\kappa$ of a generic function, $g(x)$, as
\[
\int_{\kappa}g(x)dx=\int_{\widehat{\kappa}}g(x(\xi))\left|J_{\kappa}(\xi)\right|d\xi\approx\sum_{q=1}^{n_{q}}g\left(x(\xi_{q})\right)\left|J_{\kappa}(\xi_{q})\right|w_{q},
\]
If $g(x)$ is a polynomial, then quadrature with sufficiently high
$n_{q}$ gives the exact value.

Similarly, consider a surface quadrature rule with $n_{q}^{\partial}$
points and positive weights, denoted $\zeta_{l}$ and $w_{l}^{\partial}$,$l=1,\ldots,n_{q}^{\partial}$.
The weights are scaled such that $\sum_{l=1}^{n_{q}^{\partial}}w_{l}^{\partial}=\left|\widehat{\epsilon}\right|$,
where $\left|\widehat{\epsilon}\right|$ is the surface area $\widehat{\epsilon}$.
The surface quadrature rule can be used to evaluate the surface integral
over $\partial\kappa^{(f)}$ of a generic function as
\[
\int_{\partial\kappa^{(f)}}g(x)ds=\int_{\widehat{\epsilon}}g\left(x\left(\zeta^{(f)}\right)\right)\left|J_{\partial\kappa}^{(k)}\left(\zeta^{(f)}\right)\right|d\zeta\approx\sum_{l=1}^{n_{q}^{\partial}}g\left(x\left(\zeta_{l}^{(f)}\right)\right)\left|J_{\partial\kappa}^{(f)}\left(\zeta_{l}^{(f)}\right)\right|w_{f,l}^{\partial}=\sum_{l=1}^{n_{q}^{\partial}}g\left(x\left(\zeta_{l}^{(f)}\right)\right)\nu_{f,l}^{\partial},
\]
where $\nu_{f,l}^{\partial}=\left|J_{\partial\kappa}^{(f)}\left(\zeta_{l}^{(f)}\right)\right|w_{f,l}^{\partial}$.
If $g(x)$ is a polynomial, then quadrature with sufficiently high
$n_{q}^{\partial}$ yields the exact value. The closed surface integral
over $\partial\kappa$ can be computed as
\[
\int_{\partial\kappa}g(x)ds=\sum_{f=1}^{n_{f}}\int_{\partial\kappa^{(f)}}g(x)ds=\sum_{f=1}^{n_{f}}\int_{\widehat{\epsilon}}g\left(x\left(\zeta^{(f)}\right)\right)\left|J_{\partial\kappa}^{(f)}\left(\zeta^{(f)}\right)\right|d\zeta\approx\sum_{f=1}^{n_{f}}\sum_{l=1}^{n_{q,f}^{\partial}}g\left(x\left(\zeta_{l}^{(f)}\right)\right)\nu_{f,l}^{\partial},
\]
where we allow a different quadrature rule to be used for each face.

\paragraph{Additional considerations}

In the following, assume that the surface integrals in Equation~(\ref{eq:semi-discrete-form})
are computed using $\left\{ \zeta_{1}^{(f)},\ldots,\zeta_{n_{q,f}^{\partial}}^{(f)}\right\} _{f=1}^{n_{f}}$
as integration points. Define $\partial\mathcal{D}_{\kappa}$ and
$\mathcal{D}_{\kappa}$ as
\[
\partial\mathcal{D}_{\kappa}=\bigcup_{f=1}^{n_{f}}\left\{ x\left(\zeta_{l}^{(f)}\right),l=1,\ldots,n_{q,f}^{\partial}\right\} ,
\]
and
\[
\mathcal{D_{\kappa}}=\partial\mathcal{D}_{\kappa}\bigcup\left\{ x(\xi_{q}),q=1,\ldots,n_{q}\right\} =\bigcup_{f=1}^{n_{f}}\left\{ x\left(\zeta_{l}^{(f)}\right),l=1,\ldots,n_{q,f}^{\partial}\right\} \bigcup\left\{ x(\xi_{q}),q=1,\ldots,n_{q}\right\} ,
\]
 respectively. The points in $\left\{ x(\xi_{q}),q=1,\ldots,n_{q}\right\} $
need not be used in the evaluation of any volume integrals in Equation~(\ref{eq:semi-discrete-form}).
Without loss of generality, we define $\nu_{f,l}^{\partial}$ as
\begin{equation}
\nu_{f,l}^{\partial}=\begin{cases}
\left|J_{\partial\kappa}^{(f)}(\zeta_{l})\right|w_{f,l}^{\partial}, & l=1,\ldots,n_{q,f}^{\partial}\\
0, & l=n_{q,f}^{\partial}+1,\ldots,N
\end{cases},\label{eq:Jacobian_weights_piecewise}
\end{equation}
where the faces are ordered such that $N=\max_{f}\left\{ n_{q,f}^{\partial}\right\} =n_{q,n_{f}}^{\partial}$.
As a result, we have
\[
\sum_{f=1}^{n_{f}}\sum_{l=1}^{N}\nu_{f,l}^{\partial}=\sum_{f=1}^{n_{f}}\sum_{l=1}^{n_{q,f}^{\partial}}\nu_{f,l}^{\partial}=\sum_{f=1}^{n_{f}}\left|\partial\kappa^{(f)}\right|=\left|\partial\kappa\right|,
\]
where $|\partial\kappa|$ is the surface area of $\kappa$ and $\left|\partial\kappa^{(f)}\right|$
is the surface area of the $f$th face.

Note that although a quadrature-free implementation~\citep{Atk96,Atk98}
is used in this work to compute the integrals in Equation~(\ref{eq:semi-discrete-form}),
recall from Section~\ref{subsec:1D-PPDG-high-order} that the analysis
is performed on the scheme satisfied by the element averages, which
is identical between quadrature-based and quadrature-free approaches.
Nevertheless, the scheme satisfied by the element averages is presented
in terms of a quadrature-based approach for consistency with previous
studies.

\subsubsection{First-order DG scheme in multiple dimensions\label{subsec:2D-PPDG-first-order}}

Consider the following $p=0$, element-local DG discretization with
forward Euler time stepping:
\begin{equation}
\begin{split}y_{\kappa}^{j+1}= & y_{\kappa}^{j+1}-\sum_{f=1}^{n_{f}}\sum_{l=1}^{n_{q,f}^{\partial}}\frac{\Delta t\nu_{f,l}^{\partial}}{|\kappa|}\mathcal{F}^{c\dagger}\left(y_{\kappa}^{j},y_{\kappa^{(f)}}^{j},n\left(\zeta_{l}^{(f)}\right)\right)\\
 & +\sum_{f=1}^{n_{f}}\sum_{l=1}^{n_{q,f}^{\partial}}\frac{\Delta t\nu_{f,l}^{\partial}}{|\kappa|}\left[\frac{1}{2}\mathcal{F}^{v}\left(y_{\kappa}^{j},\nabla y_{\kappa}^{j}\right)\cdot n\left(\zeta_{l}^{(f)}\right)+\frac{1}{2}\mathcal{F}^{v}\left(y_{\kappa^{(f)}}^{j},\nabla y_{\kappa^{(f)}}^{j}\right)\cdot n\left(\zeta_{l}^{(f)}\right)\right.\\
 & \left.-\delta^{v}\left(y_{\kappa}^{j},y_{\kappa^{(f)}}^{j},\nabla y_{\kappa}^{j},\nabla y_{\kappa^{(f)}}^{j},n\left(\zeta_{l}^{(f)}\right)\right)\right],
\end{split}
\label{eq:p0-PPDG-2D}
\end{equation}
which can be rearranged to split the convective and diffusive contributions
as
\begin{align}
y_{\kappa}^{j+1}= & \frac{1}{2}\left(y_{\kappa,c}^{j+1}+y_{\kappa,v}^{j+1}\right),\\
y_{\kappa,c}^{j+1}= & y_{\kappa}^{j}-\sum_{f=1}^{n_{f}}\sum_{l=1}^{n_{q,f}^{\partial}}\frac{\Delta t^{*}\nu_{f,l}^{\partial}}{|\kappa|}\mathcal{F}^{c\dagger}\left(y_{\kappa}^{j},y_{\kappa^{(f)}}^{j},n\left(\zeta_{l}^{(f)}\right)\right),\\
y_{\kappa,v}^{j+1}= & y_{\kappa}^{j}+\sum_{f=1}^{n_{f}}\sum_{l=1}^{n_{q,f}^{\partial}}\frac{\Delta t^{*}\nu_{f,l}^{\partial}}{|\kappa|}\left[\frac{1}{2}\mathcal{F}^{v}\left(y_{\kappa}^{j},\nabla y_{\kappa}^{j}\right)\cdot n\left(\zeta_{l}^{(f)}\right)+\frac{1}{2}\mathcal{F}^{v}\left(y_{\kappa^{(f)}}^{j},\nabla y_{\kappa^{(f)}}^{j}\right)\cdot n\left(\zeta_{l}^{(f)}\right)\right.\nonumber \\
 & \left.-\delta^{v}\left(y_{\kappa}^{j},y_{\kappa^{(f)}}^{j},\nabla y_{\kappa}^{j},\nabla y_{\kappa^{(f)}}^{j},n\left(\zeta_{l}^{(f)}\right)\right)\right],\label{eq:p0-2D-viscous-contribution}
\end{align}
where $\left|\kappa\right|$ is the volume of the element. Since $\mathcal{F}^{v}\left(y_{\kappa},\nabla y_{\kappa}\right)=0$,
Equation~(\ref{eq:p0-2D-viscous-contribution}) reduces to
\begin{align}
y_{\kappa,v}^{j+1}= & y_{\kappa}^{j}-\sum_{f=1}^{n_{f}}\sum_{l=1}^{n_{q,f}^{\partial}}\frac{\Delta t^{*}\nu_{f,l}^{\partial}}{|\kappa|}\delta^{v}\left(y_{\kappa}^{j},y_{\kappa^{(f)}}^{j},\nabla y_{\kappa}^{j},\nabla y_{\kappa^{(f)}}^{j},n\left(\zeta_{l}^{(f)}\right)\right)\nonumber \\
 & =y_{\kappa}^{j}-\sum_{f=1}^{n_{f}}\frac{\Delta t^{*}\left|\partial\kappa^{(f)}\right|}{|\kappa|}\frac{\beta_{f}}{2}\left[y_{\kappa}^{j}-y_{\kappa^{(f)}}^{j}\right]\nonumber \\
 & =\left[1-\sum_{f=1}^{n_{f}}\frac{\Delta t^{*}\left|\partial\kappa^{(f)}\right|}{2\left|\kappa\right|}\beta_{f}\right]y_{\kappa}^{j}+\sum_{f=1}^{n_{f}}\frac{\Delta t^{*}\left|\partial\kappa^{(f)}\right|}{2|\kappa|}\beta_{f}y_{\kappa^{(f)}}^{j}\label{eq:p0-2D-viscous-contribution-convex-combination}
\end{align}
Under the time-step-size constraint
\[
\sum_{f=1}^{n_{f}}\frac{\Delta t^{*}\left|\partial\kappa^{(f)}\right|}{2\left|\kappa\right|}\beta_{f}\leq1,
\]
the RHS of Equation~(\ref{eq:p0-2D-viscous-contribution-convex-combination})
is a convex combination of $y_{\kappa}^{j}$ and $y_{\kappa^{(f)}}^{j},f=1,\ldots,n_{f}$.
As such, $y_{\kappa}^{j}\in\mathcal{G}$ and $y_{\kappa^{(f)}}^{j}\in\mathcal{G},f=1,\ldots,n_{f}$
imply $y_{\kappa,v}^{j+1}\in\mathcal{G}$.

\subsubsection{High-order DG scheme in multiple dimensions\label{subsec:2D-PPDG-high-order}}

As in the one-dimensional case, the element-local solution average
can be expanded as~\citep{Lv15_2}
\begin{align}
\overline{y}_{\kappa} & =\sum_{q=1}^{n_{q}}\frac{\left|J_{\kappa}(\xi_{q})\right|w_{q}}{|\kappa|}y_{\kappa}\left(\xi_{q}\right),\nonumber \\
 & =\sum_{q=1}^{n_{q}}\theta_{q}y_{\kappa}\left(\xi_{q}\right)+\sum_{f=1}^{n_{f}}\sum_{l=1}^{n_{q,f}^{\partial}}\theta_{f,l}y_{\kappa}\left(\xi\left(\zeta_{l}^{(f)}\right)\right).\label{eq:element-average-convex-combination-2D}
\end{align}
where, if $\partial\mathcal{D}_{\kappa}\subseteq\left\{ x(\xi_{q}),q=1,\ldots,n_{q}\right\} $,
\[
\theta_{q}=\begin{cases}
\frac{\left|J_{\kappa}(\xi_{q})\right|w_{q}}{|\kappa|} & x\left(\xi_{q}\right)\notin\partial\mathcal{D}_{\kappa}\\
0 & x\left(\xi_{q}\right)\in\partial\mathcal{D}_{\kappa}
\end{cases}
\]
and
\[
\theta_{f,l}=\frac{\left|J_{\kappa}\left(\xi\left(\zeta_{l}^{(f)}\right)\right)\right|w_{f,l}}{|\kappa|N_{f,l}},
\]
with $w_{f,l}$ denoting the volume quadrature weight corresponding
to the quadrature point that satisfies $\xi_{q}=\xi\left(\zeta_{l}^{(f)}\right)$
and $N_{f,l}$ denoting the number of faces belonging to $\kappa$
shared by the given point. Otherwise, we take
\[
\theta_{q}=\frac{\left|J_{\kappa}(\xi_{q})\right|w_{q}}{|\kappa|}-\sum_{f=1}^{n_{f}}\sum_{l=1}^{n_{q,f}^{\partial}}\theta_{f,l}\psi_{q}\left(\xi\left(\zeta_{l}^{(f)}\right)\right),
\]
where $\psi_{1},\ldots,\psi_{n_{d}}$ form a set of Lagrange basis
functions whose nodes are located at $n_{d}$ points of the set $\left\{ x_{q},q=1,\ldots,n_{q}\right\} $,
with $n_{b}\leq n_{d}\leq n_{q}$, and $\psi_{n_{d}+1},\ldots,\psi_{n_{q}}$
are equal to zero. As a result, $\sum_{q=1}^{n_{q}}\theta_{q}y_{\kappa}\left(\xi_{q}\right)$
can be written as
\begin{align*}
\sum_{q=1}^{n_{q}}\theta_{q}y_{\kappa}\left(\xi_{q}\right) & =\sum_{q=1}^{n_{q}}\left[\frac{\left|J_{\kappa}(\xi_{q})\right|w_{q}}{|\kappa|}-\sum_{f=1}^{n_{f}}\sum_{l=1}^{n_{q,f}^{\partial}}\theta_{f,l}\psi_{q}\left(\xi\left(\zeta_{l}^{(f)}\right)\right)\right]y_{\kappa}\left(\xi_{q}\right)\\
 & =\sum_{q=1}^{n_{q}}\frac{\left|J_{\kappa}(\xi_{q})\right|w_{q}}{|\kappa|}y_{\kappa}\left(\xi_{q}\right)-\sum_{f=1}^{n_{f}}\sum_{l=1}^{n_{q,f}^{\partial}}\theta_{f,l}\sum_{q=1}^{n_{q}}y_{\kappa}\left(\xi_{q}\right)\psi_{q}\left(\xi\left(\zeta_{l}^{(f)}\right)\right)\\
 & =\sum_{q=1}^{n_{q}}\frac{\left|J_{\kappa}(\xi_{q})\right|w_{q}}{|\kappa|}y_{\kappa}\left(\xi_{q}\right)-\sum_{f=1}^{n_{f}}\sum_{l=1}^{n_{q,f}^{\partial}}\theta_{f,l}y_{\kappa}\left(\xi\left(\zeta_{l}^{(f)}\right)\right).
\end{align*}
$\theta_{f,l}$ will be related to a constraint on the time step size
(see~\citep{Lv15_2} and~\citep{Chi22_2} for additional details).
Since $w_{q}>0,q=1,\ldots n_{q}$, there exist positive values of
$\theta_{f,l}$ that yield $\theta_{q}\geq0$~\citep{Lv15_2}. Furthermore,
we have $\sum_{q=1}^{n_{q}}\theta_{q}+\sum_{f=1}^{n_{f}}\sum_{l=1}^{n_{q,f}^{\partial}}\theta_{f,l}=1.$

Employing the forward Euler time-integration scheme and taking $\mathfrak{v}\in V_{h}^{0}$
yields the fully discrete scheme satisfied by the element averages,
\[
\overline{y}_{\kappa}^{j+1}=\frac{1}{2}\left(\overline{y}_{\kappa,c}^{j+1}+\overline{y}_{\kappa,v}^{j+1}\right),
\]
where
\begin{align}
\overline{y}_{\kappa,c}^{j+1} & =\overline{y}_{\kappa}^{j}-\sum_{f=1}^{n_{f}}\sum_{l=1}^{n_{q,f}^{\partial}}\frac{\Delta t^{*}\nu_{f,l}^{\partial}}{|\kappa|}\mathcal{F}^{\dagger}\left(y_{\kappa}^{j}\left(\xi\left(\zeta_{l}^{(f)}\right)\right),y_{\kappa^{(f)}}^{j}\left(\xi\left(\zeta_{l}^{(f)}\right)\right),n\left(\zeta_{l}^{(f)}\right)\right)\nonumber \\
 & =\sum_{q=1}^{n_{q}}\theta_{q}y_{\kappa}^{j}\left(\xi_{q}\right)+\sum_{f=1}^{n_{f}}\sum_{l=1}^{n_{q,f}^{\partial}}\left[\theta_{f,l}y_{\kappa}^{j}\left(\xi\left(\zeta_{l}^{(f)}\right)\right)-\frac{\Delta t^{*}\nu_{f,l}^{\partial}}{|\kappa|}\mathcal{F}^{\dagger}\left(y_{\kappa}^{j}\left(\xi\left(\zeta_{l}^{(f)}\right)\right),y_{\kappa^{(f)}}^{j}\left(\xi\left(\zeta_{l}^{(f)}\right)\right),n\left(\zeta_{l}^{(f)}\right)\right)\right]\label{eq:fully-discrete-form-average-convective-2d}
\end{align}
and

\begin{eqnarray}
\overline{y}_{\kappa,v}^{j+1} & = & \overline{y}_{\kappa}^{j}+\sum_{f=1}^{n_{f}}\sum_{l=1}^{n_{q,f}^{\partial}}\frac{\Delta t^{*}\nu_{f,l}^{\partial}}{|\kappa|}\left[\frac{1}{2}\mathcal{F}^{v}\left(y_{\kappa}^{j}\left(\xi\left(\zeta_{l}^{(f)}\right)\right),\nabla y_{\kappa}^{j}\left(\xi\left(\zeta_{l}^{(f)}\right)\right)\right)\cdot n\left(\zeta_{l}^{(f)}\right)\right.\nonumber \\
 &  & +\frac{1}{2}\mathcal{F}^{v}\left(y_{\kappa^{(f)}}^{j}\left(\xi\left(\zeta_{l}^{(f)}\right)\right),\nabla y_{\kappa^{(f)}}^{j}\left(\xi\left(\zeta_{l}^{(f)}\right)\right)\right)\cdot n\left(\zeta_{l}^{(f)}\right)\nonumber \\
 &  & \left.-\delta^{v}\left(y_{\kappa}^{j}\left(\xi\left(\zeta_{l}^{(f)}\right)\right),y_{\kappa^{(f)}}^{j}\left(\xi\left(\zeta_{l}^{(f)}\right)\right),\nabla y_{\kappa}^{j}\left(\xi\left(\zeta_{l}^{(f)}\right)\right),\nabla y_{\kappa^{(f)}}^{j}\left(\xi\left(\zeta_{l}^{(f)}\right)\right),n\left(\zeta_{l}^{(f)}\right)\right)\right]\nonumber \\
 & = & \sum_{q=1}^{n_{q}}\theta_{q}y_{\kappa}^{j}\left(\xi_{q}\right)+\sum_{f=1}^{n_{f}}\sum_{l=1}^{n_{q,f}^{\partial}}\left[\frac{\Delta t^{*}\nu_{f,l}^{\partial}}{2|\kappa|}\beta_{f,l}y_{\kappa^{(f)}}^{j}\left(\xi\left(\zeta_{l}^{\left(f\right)}\right)\right)\right.\nonumber \\
 &  & +\frac{\Delta t^{*}\nu_{f,l}^{\partial}}{2|\kappa|}\mathcal{F}^{v}\left(y_{\kappa^{(f)}}^{j}\left(\xi\left(\zeta_{l}^{\left(f\right)}\right)\right),\nabla y_{\kappa^{(f)}}^{j}\left(\xi\left(\zeta_{l}^{\left(f\right)}\right)\right)\right)\cdot n\left(\zeta_{l}^{\left(f\right)}\right)+\theta_{f,l}y_{\kappa}^{j}\left(\xi\left(\zeta_{l}^{\left(f\right)}\right)\right)\nonumber \\
 &  & \left.-\frac{\Delta t^{*}\nu_{f,l}^{\partial}}{2|\kappa|}\beta_{f,l}y_{\kappa}^{j}\left(\xi\left(\zeta_{l}^{\left(f\right)}\right)\right)+\frac{\Delta t^{*}\nu_{f,l}^{\partial}}{2|\kappa|}\mathcal{F}^{v}\left(y_{\kappa}^{j}\left(\xi\left(\zeta_{l}^{\left(f\right)}\right)\right),\nabla y_{\kappa}^{j}\left(\xi\left(\zeta_{l}^{\left(f\right)}\right)\right)\right)\cdot n\left(\zeta_{l}^{\left(f\right)}\right)\right]\nonumber \\
 & = & \sum_{q=1}^{n_{q}}\theta_{q}y_{\kappa}^{j}\left(\xi_{q}\right)\nonumber \\
 &  & +\sum_{f=1}^{n_{f}}\sum_{l=1}^{n_{q,f}^{\partial}}\frac{\Delta t^{*}\nu_{f,l}^{\partial}}{2|\kappa|}\beta_{f,l}\left[y_{\kappa^{(f)}}^{j}\left(\xi\left(\zeta_{l}^{\left(f\right)}\right)\right)+\beta_{f,l}^{-1}\mathcal{F}^{v}\left(y_{\kappa^{(f)}}^{j}\left(\xi\left(\zeta_{l}^{\left(f\right)}\right)\right),\nabla y_{\kappa^{(f)}}^{j}\left(\xi\left(\zeta_{l}^{\left(f\right)}\right)\right)\right)\cdot n\left(\zeta_{l}^{\left(f\right)}\right)\right]\nonumber \\
 &  & +\sum_{f=1}^{n_{f}}\sum_{l=1}^{n_{q,f}^{\partial}}\Lambda_{f,l}\Biggl[y_{\kappa}^{j}\left(\xi\left(\zeta_{l}^{\left(f\right)}\right)\right)+\frac{\Delta t^{*}\nu_{f,l}^{\partial}}{2|\kappa|}\Lambda_{f,l}^{-1}\mathcal{F}^{v}\left(y_{\kappa}^{j}\left(\xi\left(\zeta_{l}^{\left(f\right)}\right)\right),\nabla y_{\kappa}^{j}\left(\xi\left(\zeta_{l}^{\left(f\right)}\right)\right)\right)\cdot n\left(\zeta_{l}^{\left(f\right)}\right)\Biggr],\label{eq:fully-discrete-form-average-viscous-2d}
\end{eqnarray}
with $\Lambda_{f,l}=\theta_{f,l}-\frac{\Delta t^{*}\nu_{f,l}^{\partial}}{2|\kappa|}\beta_{f,l}$.
Standard flux interpolation, as in Equation~(\ref{eq:flux-projection}),
is assumed here; the modified flux interpolation~(\ref{eq:modified-flux-projection})
will be accounted for in Section~\ref{subsec:modified-flux-interpolation}.
Note that Equations~(\ref{eq:fully-discrete-form-average-convective-2d})
and~(\ref{eq:fully-discrete-form-average-viscous-2d}) still hold
for the quadrature-free implementation~\citep{Atk96,Atk98} used
to evaluate the integrals in Equation~(\ref{eq:semi-discrete-form})
since the integrals of the basis functions over the reference element
(required in the quadrature-free implementation) can be considered
the weights of a generalized Newton-Cotes quadrature rule~\citep{Win08}.

It can be shown that if $y_{\kappa}^{j}(x)\in\mathcal{G}_{s_{b}},\;\forall x\in\mathcal{D}_{\kappa}$
, and $y_{\kappa}^{-,j}\in\mathcal{G}_{s_{b}},\;\forall x\in\partial\mathcal{D}_{\kappa}$,
then $\overline{y}_{\kappa,c}^{j+1}$ is in $\mathcal{G}_{s_{b}}$
under the time-step-size constraint~\citep{Chi22_2}
\begin{align}
\frac{\Delta t^{*}\lambda}{|\kappa|} & \leq\frac{1}{2}\min\left\{ L_{A},L_{B},L_{C}\right\} ,\label{eq:CFL-condition-convective-2D}\\
L_{A} & =\min\left\{ \left.\frac{\theta_{f,l}}{\nu_{f,l}^{\partial}}\right|f=1,\ldots,n_{f}-1,\;l=1,\ldots,n_{q,f}^{\partial}\right\} ,\nonumber \\
L_{B} & =\min\left\{ \left.\frac{\theta_{n_{f},l}}{\nu_{f,l}^{\partial}}\frac{\left|\partial\kappa^{(f)}\right|}{|\partial\kappa|}\right|,f=1,\ldots,n_{f},\;l=1,\ldots,\min\left\{ n_{q,f}^{\partial},N-1\right\} \right\} ,\nonumber \\
L_{C} & =\frac{\theta_{n_{f},N}}{|\partial\kappa|},\nonumber
\end{align}
and the conditions
\begin{equation}
\begin{cases}
\theta_{q}\geq0, & q=1,\ldots,n_{q}\\
\theta_{f,l}>0, & f=1,\ldots,n_{f},\;l=1,\ldots,n_{q,f}^{\partial}.
\end{cases}\label{eq:theta-conditions-2D}
\end{equation}
The entropy bound, $s_{b}$, is computed as
\begin{equation}
s_{b,\kappa}^{j+1}(y)=\min\left\{ s\left(y^{j}(x)\right)\left|x\in\bigcup_{f=1}^{n_{f}}\mathcal{D}_{\kappa^{(f)}}\bigcup\mathcal{D}_{\kappa}\right.\right\} .\label{eq:local-entropy-bound-2D}
\end{equation}
In the following theorem, we analyze the conditions under which $\overline{y}_{\kappa,v}^{j+1}\in\mathcal{G}$.
\begin{thm}
\label{thm:CFL-condition-2D-viscous}If $y_{\kappa}^{j}(x)\in\mathcal{G},\;\forall x\in\mathcal{D_{\kappa}}$,
and $y_{\kappa}^{-,j}\in\mathcal{G},\;\forall x\in\partial\mathcal{D}_{\kappa}$,
then $\overline{y}_{\kappa,v}^{j+1}$ is also in $\mathcal{G}$ under
the time-step-size constraint
\begin{align}
\frac{\Delta t^{*}}{\left|\kappa\right|} & \leq\min\left\{ \left.\frac{\theta_{f,l}}{\beta_{f,l}\nu_{f,l}^{\partial}}\right|f=1,\ldots,n_{f},\;l=1,\ldots,n_{q,f}^{\partial}\right\} ,\label{eq:CFL-condition-viscous-2D}
\end{align}
the constraints on $\beta$,
\begin{align}
\beta_{f,l} & >\max\left\{ \beta_{f,l}^{(1)},\beta_{f,l}^{(2)}\right\} ,\label{eq:beta-constraint-2d}\\
\beta_{f,l}^{(1)} & =\beta^{*}\left(y_{\kappa}^{j}\left(\xi\left(\zeta_{l}^{\left(f\right)}\right)\right),\mathcal{F}^{v}\left(y_{\kappa}^{j}\left(\xi\left(\zeta_{l}^{\left(f\right)}\right)\right),\nabla y_{\kappa}^{j}\left(\xi\left(\zeta_{l}^{\left(f\right)}\right)\right)\right),n\left(\zeta_{l}^{\left(f\right)}\right)\right),\nonumber \\
\beta_{f,l}^{(2)} & =\beta^{*}\left(y_{\kappa^{(f)}}^{j}\left(\xi\left(\zeta_{l}^{\left(f\right)}\right)\right),\mathcal{F}^{v}\left(y_{\kappa^{(f)}}^{j}\left(\xi\left(\zeta_{l}^{\left(f\right)}\right)\right),\nabla y_{\kappa^{(f)}}^{j}\left(\xi\left(\zeta_{l}^{\left(f\right)}\right)\right)\right),n\left(\zeta_{l}^{\left(f\right)}\right)\right),\nonumber
\end{align}
and the conditions~(\ref{eq:theta-conditions-2D}).
\end{thm}

\begin{proof}
The proof follows similar logic to that for Theorem~\ref{thm:CFL-condition-1D-viscous}.
The inequality~(\ref{eq:beta-constraint-2d}) guarantees that
\[
y_{\kappa^{(f)}}^{j}\left(\xi\left(\zeta_{l}^{\left(f\right)}\right)\right)+\beta_{f,l}^{-1}\mathcal{F}^{v}\left(y_{\kappa^{(f)}}^{j}\left(\xi\left(\zeta_{l}^{\left(f\right)}\right)\right),\nabla y_{\kappa^{(f)}}^{j}\left(\xi\left(\zeta_{l}^{\left(f\right)}\right)\right)\right)\cdot n\left(\zeta_{l}^{\left(f\right)}\right)\in\mathcal{G}.
\]
According to the time-step-size constraint~(\ref{eq:CFL-condition-viscous-2D}),
we have
\[
\frac{\Delta t^{*}}{\left|\kappa\right|}\leq\frac{\theta_{f,l}}{\beta_{f,l}\nu_{f,l}^{\partial}}
\]
such that
\[
\theta_{f,l}-\frac{\Delta t^{*}\nu_{f,l}^{\partial}}{2\left|\kappa\right|}\beta_{f,l}\geq\frac{\Delta t^{*}\nu_{f,l}^{\partial}}{2\left|\kappa\right|}\beta_{f,l},
\]
for $f=1,\ldots,n_{f},\;l=1,\ldots,n_{q,f}^{\partial}$. It follows
that
\begin{align*}
\frac{\Delta t^{*}\nu_{f,l}^{\partial}}{2|\kappa|}\Lambda_{f,l}^{-1} & =\frac{\Delta t^{*}\nu_{f,l}^{\partial}}{2|\kappa|\left(\theta_{f,l}-\frac{\Delta t^{*}\nu_{f,l}^{\partial}}{2|\kappa|}\beta_{f,l}\right)}\\
 & \leq\frac{\Delta t^{*}\nu_{f,l}^{\partial}}{2|\kappa|\left(\frac{\Delta t^{*}\nu_{f,l}^{\partial}}{2\left|\kappa\right|}\beta_{f,l}\right)}\\
 & =\beta_{f,l}^{-1},
\end{align*}
which means
\[
y_{\kappa}^{j}\left(\xi\left(\zeta_{l}^{\left(f\right)}\right)\right)+\frac{\Delta t^{*}\nu_{f,l}^{\partial}}{2|\kappa|}\Lambda_{f,l}^{-1}\mathcal{F}^{v}\left(y_{\kappa}^{j}\left(\xi\left(\zeta_{l}^{\left(f\right)}\right)\right),\nabla y_{\kappa}^{j}\left(\xi\left(\zeta_{l}^{\left(f\right)}\right)\right)\right)\cdot n\left(\zeta_{l}^{\left(f\right)}\right)\in\mathcal{G}.
\]
Moreover, we have $\frac{\Delta t^{*}\nu_{f,l}^{\partial}}{2|\kappa|}\beta_{f,l}\leq\theta_{f,l}\leq1$.
Therefore, $\overline{y}_{\kappa,v}^{j+1}$ is a convex combination
of states in $\mathcal{G}$, such that $\overline{y}_{\kappa,v}^{j+1}\in\mathcal{G}$.
\end{proof}
\begin{rem}
The same limiting strategy as in the one-dimensional case is employed
to ensure $y_{\kappa,c}^{j+1}(x)\in\mathcal{G}_{s_{b}}$ and $y_{\kappa,v}^{j+1}(x)\in\mathcal{G}$,
for all $x\in\mathcal{D}_{\kappa}$, such that $y_{\kappa}^{j+1}(x)\in\mathcal{G},\;\forall x\in\mathcal{D_{\kappa}}$.
\end{rem}

\begin{rem}
The multidimensional formulation is compatible with curved elements
of arbitrary shape, provided that appropriate quadrature rules exist.
Note that the consideration of non-constant geometric Jacobians is
significantly more straightforward for the Lax-Friedrichs-type viscous
flux function than for invariant-region-preserving inviscid flux functions
since the former \emph{algebraically }satisfies the positivity property
while the latter relies on the notion of a Riemann problem. It is
worth mentioning, however, that the Lax-Friedrichs inviscid flux function
also satisfies the positivity property algebraically~\citep{Zha10,Zha17}.
\end{rem}

\subsubsection{Modified flux interpolation\label{subsec:modified-flux-interpolation}}

With the modified flux interpolation~(\ref{eq:modified-flux-projection}),
the scheme satisfied by the element averages (for the viscous contribution)
becomes

\begin{eqnarray}
\overline{y}_{\kappa,v}^{j+1} & = & \overline{y}_{\kappa}^{j}+\sum_{f=1}^{n_{f}}\sum_{l=1}^{n_{q,f}^{\partial}}\frac{\Delta t^{*}\nu_{f,l}^{\partial}}{|\kappa|}\left[\frac{1}{2}\mathcal{F}^{v}\left(\widetilde{y}_{\kappa}^{j}\left(\xi\left(\zeta_{l}^{(f)}\right)\right),\nabla y_{\kappa}^{j}\left(\xi\left(\zeta_{l}^{(f)}\right)\right)\right)\cdot n\left(\zeta_{l}^{(f)}\right)\right.\nonumber \\
 &  & +\frac{1}{2}\mathcal{F}^{v}\left(\widetilde{y}_{\kappa^{(f)}}^{j}\left(\xi\left(\zeta_{l}^{(f)}\right)\right),\nabla y_{\kappa^{(f)}}^{j}\left(\xi\left(\zeta_{l}^{(f)}\right)\right)\right)\cdot n\left(\zeta_{l}^{(f)}\right)\nonumber \\
 &  & \left.-\delta^{v}\left(\widetilde{y}_{\kappa}^{j}\left(\xi\left(\zeta_{l}^{(f)}\right)\right),\widetilde{y}_{\kappa^{(f)}}^{j}\left(\xi\left(\zeta_{l}^{(f)}\right)\right),\nabla\widetilde{y}_{\kappa}^{j}\left(\xi\left(\zeta_{l}^{(f)}\right)\right),\nabla\widetilde{y}_{\kappa^{(f)}}^{j}\left(\xi\left(\zeta_{l}^{(f)}\right)\right),n\left(\zeta_{l}^{(f)}\right)\right)\right]\nonumber \\
 & = & \sum_{q=1}^{n_{q}}\theta_{q}y_{\kappa}^{j}\left(\xi_{q}\right)+\sum_{f=1}^{n_{f}}\sum_{l=1}^{n_{q,f}^{\partial}}\left[\frac{\Delta t^{*}\nu_{f,l}^{\partial}}{2|\kappa|}\beta_{f,l}\widetilde{y}_{\kappa^{(f)}}^{j}\left(\xi\left(\zeta_{l}^{\left(f\right)}\right)\right)\right.\nonumber \\
 &  & +\frac{\Delta t^{*}\nu_{f,l}^{\partial}}{2|\kappa|}\mathcal{F}^{v}\left(\widetilde{y}_{\kappa^{(f)}}^{j}\left(\xi\left(\zeta_{l}^{\left(f\right)}\right)\right),\nabla y_{\kappa^{(f)}}^{j}\left(\xi\left(\zeta_{l}^{\left(f\right)}\right)\right)\right)\cdot n\left(\zeta_{l}^{\left(f\right)}\right)+\theta_{f,l}y_{\kappa}^{j}\left(\xi\left(\zeta_{l}^{\left(f\right)}\right)\right)\nonumber \\
 &  & \left.-\frac{\Delta t^{*}\nu_{f,l}^{\partial}}{2|\kappa|}\beta_{f,l}\widetilde{y}_{\kappa}^{j}\left(\xi\left(\zeta_{l}^{\left(f\right)}\right)\right)+\frac{\Delta t^{*}\nu_{f,l}^{\partial}}{2|\kappa|}\mathcal{F}^{v}\left(\widetilde{y}_{\kappa}^{j}\left(\xi\left(\zeta_{l}^{\left(f\right)}\right)\right),\nabla y_{\kappa}^{j}\left(\xi\left(\zeta_{l}^{\left(f\right)}\right)\right)\right)\cdot n\left(\zeta_{l}^{\left(f\right)}\right)\right]\nonumber \\
 & = & \sum_{q=1}^{n_{q}}\theta_{q}y_{\kappa}^{j}\left(\xi_{q}\right)\nonumber \\
 &  & +\sum_{f=1}^{n_{f}}\sum_{l=1}^{n_{q,f}^{\partial}}\frac{\Delta t^{*}\nu_{f,l}^{\partial}}{2|\kappa|}\beta_{f,l}\left[\widetilde{y}_{\kappa^{(f)}}^{j}\left(\xi\left(\zeta_{l}^{\left(f\right)}\right)\right)+\beta_{f,l}^{-1}\mathcal{F}^{v}\left(\widetilde{y}_{\kappa^{(f)}}^{j}\left(\xi\left(\zeta_{l}^{\left(f\right)}\right)\right),\nabla y_{\kappa^{(f)}}^{j}\left(\xi\left(\zeta_{l}^{\left(f\right)}\right)\right)\right)\cdot n\left(\zeta_{l}^{\left(f\right)}\right)\right]\nonumber \\
 &  & +\sum_{f=1}^{n_{f}}\sum_{l=1}^{n_{q,f}^{\partial}}\Lambda_{f,l}\Biggl[\acute{y}_{\kappa}^{j}\left(\xi\left(\zeta_{l}^{\left(f\right)}\right)\right)-\frac{\Delta t^{*}\nu_{f,l}^{\partial}}{2|\kappa|}\beta_{f,l}\Lambda_{f,l}^{-1}\Delta\widetilde{y}_{\kappa}^{j}\left(\xi\left(\zeta_{l}^{\left(f\right)}\right)\right)\Biggr],\label{eq:fully-discrete-form-average-viscous-2d-modified}
\end{eqnarray}
where
\[
\acute{y}_{\kappa}^{j}\left(\xi\left(\zeta_{l}^{\left(f\right)}\right)\right)=y_{\kappa}^{j}\left(\xi\left(\zeta_{l}^{\left(f\right)}\right)\right)+\frac{\Delta t^{*}\nu_{f,l}^{\partial}}{2|\kappa|}\Lambda_{f,l}^{-1}\mathcal{F}^{v}\left(\widetilde{y}_{\kappa}^{j}\left(\xi\left(\zeta_{l}^{\left(f\right)}\right)\right),\nabla y_{\kappa}^{j}\left(\xi\left(\zeta_{l}^{\left(f\right)}\right)\right)\right)\cdot n\left(\zeta_{l}^{\left(f\right)}\right)
\]
and
\[
\Delta\widetilde{y}_{\kappa}^{j}=\widetilde{y}_{\kappa}^{j}-y_{\kappa}^{j}.
\]
Under the time-step-size constraint~(\ref{eq:CFL-condition-viscous-2D})
and the conditions~(\ref{eq:theta-conditions-2D}), we have
\begin{align*}
\widetilde{y}_{\kappa^{(f)}}^{j}\left(\xi\left(\zeta_{l}^{\left(f\right)}\right)\right)+\beta_{f,l}^{-1}\mathcal{F}\left(\widetilde{y}_{\kappa^{(f)}}^{j}\left(\xi\left(\zeta_{l}^{\left(f\right)}\right)\right),\nabla\widetilde{y}_{\kappa^{(f)}}^{j}\left(\xi\left(\zeta_{l}^{\left(f\right)}\right)\right)\right)\cdot n\left(\zeta_{l}^{\left(f\right)}\right) & \in\mathcal{G},\\
\acute{y}_{\kappa}^{j}\left(\xi\left(\zeta_{l}^{\left(f\right)}\right)\right) & \in\mathcal{G},
\end{align*}
provided that the constraints on $\beta$ are modified as
\begin{align}
\beta_{f,l} & >\max\left\{ \beta_{f,l}^{(1)},\beta_{f,l}^{(2)}\right\} ,\label{eq:beta-constraint-2d-modified}\\
\beta_{f,l}^{(1)} & =\beta^{*}\left(y_{\kappa}^{j}\left(\xi\left(\zeta_{l}^{\left(f\right)}\right)\right),\mathcal{F}^{v}\left(\widetilde{y}_{\kappa}^{j}\left(\xi\left(\zeta_{l}^{\left(f\right)}\right)\right),\nabla y_{\kappa}^{j}\left(\xi\left(\zeta_{l}^{\left(f\right)}\right)\right)\right),n\left(\zeta_{l}^{\left(f\right)}\right)\right),\nonumber \\
\beta_{f,l}^{(2)} & =\beta^{*}\left(\widetilde{y}_{\kappa^{(f)}}^{j}\left(\xi\left(\zeta_{l}^{\left(f\right)}\right)\right),\mathcal{F}^{v}\left(\widetilde{y}_{\kappa^{(f)}}^{j}\left(\xi\left(\zeta_{l}^{\left(f\right)}\right)\right),\nabla y_{\kappa^{(f)}}^{j}\left(\xi\left(\zeta_{l}^{\left(f\right)}\right)\right)\right),n\left(\zeta_{l}^{\left(f\right)}\right)\right).\nonumber
\end{align}
By Lemma~\ref{lem:alpha-constraints} in~\ref{sec:supporting-lemma},
$\acute{y}_{\kappa}^{j}\left(\xi\left(\zeta_{l}^{\left(f\right)}\right)\right)-\frac{\Delta t^{*}\nu_{f,l}^{\partial}}{2|\kappa|}\beta_{f,l}\Lambda_{f,l}^{-1}\Delta\widetilde{y}_{\kappa}^{j}\left(\xi\left(\zeta_{l}^{\left(f\right)}\right)\right)\in\mathcal{G}$
if
\[
\frac{\Delta t^{*}}{2\left|\kappa\right|}\beta_{f,l}<\frac{\theta_{f,l}}{\nu_{f,l}^{\partial}\left(1+\alpha^{*}\left(\acute{y}_{\kappa}^{j}\left(\xi\left(\zeta_{l}^{\left(f\right)}\right)\right),\Delta\widetilde{y}_{\kappa}^{j}\left(\xi\left(\zeta_{l}^{\left(f\right)}\right)\right)\right)\right)},
\]
where $\alpha^{*}$ is defined as in~(\ref{eq:alpha-constraint}).
Therefore, $\overline{y}_{\kappa,v}^{j+1}$ remains a convex combination
of states in $\mathcal{G}$ if, in addition to the time-step-size
constraint~(\ref{eq:CFL-condition-viscous-2D}), the following additional
condition is satisfied:
\[
\frac{\Delta t^{*}}{2\left|\kappa\right|}<\min_{f,l}\left\{ \frac{\theta_{f,l}}{\beta_{f,l}\nu_{f,l}^{\partial}\left(1+\alpha^{*}\left(\acute{y}_{\kappa}^{j}\left(\xi\left(\zeta_{l}^{\left(f\right)}\right)\right),\Delta\widetilde{y}_{\kappa}^{j}\left(\xi\left(\zeta_{l}^{\left(f\right)}\right)\right)\right)\right)}\right\}
\]
 We then have $\overline{y}_{\kappa,v}^{j+1}\in\mathcal{G}$.

\subsection{Boundary conditions}

Thus far, $\partial\kappa$ has been assumed to be in $\mathcal{E_{I}}$,
the set of interior interfaces. Here, we discuss how to enforce boundary
conditions, focusing on the viscous contribution in the multidimensional
case. For simplicity, but without loss of generality, we assume $\partial\kappa\in\mathcal{\mathcal{E}}_{\partial}$
(i.e., all faces of $\kappa$ are boundary faces). We also assume
$y_{\partial}^{j}\in\mathcal{G},\;\forall x\in\partial\mathcal{D}_{\kappa}$.
The boundary penalty term takes the form
\[
\delta_{\partial}^{v}\left(y^{+},y_{\partial},\nabla y^{+},n^{+}\right)=\frac{\beta}{2}\left(y^{+}-y_{\partial}\right).
\]
 The scheme satisfied by the element averages (for the viscous contribution)
becomes
\begin{eqnarray}
\overline{y}_{\kappa,v}^{j+1} & = & \overline{y}_{\kappa}^{j}+\sum_{f=1}^{n_{f}}\sum_{l=1}^{n_{q,f}^{\partial}}\frac{\Delta t^{*}\nu_{f,l}^{\partial}}{|\kappa|}\left[\frac{1}{2}\mathcal{F}_{\partial}^{v}\left(y_{\partial}^{j}\left(\xi\left(\zeta_{l}^{(f)}\right)\right),\nabla y_{\kappa}^{j}\left(\xi\left(\zeta_{l}^{(f)}\right)\right)\right)\cdot n\left(\zeta_{l}^{(f)}\right)\right.\nonumber \\
 &  & +\frac{1}{2}\mathcal{F}_{\partial}^{v}\left(y_{\partial}^{j}\left(\xi\left(\zeta_{l}^{(f)}\right)\right),\nabla y_{\kappa^{(f)}}^{j}\left(\xi\left(\zeta_{l}^{(f)}\right)\right)\right)\cdot n\left(\zeta_{l}^{(f)}\right)\nonumber \\
 &  & \left.-\delta_{\partial}^{v}\left(y_{\kappa}^{j}\left(\xi\left(\zeta_{l}^{(f)}\right)\right),y_{\partial}^{j}\left(\xi\left(\zeta_{l}^{(f)}\right)\right),\nabla y_{\kappa}^{j}\left(\xi\left(\zeta_{l}^{(f)}\right)\right),\nabla y_{\partial}^{j}\left(\xi\left(\zeta_{l}^{(f)}\right)\right),n\left(\zeta_{l}^{(f)}\right)\right)\right]\nonumber \\
 & = & \sum_{q=1}^{n_{q}}\theta_{q}y_{\kappa}^{j}\left(\xi_{q}\right)+\sum_{f=1}^{n_{f}}\sum_{l=1}^{n_{q,f}^{\partial}}\left[\frac{\Delta t^{*}\nu_{f,l}^{\partial}}{2|\kappa|}\beta_{f,l}y_{\partial}^{j}\left(\xi\left(\zeta_{l}^{(f)}\right)\right)\right.\nonumber \\
 &  & +\frac{\Delta t^{*}\nu_{f,l}^{\partial}}{2|\kappa|}\mathcal{F}_{\partial}^{v}\left(y_{\partial}^{j}\left(\xi\left(\zeta_{l}^{(f)}\right)\right),\nabla y_{\kappa}^{j}\left(\xi\left(\zeta_{l}^{(f)}\right)\right)\right)\cdot n\left(\zeta_{l}^{(f)}\right)+\theta_{f,l}y_{\kappa}^{j}\left(\xi\left(\zeta_{l}^{\left(f\right)}\right)\right)\nonumber \\
 &  & \left.-\frac{\Delta t^{*}\nu_{f,l}^{\partial}}{2|\kappa|}\beta_{f,l}y_{\kappa}^{j}\left(\xi\left(\zeta_{l}^{(f)}\right)\right)+\frac{\Delta t^{*}\nu_{f,l}^{\partial}}{2|\kappa|}\mathcal{F}_{\partial}^{v}\left(y_{\partial}^{j}\left(\xi\left(\zeta_{l}^{(f)}\right)\right),\nabla y_{\kappa^{(f)}}^{j}\left(\xi\left(\zeta_{l}^{(f)}\right)\right)\right)\cdot n\left(\zeta_{l}^{\left(f\right)}\right)\right]\nonumber \\
 & = & \sum_{q=1}^{n_{q}}\theta_{q}y_{\kappa}^{j}\left(\xi_{q}\right)\nonumber \\
 &  & +\sum_{f=1}^{n_{f}}\sum_{l=1}^{n_{q,f}^{\partial}}\frac{\Delta t^{*}\nu_{f,l}^{\partial}}{2|\kappa|}\beta_{f,l}\left[y_{\partial}^{j}\left(\xi\left(\zeta_{l}^{(f)}\right)\right)+\beta_{f,l}^{-1}\mathcal{F}_{\partial}^{v}\left(y_{\partial}^{j}\left(\xi\left(\zeta_{l}^{(f)}\right)\right),\nabla y_{\kappa}^{j}\left(\xi\left(\zeta_{l}^{(f)}\right)\right)\right)\cdot n\left(\zeta_{l}^{\left(f\right)}\right)\right]\nonumber \\
 &  & +\sum_{f=1}^{n_{f}}\sum_{l=1}^{n_{q,f}^{\partial}}\Lambda_{f,l}\Biggl[y_{\kappa}^{j}\left(\xi\left(\zeta_{l}^{(f)}\right)\right)+\frac{\Delta t^{*}\nu_{f,l}^{\partial}}{2|\kappa|}\Lambda_{f,l}^{-1}\mathcal{F}_{\partial}^{v}\left(y_{\partial}^{j}\left(\xi\left(\zeta_{l}^{(f)}\right)\right),\nabla y_{\kappa}^{j}\left(\xi\left(\zeta_{l}^{(f)}\right)\right)\right)\cdot n\left(\zeta_{l}^{\left(f\right)}\right)\Biggr].\label{eq:fully-discrete-form-average-viscous-2d-boundary}
\end{eqnarray}
Under the time-step-size constraint~(\ref{eq:CFL-condition-viscous-2D})
and the conditions~(\ref{eq:theta-conditions-2D}), we have
\begin{align*}
y_{\partial}^{j}\left(\xi\left(\zeta_{l}^{(f)}\right)\right)+\beta_{f,l}^{-1}\mathcal{F}_{\partial}^{v}\left(y_{\partial}^{j}\left(\xi\left(\zeta_{l}^{(f)}\right)\right),\nabla y_{\kappa}^{j}\left(\xi\left(\zeta_{l}^{(f)}\right)\right)\right)\cdot n\left(\zeta_{l}^{\left(f\right)}\right) & \in\mathcal{G},\\
y_{\kappa}^{j}\left(\xi\left(\zeta_{l}^{(f)}\right)\right)+\frac{\Delta t^{*}\nu_{f,l}^{\partial}}{2|\kappa|}\Lambda_{f,l}^{-1}\mathcal{F}_{\partial}^{v}\left(y_{\partial}^{j}\left(\xi\left(\zeta_{l}^{(f)}\right)\right),\nabla y_{\kappa}^{j}\left(\xi\left(\zeta_{l}^{(f)}\right)\right)\right)\cdot n\left(\zeta_{l}^{\left(f\right)}\right) & \in\mathcal{G},
\end{align*}
provided that the constraints on $\beta$ are modified as
\begin{align}
\beta_{f,l} & >\max\left\{ \beta_{f,l}^{(1)},\beta_{f,l}^{(2)}\right\} ,\label{eq:beta-constraint-2d-modified-1}\\
\beta_{f,l}^{(1)} & =\beta^{*}\left(y_{\kappa}^{j}\left(\xi\left(\zeta_{l}^{\left(f\right)}\right)\right),\mathcal{F}_{\partial}^{v}\left(y_{\partial}^{j}\left(\xi\left(\zeta_{l}^{(f)}\right)\right),\nabla y_{\kappa}^{j}\left(\xi\left(\zeta_{l}^{(f)}\right)\right)\right),n\left(\zeta_{l}^{\left(f\right)}\right)\right),\nonumber \\
\beta_{f,l}^{(2)} & =\beta^{*}\left(y_{\partial}^{j}\left(\xi\left(\zeta_{l}^{(f)}\right)\right),\mathcal{F}_{\partial}^{v}\left(y_{\partial}^{j}\left(\xi\left(\zeta_{l}^{(f)}\right)\right),\nabla y_{\kappa}^{j}\left(\xi\left(\zeta_{l}^{(f)}\right)\right)\right),n\left(\zeta_{l}^{\left(f\right)}\right)\right).\nonumber
\end{align}
$\overline{y}_{\kappa,v}^{j+1}$ is then in $\mathcal{G}$, and the
same limiting strategy can be applied to ensure $y_{\kappa,v}^{j+1}\in\mathcal{G},\;\forall x\in\mathcal{D_{\kappa}}$.

\subsection{Adaptive time stepping\label{subsec:Adaptive-time-stepping}}

As discussed by Zhang~\citep{Zha17}, the time-step-size constraints~(\ref{eq:CFL-condition-convective-2D})
and~(\ref{eq:CFL-condition-viscous-2D}) are sufficient but not necessary
for $\overline{y}_{\kappa,c}^{j+1}$ and $\overline{y}_{\kappa,v}^{j+1}$
to be in $\mathcal{G}_{s_{b}}$ and $\mathcal{G}$, respectively.
Furthermore, the latter constraint can sometimes be very restrictive.
In addition, as will be demonstrated in Section~\ref{subsec:detonation-results},
provided that the positivity property remains satisfied, the BR2 viscous
flux function is often preferred to the Lax-Friedrichs-type viscous
flux function. As such, unless otherwise specified, we employ the
following adaptive time stepping procedure similar to that in~\citep{Zha17},
except with additional steps to switch between the two viscous flux
functions:
\begin{enumerate}
\item Select $\Delta t$ according to a user-prescribed CFL based on the
acoustic time scale.
\item Compute $y_{c}^{j+1}$ and $y_{v}^{j+1}$ with the BR2 scheme.
\item If $\overline{y}_{\kappa,c}^{j+1}\in\mathcal{G}_{s_{b}}$ and $\overline{y}_{\kappa,v}^{j+1}\in\mathcal{G},\:\forall\kappa$,
then employ the limiting procedure, proceed to the next time step,
and go back to Step 1. If, for some $\kappa$, $\overline{y}_{\kappa,c}^{j+1}\notin\mathcal{G}_{s_{b}}$
or $\overline{y}_{\kappa,v}^{j+1}\notin\mathcal{G}$, then proceed
to Step 4.
\item Halve the time step, and recompute $y_{c}^{j+1}$ and $y_{v}^{j+1}$
with the BR2 scheme.
\item If $\overline{y}_{\kappa,c}^{j+1}\in\mathcal{G}_{s_{b}}$ and $\overline{y}_{\kappa,v}^{j+1}\in\mathcal{G},\:\forall\kappa$,
then employ the limiting procedure, proceed to the next time step,
and go back to Step 1. If, for some $\kappa$, $\overline{y}_{\kappa,c}^{j+1}\notin\mathcal{G}_{s_{b}}$
or $\overline{y}_{\kappa,v}^{j+1}\notin\mathcal{G}$, then proceed
to Step 6.
\item Recompute $y_{c}^{j+1}$ and $y_{v}^{j+1}$ with the Lax-Friedrichs-type
viscous flux function. Go back to Step 3.
\end{enumerate}
The above assumes forward Euler time integration. With SSPRK time
integration, the solution is restarted from time step $j$ (with the
time step halved or the viscous flux function switched) if an inadmissible
state is encountered at any stage. In our experience, the initial
time step size is generally sufficiently small for $\overline{y}_{\kappa,c}^{j+1}$
and $\overline{y}_{\kappa,v}^{j+1}$ to be in $\mathcal{G}_{s_{b}}$
and $\mathcal{G}$, respectively. Here, when the viscous flux function
is switched to the Lax-Friedrichs-type function, it is employed at
all interfaces. An alternative approach is to instead use it only
at the interfaces belonging to cells with inadmissible states.

In the present study, in typically no more than ten percent of time
steps is it necessary to decrease $\Delta t$ and/or switch the viscous
flux function. Note that the BR2 scheme can sometimes result in satisfaction
of the positivity property with a larger time step size than the Lax-Friedrichs-type
viscous flux function. However, the advantage of the latter is that
Theorem~\ref{thm:CFL-condition-2D-viscous} guarantees a finite time
step size. In Sections~\ref{subsec:shock-tube} and~\ref{subsec:detonation-results},
in order to compare the BR2 and Lax-Friedrichs-type viscous flux functions,
we employ the adaptive time stepping procedure but fix the viscous
flux function to be the latter in certain simulations.

\subsection{Zero species concentrations~\label{subsec:zero-species-concentrations}}

All species concentrations have hitherto been assumed to be strictly
positive. Following Remark~\ref{rem:zero-species-concentrations},
we relax this assumption and discuss issues that may occur in the
case of zero species concentrations. Note that the presence of zero
or near-zero species concentrations is very common (and expected)
in simulations of chemically reacting flows. We then propose a strategy
to address such issues. To this end, we first rewrite Equation~(\ref{eq:fully-discrete-form-average-viscous-2d})
in terms of the $i$th species-concentration component as
\begin{equation}
\begin{aligned}\overline{C}_{i,\kappa,v}^{j+1}= & \sum_{q=1}^{n_{q}}\theta_{q}C_{i,\kappa}^{j}\left(\xi_{q}\right)\\
 & +\sum_{f=1}^{n_{f}}\sum_{l=1}^{n_{q,f}^{\partial}}\frac{\Delta t^{*}\nu_{f,l}^{\partial}}{2|\kappa|}\beta_{f,l}\grave{C}_{i,\kappa^{\left(f\right)}}^{j}\left(\xi\left(\zeta_{l}^{\left(f\right)}\right)\right)\\
 & +\sum_{f=1}^{n_{f}}\sum_{l=1}^{n_{q,f}^{\partial}}\Lambda_{f,l}\grave{C}_{i,\kappa}^{j}\left(\xi\left(\zeta_{l}^{\left(f\right)}\right)\right),
\end{aligned}
\label{eq:fully-discrete-form-average-viscous-2d-species}
\end{equation}
where
\[
\grave{C}_{i,\kappa^{\left(f\right)}}^{j}\left(\xi\left(\zeta_{l}^{\left(f\right)}\right)\right)=C_{i,\kappa^{(f)}}^{j}\left(\xi\left(\zeta_{l}^{\left(f\right)}\right)\right)+\beta_{f,l}^{-1}\mathcal{F}_{C_{i}}^{v}\left(y_{\kappa^{(f)}}^{j}\left(\xi\left(\zeta_{l}^{\left(f\right)}\right)\right),\nabla C_{\kappa^{(f)}}^{j}\left(\xi\left(\zeta_{l}^{\left(f\right)}\right)\right)\right)\cdot n\left(\zeta_{l}^{\left(f\right)}\right)
\]
and
\[
\grave{C}_{i,\kappa}^{j}\left(\xi\left(\zeta_{l}^{\left(f\right)}\right)\right)=C_{i,\kappa}^{j}\left(\xi\left(\zeta_{l}^{\left(f\right)}\right)\right)+\frac{\Delta t^{*}\nu_{f,l}^{\partial}}{2|\kappa|}\Lambda_{f,l}^{-1}\mathcal{F}_{C_{i}}^{v}\left(y_{\kappa}^{j}\left(\xi\left(\zeta_{l}^{\left(f\right)}\right)\right),\nabla C_{\kappa}^{j}\left(\xi\left(\zeta_{l}^{\left(f\right)}\right)\right)\right)\cdot n\left(\zeta_{l}^{\left(f\right)}\right),
\]
with $\mathcal{F}_{C_{i}}^{v}$, the molar flux of the $i$th species,
given by Equation~(\ref{eq:molar_flux_definition}). Note that $\mathcal{F}_{C_{i}}^{v}$
depends on the concentration gradients, but not on the momentum gradient
or the total-energy gradient. As before, $y_{\kappa}^{j}$ is assumed
to be in $\mathcal{G},\;\forall x\in\mathcal{D_{\kappa}},$ for all
$\kappa$. We make the following observations.

\begin{rem}
\label{rem:zero-concentration-issues}Following Remarks~\ref{rem:zero-species-concentrations}
and~\ref{rem:beta-not-abstract-form}, the concentration-based constraint
on $\beta$~(\ref{eq:beta-constraint-concentration}) can be undefined
for $C_{i}=0$. Even if a small positive number is added to the denominator
of the term on the RHS, $\beta$ can still be exceedingly large, which
may then necessitate extremely small $\Delta t^{*}$ and/or introduce
noticeable error in the solution. Furthermore, as an example scenario
in which $\overline{C}_{i,\kappa,v}^{j+1}<0$ regardless of $\beta$,
take $\overline{C}_{i,\kappa}^{j}=0$, $C_{i,\kappa^{(f)}}^{j}=0,\:\forall x\in\partial\mathcal{D}_{\kappa}$,
and $\mathfrak{F}<0$, where
\[
\mathfrak{F}=\sum_{f=1}^{n_{f}}\sum_{l=1}^{n_{q,f}^{\partial}}\left[\mathcal{F}_{C_{i}}^{v}\left(y_{\kappa}^{j}\left(\xi\left(\zeta_{l}^{\left(f\right)}\right)\right),\nabla C_{\kappa}^{j}\left(\xi\left(\zeta_{l}^{\left(f\right)}\right)\right)\right)+\mathcal{F}_{C_{i}}^{v}\left(y_{\kappa^{(f)}}^{j}\left(\xi\left(\zeta_{l}^{\left(f\right)}\right)\right),\nabla C_{\kappa^{(f)}}^{j}\left(\xi\left(\zeta_{l}^{\left(f\right)}\right)\right)\right)\right]\cdot n\left(\zeta_{l}^{\left(f\right)}\right).
\]
A representative schematic is given in Figure~\ref{fig:zero-species-concentration-issue}.
In this figure, defining $\kappa=\left[x_{L},x_{R}\right]$, we have
\[
\mathfrak{F}=\left.\bar{D}_{i}\frac{\partial C_{i,\kappa^{(2)}}}{\partial x_{k}}\right|_{y_{\kappa^{(2)}}\left(x_{R}\right)},
\]
which is negative. Note that there exists a small region in $\kappa^{(2)}$
with negative species concentration, which is possible since the positivity-preserving
limiter only guarantees $C_{i}\geq0$ at a finite set of points.
\begin{figure}[H]
\begin{centering}
\includegraphics[width=0.6\columnwidth]{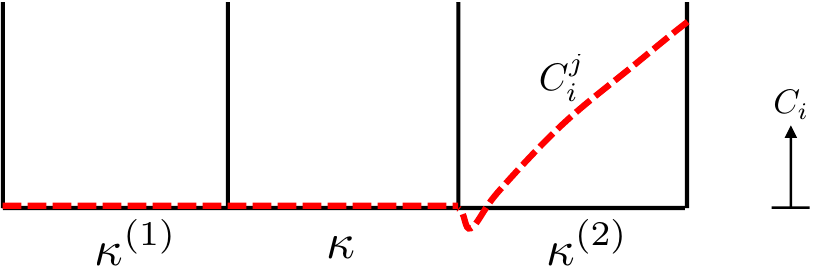}
\par\end{centering}
\caption{\label{fig:zero-species-concentration-issue}Schematic of an example
scenario in which $\overline{C}_{i,\kappa,v}^{j+1}<0$ regardless
of $\beta$.}
\end{figure}
\end{rem}

\begin{rem}
\label{rem:Cbar-p0}Suppose that $\overline{C}_{i,\kappa,v}^{j+1}<0$.
Applying the density limiter to $y_{\kappa}^{j}$ with $\omega_{\kappa}^{(1)}=0$
(which corresponds to a $p=0$ projection) and $y_{\kappa^{(f)}}^{j}$
with $\omega_{\kappa^{(f)}}^{(1)}=0,f=1,\ldots,n_{f}$ yields
\[
\begin{aligned}\overline{C}_{i,\kappa,v}^{j+1,(1)}= & \sum_{q=1}^{n_{q}}\theta_{q}\overline{C}_{i,\kappa}^{j}\left(\xi_{q}\right)\\
 & +\sum_{f=1}^{n_{f}}\sum_{l=1}^{n_{q,f}^{\partial}}\frac{\Delta t^{*}\nu_{f,l}^{\partial}}{2|\kappa|}\beta_{f,l}\overline{C}_{i,\kappa^{\left(f\right)}}^{j}\left(\xi\left(\zeta_{l}^{\left(f\right)}\right)\right)\\
 & +\sum_{f=1}^{n_{f}}\sum_{l=1}^{n_{q,f}^{\partial}}\Lambda_{f,l}\overline{C}_{i,\kappa}^{j}\left(\xi\left(\zeta_{l}^{\left(f\right)}\right)\right),
\end{aligned}
\]
which is positive.
\end{rem}

\begin{rem}
\label{rem:omega-g}$\mathcal{F}_{C_{i}}^{v}$ is directly proportional
to $\nabla C$, i.e.,
\begin{align*}
\mathcal{F}_{C_{i},k}^{v}\left(y,\omega\nabla C\right) & =\bar{D}_{i}\omega\frac{\partial C_{i}}{\partial x_{k}}-\frac{C_{i}\bar{D}_{i}}{\rho}\omega\frac{\partial\rho}{\partial x_{k}}-\frac{C_{i}}{\rho}\sum_{l=1}^{n_{s}}W_{l}\left(\bar{D}_{l}\omega\frac{\partial C_{l}}{\partial x_{k}}-\frac{C_{l}\bar{D}_{l}}{\rho}\omega\frac{\partial\rho}{\partial x_{k}}\right)\\
 & =\omega\left[\bar{D}_{i}\frac{\partial C_{i}}{\partial x_{k}}-\frac{C_{i}\bar{D}_{i}}{\rho}\frac{\partial\rho}{\partial x_{k}}-\frac{C_{i}}{\rho}\sum_{l=1}^{n_{s}}W_{l}\left(\bar{D}_{l}\frac{\partial C_{l}}{\partial x_{k}}-\frac{C_{l}\bar{D}_{l}}{\rho}\frac{\partial\rho}{\partial x_{k}}\right)\right]\\
 & =\omega\mathcal{F}_{C_{i},k}^{v}\left(y,\nabla C\right),
\end{align*}
where $\omega$ is a scaling factor for $\nabla C$.
\end{rem}

By Remark~\ref{rem:Cbar-p0}, a foolproof but low-fidelity approach
to address the issues described in Remark~\ref{rem:zero-concentration-issues}
is as follows: if $\overline{C}_{i,\kappa,v}^{j+1}<0$, then apply
the density limiter with $\omega^{(1)}=0$ and recalculate $y_{\kappa,v}^{j+1}$,
as well as the neighboring states. However, a higher-fidelity approach
is desired. Following Remark~\ref{rem:omega-g}, one such approach
is to modify the gradient in the fourth term in Equation~(\ref{eq:semi-discrete-form})
as
\begin{equation}
\left(\average{\mathcal{F}^{v}\left(y,\nabla y\right)}\cdot n-\delta^{v}\left(y,\nabla y,n\right),\left\llbracket \mathfrak{v}\right\rrbracket \right)_{\epsilon}\leftarrow\left(\average{\mathcal{F}^{v}\left(y,\omega\nabla y\right)}\cdot n-\delta^{v}\left(y,\nabla y,n\right),\left\llbracket \mathfrak{v}\right\rrbracket \right)_{\epsilon},\label{eq:flux-limit-gradient}
\end{equation}
where $\omega\in\left[0,1\right]$ is a pointwise parameter that scales
the gradient. Specifically, for a given $\xi\left(\zeta_{l}^{\left(f\right)}\right)$,
we have $\omega_{\kappa,l}^{(f)}$ and $\omega_{\kappa^{(f)},l}^{(f)}$
for the interior and exterior gradients, respectively, which yields
\begin{equation}
\grave{C}_{i,\kappa}^{j}\left(\xi\left(\zeta_{l}^{\left(f\right)}\right)\right)=C_{i,\kappa}^{j}\left(\xi\left(\zeta_{l}^{\left(f\right)}\right)\right)+\frac{\Delta t^{*}\nu_{f,l}^{\partial}}{2|\kappa|}\Lambda_{f,l}^{-1}\omega_{\kappa,l}^{(f)}\mathcal{F}_{C_{i}}^{v}\left(y_{\kappa}^{j}\left(\xi\left(\zeta_{l}^{\left(f\right)}\right)\right),\nabla C_{\kappa}^{j}\left(\xi\left(\zeta_{l}^{\left(f\right)}\right)\right)\right)\cdot n\left(\zeta_{l}^{\left(f\right)}\right)\label{eq:Cgrave-limited-gradient-kappa}
\end{equation}
and
\begin{equation}
\grave{C}_{i,\kappa^{\left(f\right)}}^{j}\left(\xi\left(\zeta_{l}^{\left(f\right)}\right)\right)=C_{i,\kappa^{(f)}}^{j}\left(\xi\left(\zeta_{l}^{\left(f\right)}\right)\right)+\beta_{f,l}^{-1}\omega_{\kappa^{(f)},l}^{(f)}\mathcal{F}_{C_{i}}^{v}\left(y_{\kappa^{(f)}}^{j}\left(\xi\left(\zeta_{l}^{\left(f\right)}\right)\right),\nabla C_{\kappa^{(f)}}^{j}\left(\xi\left(\zeta_{l}^{\left(f\right)}\right)\right)\right)\cdot n\left(\zeta_{l}^{\left(f\right)}\right).\label{eq:Cgrave-limited-gradient-kappa-f}
\end{equation}
This is akin to applying the linear-scaling limiter in Section~\ref{subsec:limiting-procedure}
to only the gradient.  In order to guarantee $\overline{C}_{i,\kappa,v}^{j+1}\geq0$,
$\omega_{\kappa,l}^{(f)}$ and $\omega_{\kappa^{(f)},l}^{(f)}$ in
Equations~(\ref{eq:Cgrave-limited-gradient-kappa}) and~(\ref{eq:Cgrave-limited-gradient-kappa-f})
can be prescribed as
\begin{equation}
\omega_{\kappa,l}^{(f)}=\min_{i}\omega_{i,\kappa,l}^{(f)},\quad\omega_{i,\kappa,l}^{(f)}=\begin{cases}
1, & Q_{i,\kappa,l}^{(f)}\geq0,\\
-\frac{\frac{1}{2\sum_{f}n_{q,f}^{\partial}}\sum_{q=1}^{n_{q}}\theta_{q}C_{i,\kappa}^{j}\left(\xi_{q}\right)+\Lambda_{f,l}C_{i,\kappa}^{j}\left(\xi\left(\zeta_{l}^{\left(f\right)}\right)\right)}{\frac{\Delta t^{*}\nu_{f,l}^{\partial}}{2|\kappa|}\mathcal{F}_{C_{i}}^{v}\left(y_{\kappa}^{j}\left(\xi\left(\zeta_{l}^{\left(f\right)}\right)\right),\nabla C_{\kappa}^{j}\left(\xi\left(\zeta_{l}^{\left(f\right)}\right)\right)\right)\cdot n\left(\zeta_{l}^{\left(f\right)}\right)}, & \mathrm{otherwise},
\end{cases}\label{eq:omega-f-kappa}
\end{equation}
and
\begin{equation}
\omega_{\kappa^{(f)},l}^{(f)}=\min_{i}\omega_{i,\kappa^{(f)},l}^{(f)},\quad\omega_{i,\kappa^{(f)},l}^{(f)}=\begin{cases}
1, & Q_{i,\kappa^{(f)},l}^{(f)}\geq0,\\
-\frac{\frac{1}{2\sum_{f}n_{q,f}^{\partial}}\sum_{q=1}^{n_{q}}\theta_{q}C_{i,\kappa}^{j}\left(\xi_{q}\right)+\frac{\Delta t^{*}\nu_{f,l}^{\partial}}{2|\kappa|}\beta_{f,l}C_{i,\kappa^{\left(f\right)}}^{j}\left(\xi\left(\zeta_{l}^{\left(f\right)}\right)\right)}{\frac{\Delta t^{*}\nu_{f,l}^{\partial}}{2|\kappa|}\mathcal{F}_{C_{i}}^{v}\left(y_{\kappa^{(f)}}^{j}\left(\xi\left(\zeta_{l}^{\left(f\right)}\right)\right),\nabla C_{\kappa^{(f)}}^{j}\left(\xi\left(\zeta_{l}^{\left(f\right)}\right)\right)\right)\cdot n\left(\zeta_{l}^{\left(f\right)}\right)}, & \mathrm{otherwise},
\end{cases}\label{eq:omega-f-kappa-f}
\end{equation}
respectively, where $Q_{i,\kappa,l}^{(f)}$ and $Q_{i,\kappa^{(f)},l}^{(f)}$
are defined as
\begin{align*}
Q_{i,\kappa,l}^{(f)}= & \frac{1}{2\sum_{f}n_{q,f}^{\partial}}\sum_{q=1}^{n_{q}}\theta_{q}C_{i,\kappa}^{j}\left(\xi_{q}\right)+\Lambda_{f,l}C_{i,\kappa}^{j}\left(\xi\left(\zeta_{l}^{\left(f\right)}\right)\right)\\
 & +\frac{\Delta t^{*}\nu_{f,l}^{\partial}}{2|\kappa|}\mathcal{F}_{C_{i}}^{v}\left(y_{\kappa}^{j}\left(\xi\left(\zeta_{l}^{\left(f\right)}\right)\right),\nabla C_{\kappa}^{j}\left(\xi\left(\zeta_{l}^{\left(f\right)}\right)\right)\right)\cdot n\left(\zeta_{l}^{\left(f\right)}\right),\\
Q_{i,\kappa^{(f)},l}^{(f)}= & \frac{1}{2\sum_{f}n_{q,f}^{\partial}}\sum_{q=1}^{n_{q}}\theta_{q}C_{i,\kappa}^{j}\left(\xi_{q}\right)+\frac{\Delta t^{*}\nu_{f,l}^{\partial}}{2|\kappa|}\beta_{f,l}C_{i,\kappa^{\left(f\right)}}^{j}\left(\xi\left(\zeta_{l}^{\left(f\right)}\right)\right)\\
 & +\frac{\Delta t^{*}\nu_{f,l}^{\partial}}{2|\kappa|}\mathcal{F}_{C_{i}}^{v}\left(y_{\kappa^{(f)}}^{j}\left(\xi\left(\zeta_{l}^{\left(f\right)}\right)\right),\nabla C_{\kappa^{(f)}}^{j}\left(\xi\left(\zeta_{l}^{\left(f\right)}\right)\right)\right)\cdot n\left(\zeta_{l}^{\left(f\right)}\right),
\end{align*}
such that $\overline{C}_{i,\kappa,v}^{j+1}$ in Equation~(\ref{eq:fully-discrete-form-average-viscous-2d-species})
can be rewritten as
\[
\overline{C}_{i,\kappa,v}^{j+1}=\sum_{f=1}^{n_{f}}\sum_{l=1}^{n_{q,f}^{\partial}}Q_{i,\kappa,l}^{(f)}+\sum_{f=1}^{n_{f}}\sum_{l=1}^{n_{q,f}^{\partial}}Q_{i,\kappa^{(f)},l}^{(f)}.
\]
$\beta_{f,l}$ can then be prescribed using only the $\beta_{T}$
constraint~(\ref{eq:beta_T}) (i.e., the constraint~(\ref{eq:beta-constraint-concentration})
can be ignored); furthermore, $\beta_{f,l}$ does not need to be recomputed
since $\mathcal{F}^{v}\left(y,\omega\nabla y\right)=G\left(y\right):\omega\nabla y=\omega\mathcal{F}^{v}\left(y,\nabla y\right)$
and $\beta^{*}\left(y,\mathcal{F}^{v}\left(y,\nabla y\right),n\right)\geq\beta^{*}\left(y,\omega\mathcal{F}^{v}\left(y,\nabla y\right),n\right)=\omega\beta^{*}\left(y,\mathcal{F}^{v}\left(y,\nabla y\right),n\right)$,
for $\omega\in\left[0,1\right]$. It should also be noted that the
constraint~(\ref{eq:beta-constraint-concentration}) can be overly
restrictive, such that nonnegativity of species concentrations can
often be maintained even if the constraint~(\ref{eq:beta-constraint-concentration})
is neglected and $\omega_{\kappa,l}^{(f)}=\omega_{\kappa^{(f)},l}^{(f)}=1$
for all $l,f$. I

In this study, we limit the gradient at interfaces belonging to cells
for which $\Delta t$ needs to be halved three or more times for $\overline{C}_{i,\kappa,v}^{j+1}\geq0$.
The time-step size required to ensure nonnegative species concentrations
can be computed algebraically by rewriting Equation~(\ref{eq:fully-discrete-form-average-viscous-2d})
in terms of species concentrations and setting the LHS to zero. Note
that the need to limit the gradient is extremely rare. In this work,
gradient limiting is applied only in early time steps in certain two-dimensional
detonation simulations, which will be further discussed in Section~\ref{subsec:detonation-results}.
An alternative to (\ref{eq:flux-limit-gradient}) is to instead apply
the density limiter in Section~\ref{subsec:limiting-procedure} to
the state and modify the fourth term in Equation~(\ref{eq:semi-discrete-form})
as
\begin{equation}
\left(\average{\mathcal{F}^{v}\left(y,\nabla y\right)}\cdot n-\delta^{v}\left(y,\nabla y,n\right),\left\llbracket \mathfrak{v}\right\rrbracket \right)_{\epsilon}\leftarrow\left(\average{\mathcal{F}^{v}\left(\ddot{y},\nabla\ddot{y}\right)}\cdot n-\delta^{v}\left(\ddot{y},\nabla\ddot{y},n\right),\left\llbracket \mathfrak{v}\right\rrbracket \right)_{\epsilon},\label{eq:flux-limit-state}
\end{equation}
where
\[
\ddot{y}=\left(\rho v_{1},\ldots,\rho v_{d},\rho e_{t},\ddot{C_{1}},\ldots,\ddot{C}_{n_{s}}\right),\quad\ddot{C_{i}}=\overline{C}_{i}+\omega\left(C_{i}-\overline{C}_{i}\right).
\]
However, iteration would then be required to determine $\omega_{\kappa,l}^{(f)}$
and $\omega_{\kappa^{(f)},l}^{(f)}$ such that $\overline{C}_{i,\kappa,v}^{j+1}\geq0$.

\section{Results\label{sec:results}}

We consider three one-dimensional test cases: advection-diffusion
of a thermal bubble, a premixed flame, and viscous shock-tube flow.
Next, we compute two multidimensional reacting flows: a moving detonation
wave enclosed by adiabatic walls and shock/mixing-layer interaction.
Unless otherwise specified, the adaptive time stepping strategy described
in Section~\ref{subsec:Adaptive-time-stepping} with second-order
strong-stability-preserving Runge-Kutta method (SSPRK2)~\citep{Got01,Spi02}
is employed. All simulations are performed using a modified version
of the JENRE\textregistered~Multiphysics Framework~\citep{Cor18_SCITECH,Joh20_2}
that incorporates the developments and extensions described in this
work.

\subsection{One-dimensional thermal bubble advection-diffusion\label{subsec:thermal-bubble}}

In this problem, we assess the order of accuracy of the positivity-preserving
and entropy-bounded DG formulation (without artificial viscosity).
The computational domain is $\Omega=[-25,25]\:\mathrm{m}$. Periodicity
is imposed at the left and right boundaries. The initial conditions
are given by

\begin{eqnarray}
v_{1} & = & 1\textrm{ m/s},\nonumber \\
Y_{H_{2}} & = & \frac{1}{2}\left[1-\tanh\left(|x|-10\right)\right],\nonumber \\
Y_{O_{2}} & = & 1-Y_{H_{2}},\label{eq:thermal-bubble}\\
T & = & 1200-900\tanh\left(|x|-10\right)\textrm{ K},\nonumber \\
P & = & 1\textrm{ bar}.\nonumber
\end{eqnarray}
In~\citep{Joh20_2}, optimal convergence without any additional stabilization,
including limiting, was demonstrated. In~\citep{Chi22}, we showed
optimal convergence from $p=1$ to $p=3$ using the positivity-preserving,
entropy-bounded DG method for \emph{inviscid}, reacting flows. Four
element sizes were considered: $h$, $h/2$, $h/4$, and $h/8$, where
$h=2\:\mathrm{m}$. The limiters were not activated when finer meshes
were employed. Here, we repeat this investigation in the viscous setting.
The thermodynamic relations can be found in~\citep[Section 8.1]{Chi22}.
Instead of the adaptive time stepping strategy described in Section~\ref{subsec:Adaptive-time-stepping},
we separately consider both viscous flux functions with fixed $\mathrm{CFL}=0.1$
to minimize temporal errors. The ``exact'' solution is obtained
with $p=3$ and $h/256$. The $L^{2}$ error at $t=5\:\mathrm{s}$
is computed in terms of the normalized state variables,
\[
\widehat{\rho v}_{k}=\frac{1}{\sqrt{\rho_{r}P_{r}}}\rho v_{k},\quad\widehat{\rho e}_{t}=\frac{1}{P_{r}}\rho e_{t},\quad\widehat{C}_{i}=\frac{R^{0}T_{r}}{P_{r}}C_{i},
\]
where $\rho_{r}=1\,\mathrm{kg\cdot}\mathrm{m}^{-3}$, $T_{r}=1000\,\mathrm{K}$,
and $P_{r}=101325\,\mathrm{Pa}$. Figure~\ref{fig:thermal-bubble-convergence}
shows the convergence results for both viscous flux functions. The
theoretical convergence rates are denoted with dashed lines. The ``$\times$''
symbol indicates that the positivity-preserving limiter is activated,
the ``$\Circle$'' symbol indicates that the entropy limiter is
activated, and the ``$\triangle$'' symbol indicates that neither
limiter is activated. If both limiters are activated, then the corresponding
symbols are superimposed as ``$\otimes$''. The results are extremely
similar between the two viscous flux functions. Apart from the coarser
grids with $p=1$, which are likely outside the asymptotic regime,
optimal convergence is demonstrated. For $h$ and $h/2$, both limiters
are activated across all $p$; for $h/4$ and $p=1$, only the positivity-preserving
limiter is activated. At higher resolutions, the limiters are not
engaged since the solutions are fairly well-resolved.
\begin{figure}[H]
\subfloat[\label{fig:thermal_bubble_convergence_diffusion_BR2}BR2 viscous flux
function.]{\includegraphics[width=0.45\columnwidth]{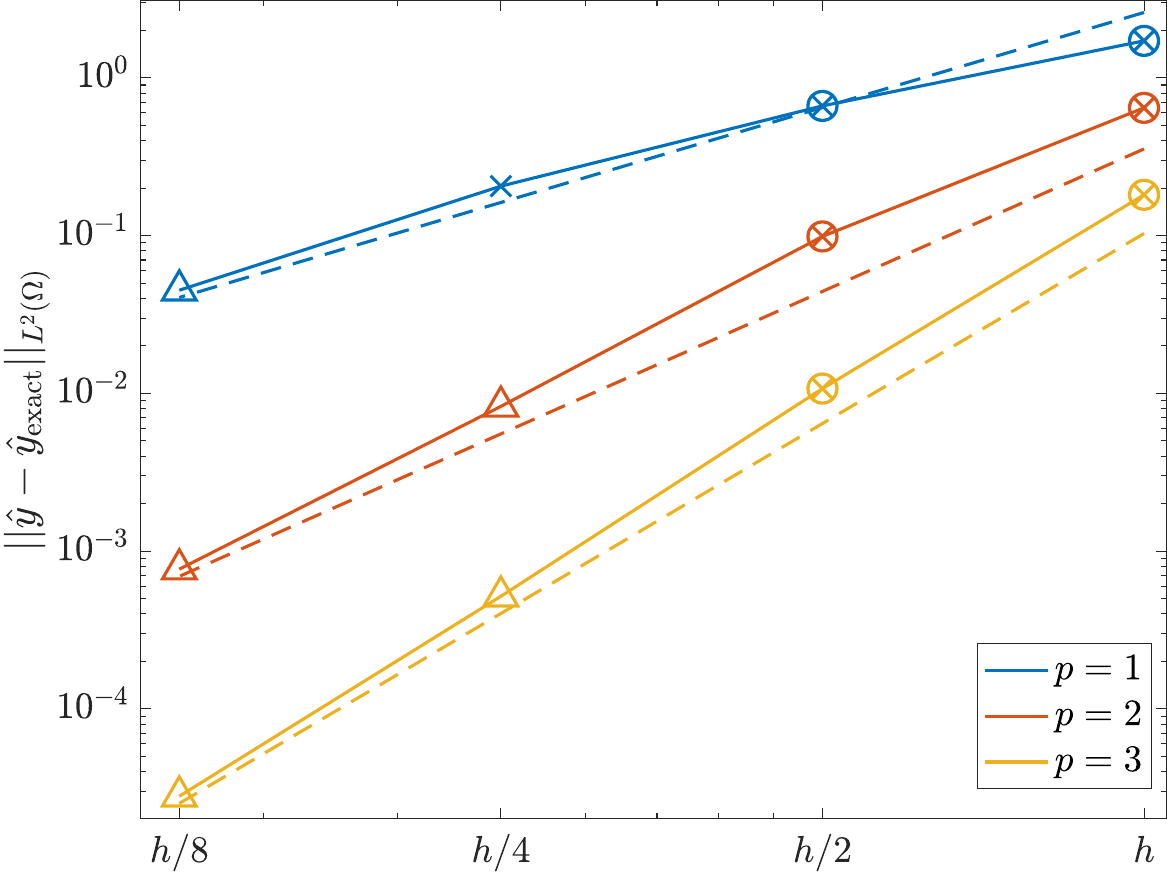}}\hfill{}\subfloat[\label{fig:thermal_bubble_convergence_diffusion_LLF}Lax-Friedrichs-type
viscous flux function.]{\includegraphics[width=0.45\columnwidth]{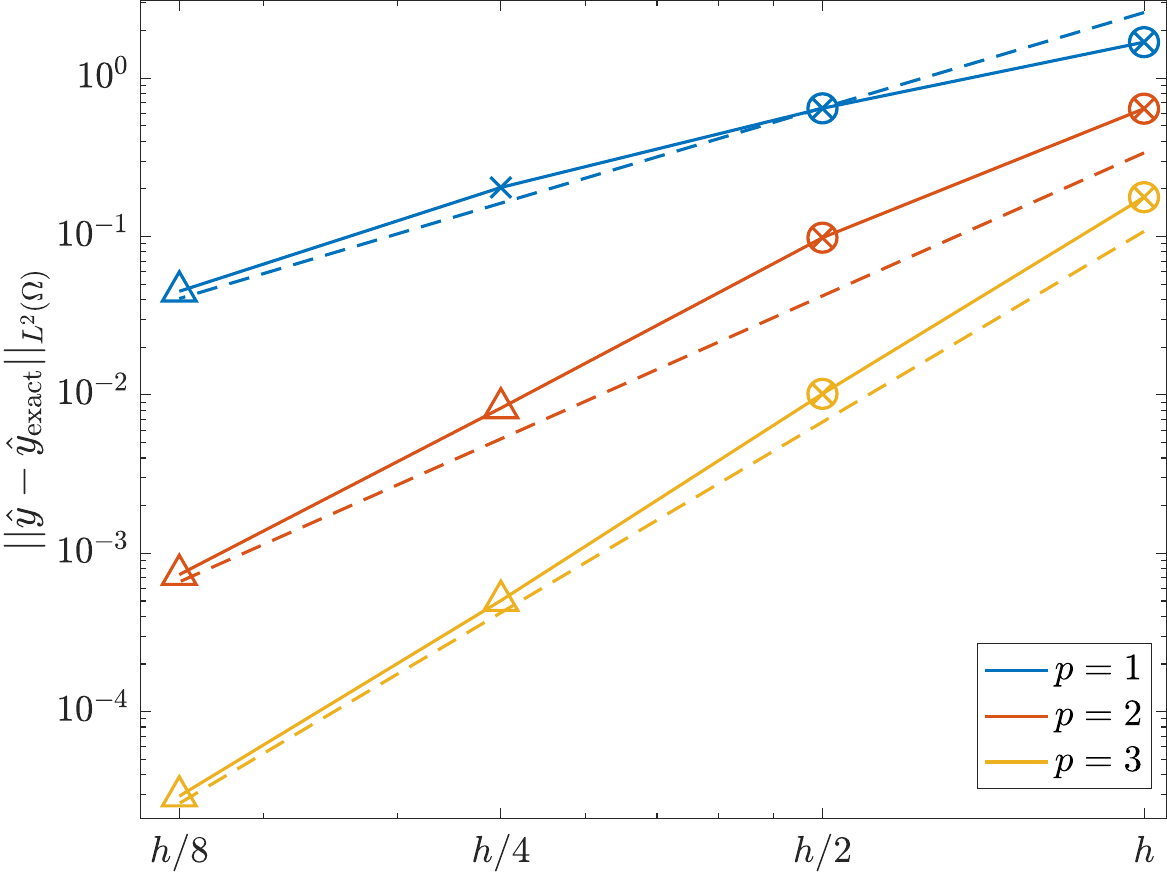}}

\caption{\label{fig:thermal-bubble-convergence} Convergence under grid refinement,
with $h=2\:\mathrm{m}$, for the one-dimensional thermal bubble test
case. The $L^{2}$ error of the normalized state with respect to the
exact solution at $t=5\:\mathrm{s}$ is computed. The dashed lines
represent the theroretical convergence rates. The ``$\times$''
symbol indicates that the positivity-preserving limiter is activated,
the ``$\Circle$'' symbol indicates that the entropy limiter is
activated, and the ``$\triangle$'' symbol indicates that neither
limiter is activated. If both limiters are activated, then the corresponding
symbols are superimposed as ``$\otimes$''.}
\end{figure}

\subsection{One-dimensional premixed flame~\label{subsec:1d-flame}}

In this problem, we consider a smooth, viscous flow with chemical
reactions. A freely propagating flame is calculated in Cantera~\citep{cantera}
on a 1 cm long grid using the left state in Equations~(\ref{eq:freely_propagating_flame_1})
and~(\ref{eq:freely_propagating_flame_2}) below. The computational
domain is $\Omega=[0,0.01]\text{ m}$. For the DG calculations, we
generate a mesh that contains a refinement zone between 1.8 mm and
2.5 mm with grid spacing $h=150$~$\mu$m, a target size that is
150 times larger than the smallest grid spacing from the resulting
refinement procedure in Cantera. The mesh transitions to a spacing
of $500$~$\mu$m at the boundaries. Note that the Cantera solution,
unlike the DG solution, assumes constant pressure. The objective here
is to ignite the flame and establish a solution in which the flame
anchors itself in the fine region of the one-dimensional mesh. The
initial conditions are given by

\begin{eqnarray}
\left(v_{1},T,P\right) & = & \begin{cases}
\left(9.53\text{ m/s},2122\text{ K},1\text{ atm}\right), & x\geq2.5\text{ mm}\\
\left(1.53\text{ m/s},300\text{ K},1\text{ atm}\right), & x<2.5\text{ mm}
\end{cases},\label{eq:freely_propagating_flame_1}
\end{eqnarray}
with mass fractions
\begin{equation}
\begin{split}\left(Y_{H_{2}},Y_{O_{2}},Y_{N_{2}},Y_{H},Y_{O}\right)= & \begin{cases}
\left(7\times10^{-5},0.0572,0.745,4.2\times10^{-6},2.2\times10^{-4}\right), & x\geq2.5\text{ mm}\\
\left(0.023,0.24,0.737,0,0\right), & x<2.5\text{ mm}
\end{cases}\\
\left(Y_{OH},Y_{H_{2}O},Y_{HO_{2}},Y_{H_{2}O_{2}}\right)= & \begin{cases}
\left(2.7\times10^{-4},0.194,3\times10^{-6},2.1\times10^{-7}\right), & x\geq2.5\text{ mm}\\
\left(0,0,0,0\right), & x<2.5\text{ mm}
\end{cases}
\end{split}
.\label{eq:freely_propagating_flame_2}
\end{equation}
The right state corresponding to $x\geq2.5\text{ mm}$ is the final
fully reacted state from the Cantera solution. The left boundary condition
is a characteristic inflow condition that allows pressure waves to
leave the domain. The right boundary is a reflective outflow condition
with the pressure set to 1 atm. The chemical mechanism used here is
based on the Westbrook mechanism and can be found in~\citep[Appendix D]{Chi22_2}.

We perform $p=1$ and $p=3$ calculations without artificial viscosity.
A $p=1$ solution with conventional species clipping (instead of the
positivity-preserving and entropy limiters), in which negative species
concentrations are simply set to zero, is also computed. The default
CFL is set to 0.4. In the beginning of the simulation, the states
on both sides of the discontinuity immediately diffuse to form a smooth
profile. As the reactions progress, the flame accelerates against
the right-moving reactants and then slows down to the flame speed.
 These initial transient processes can cause slight movement of the
flame at early times, especially in coarser solutions, and are not
directly accounted for in the Cantera simulation; therefore, in the
comparisons below, the Cantera solution is shifted such that temperature
is equal to 1000 K at the same point as the corresponding DG solution.

Figure~\ref{fig:diffusion_flame_clipping} shows instantaneous solutions
at $t=0.005$~s for $p=1$ with species clipping. Clear discrepancies
between the solution and the Cantera solution are observed. The velocity
and the mass fractions of $\mathrm{OH}$ and $\mathrm{H_{2}O_{2}}$
are overpredicted.

The $p=1$ and $p=3$ results obtained with the proposed positivity-preserving,
entropy-bounded methodology are given in Figures~\ref{fig:diffusion_limited_p1}
and~\ref{fig:diffusion_limited_p3}, respectively. Although the $p=1$
solution does not fully capture the species profiles and overpredicts
the velocity, it agrees much more closely with the Cantera solution
than the $p=1$ solution obtained with species clipping. The velocity
and mass-fraction profiles of the $p=3$ solution are in noticeably
better agreement with those of the Cantera solution than the $p=1$
solution. These results illustrate the benefits of employing the
proposed positivity-preserving, entropy-bounded DG formulation.

\begin{figure}[H]
\subfloat[Temperature.]{\includegraphics[width=0.31\columnwidth]{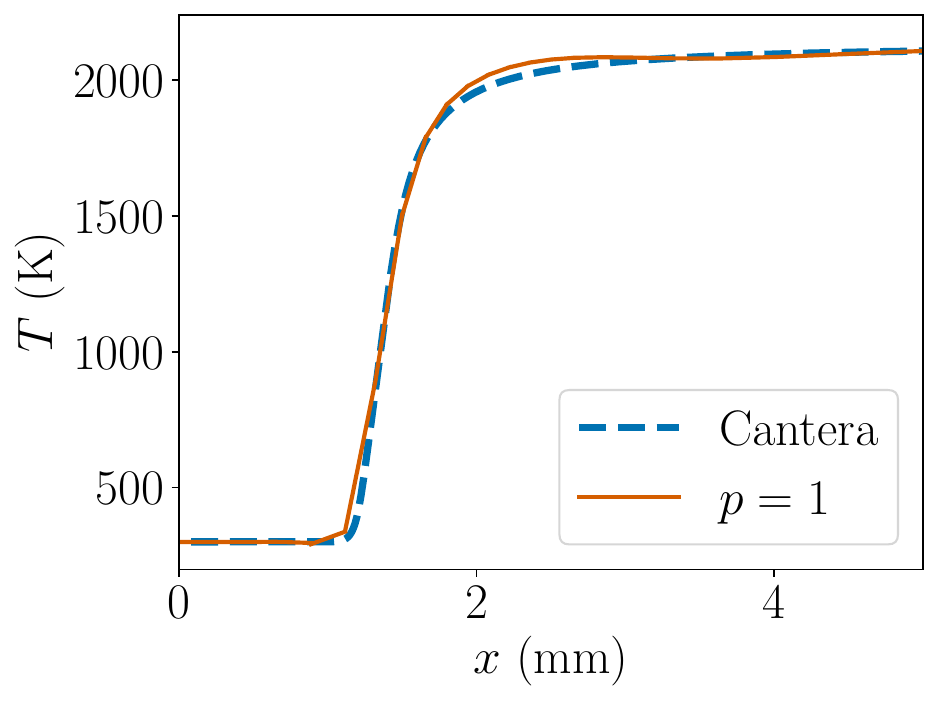}}\hfill{}\subfloat[Velocity.]{\includegraphics[width=0.29\columnwidth]{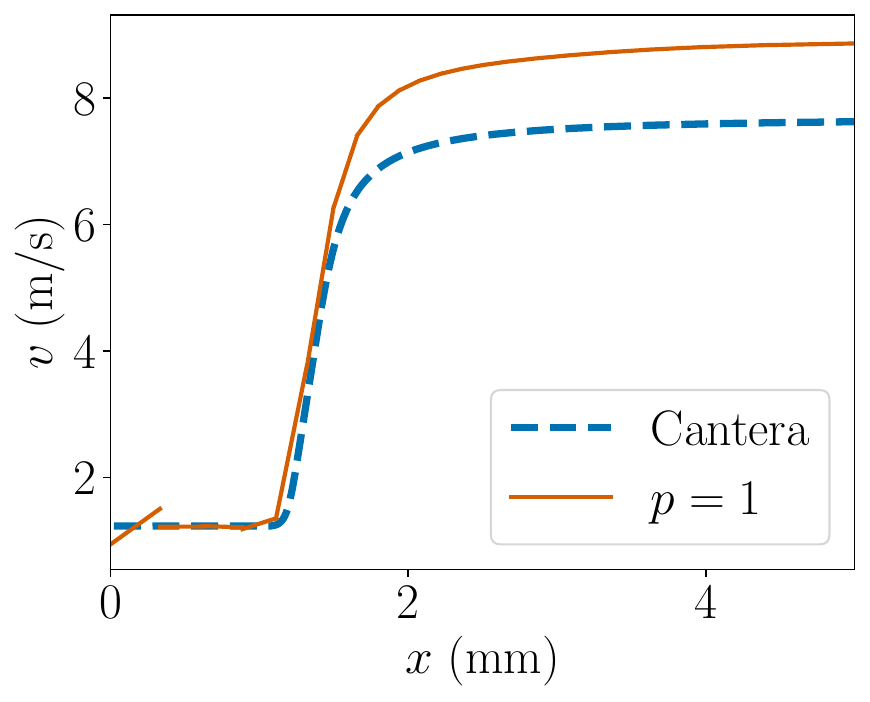}}\hfill{}\subfloat[Mass fractions.]{\includegraphics[width=0.36\columnwidth]{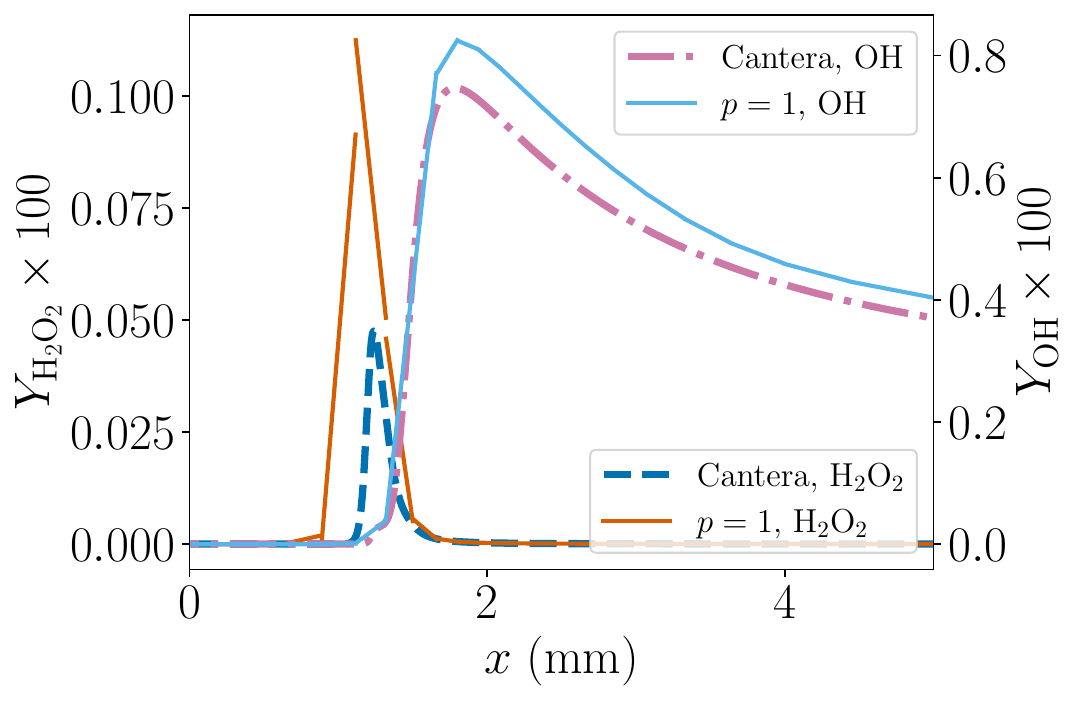}}

\caption{\label{fig:diffusion_flame_clipping}$p=1$ solution to a one-dimensional
premixed flame at $t=0.005$~s obtained with species clipping.}
\end{figure}

\begin{figure}[H]
\subfloat[Temperature.]{\includegraphics[width=0.31\columnwidth]{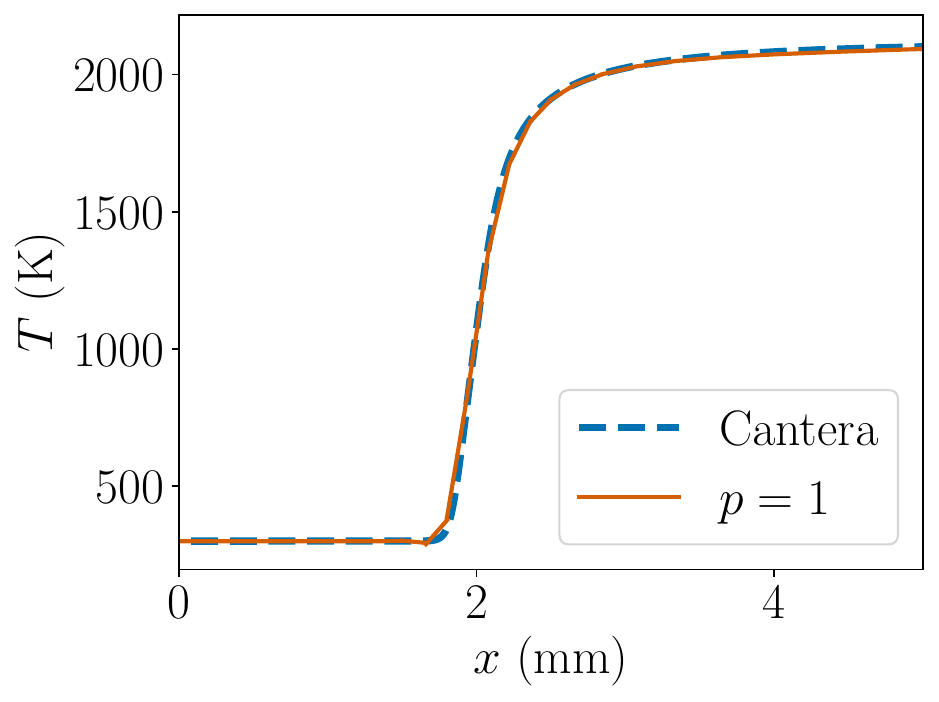}}\hfill{}\subfloat[Velocity.]{\includegraphics[width=0.29\columnwidth]{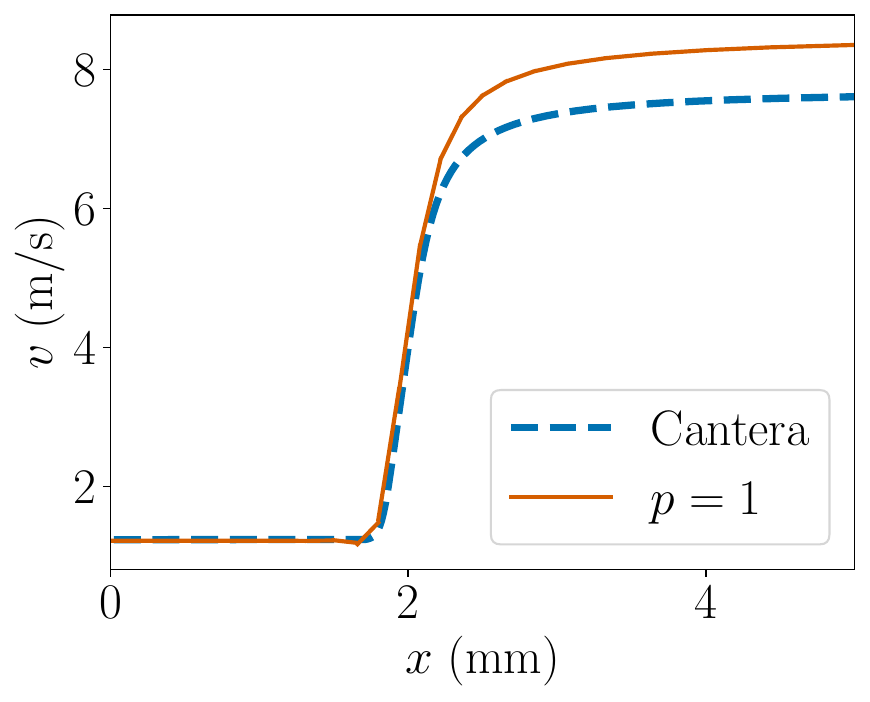}}\hfill{}\subfloat[Mass fractions.]{\includegraphics[width=0.36\columnwidth]{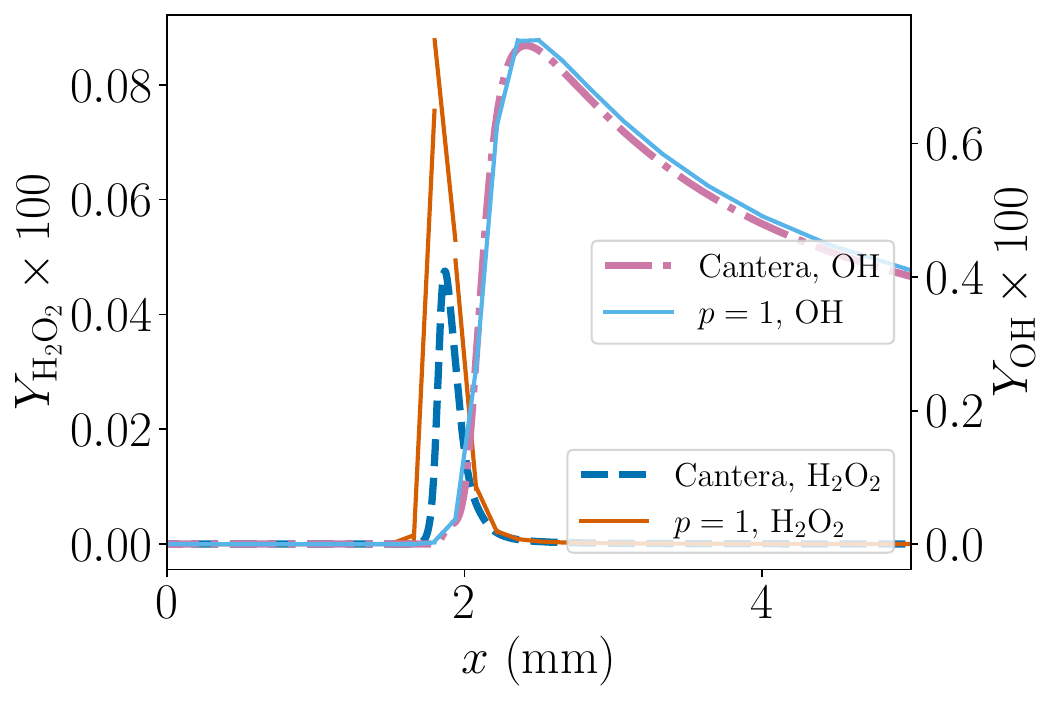}}

\caption{\label{fig:diffusion_limited_p1}$p=1$ solution to a one-dimensional
premixed flame at $t=0.005$~s obtained with the proposed positivity-preserving
and entropy-bounded formulation.}
\end{figure}
\begin{figure}[H]
\subfloat[Temperature.]{\includegraphics[width=0.31\columnwidth]{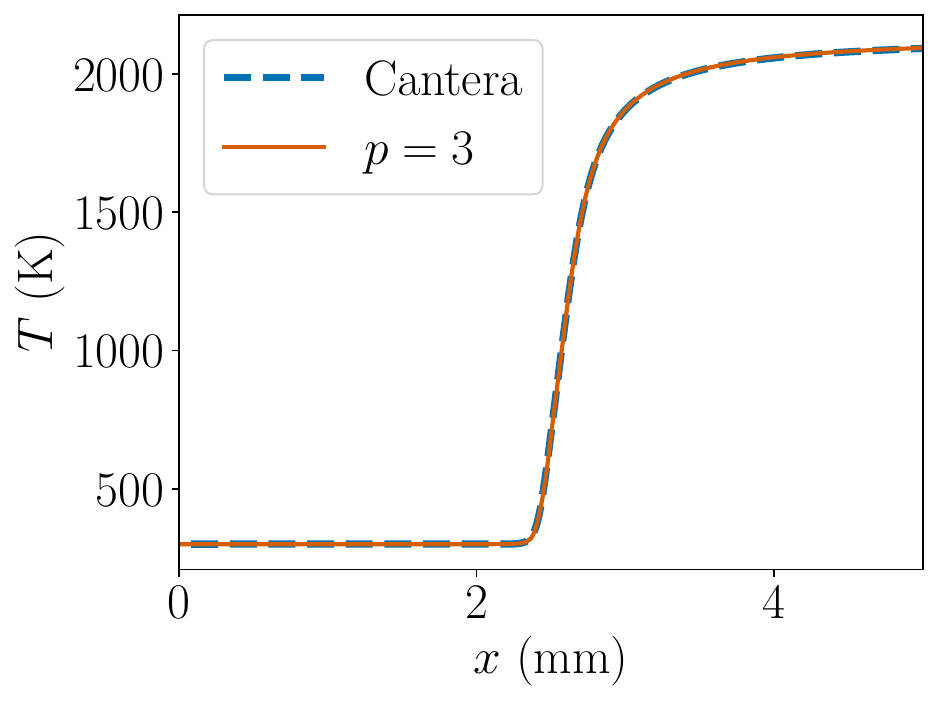}}\hfill{}\subfloat[Velocity.]{\includegraphics[width=0.29\columnwidth]{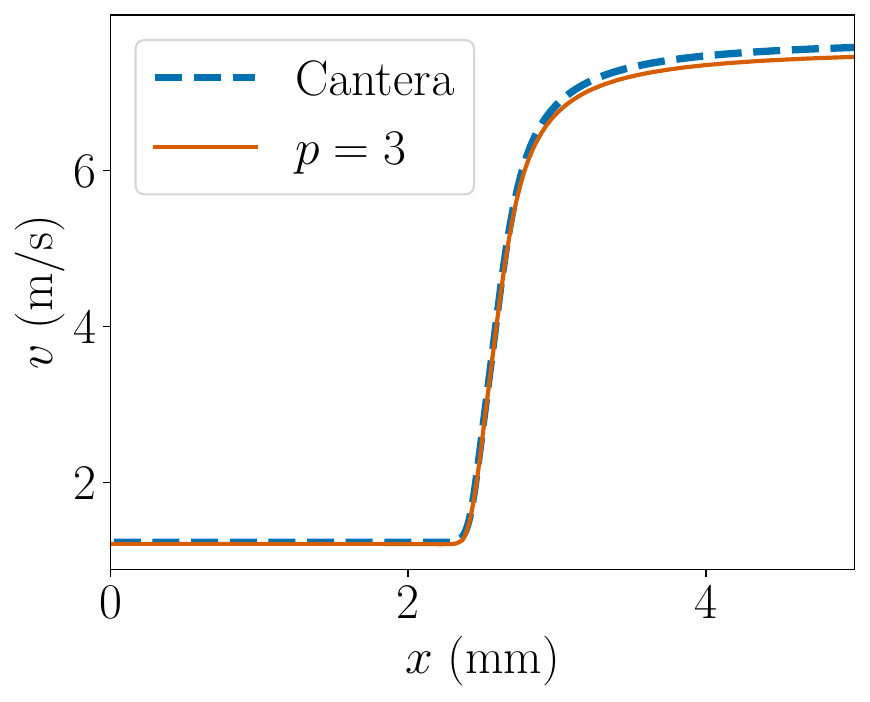}}\hfill{}\subfloat[Mass fractions.]{\includegraphics[width=0.36\columnwidth]{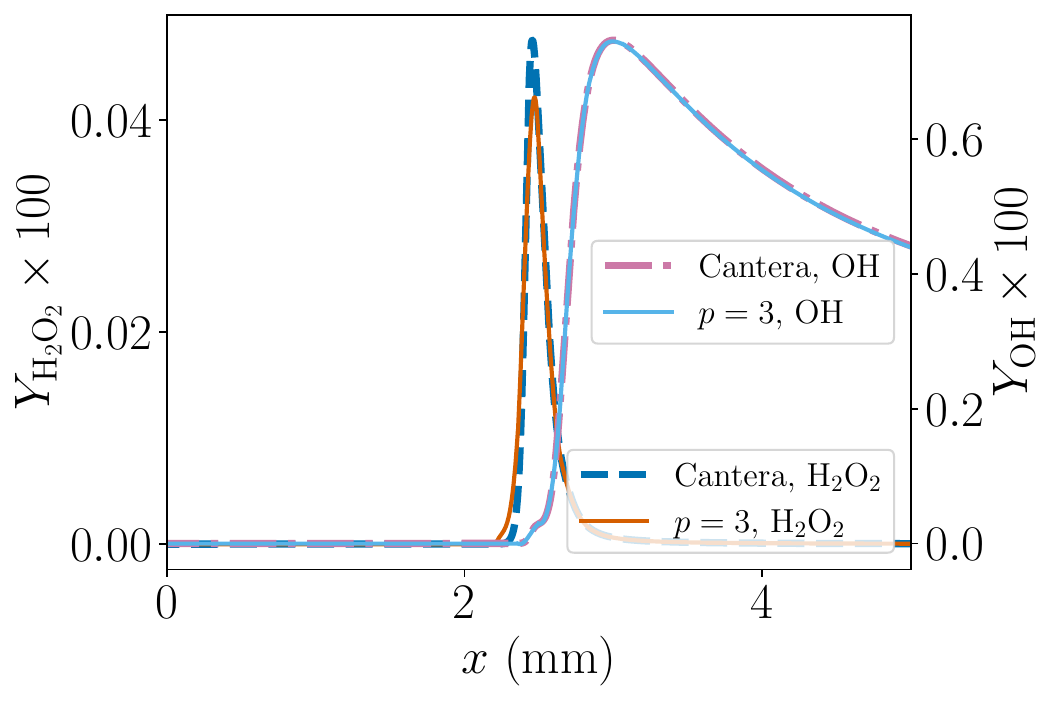}}

\caption{\label{fig:diffusion_limited_p3}$p=3$ solution to a one-dimensional
premixed flame at $t=0.005$~s obtained with the proposed positivity-preserving
and entropy-bounded formulation.}
\end{figure}

\subsection{One-dimensional shock tube\label{subsec:shock-tube}}

This test case was computed without viscous effects by Houim and Kuo~\citep{Hou11},
by Johnson and Kercher~\citep{Joh20_2}, and in our previous work~\citep{Chi22},
where we showed that (a) instabilities in the multicomponent, thermally
perfect case are much greater than in the monocomponent, calorically
perfect case and (b) enforcement of an entropy bound suppresses large-scale
nonphysical oscillations much more effectively than enforcement of
the positivity property. Our goals here are to investigate whether
these observations hold in the viscous setting and to further compare
the BR2 and Lax-Friedrichs-type viscous flux functions. The computational
domain is $\Omega=[0,1]\text{ m}$, and the final time is $t=300\;\mu\mathrm{s}$.
Walls are imposed at the left and right boundaries. The initial conditions
are written as
\begin{equation}
\left(v_{1},T,P,Y_{N_{2}},Y_{He}\right)=\begin{cases}
\left(0\text{ m/s},300\text{ K},1\text{ atm},1,0\right), & x\geq0.4\\
\left(0\text{ m/s},300\text{ K},10\text{ atm},0,1\right), & x<0.4
\end{cases}.\label{eq:shock-tube-IC-Houim}
\end{equation}
For consistency with~\citep{Chi22}, the default CFL is set to 0.1.
For the remainder of this subsection, ``BR2'' refers to the adaptive
time stepping strategy exactly as described in Section~\ref{subsec:Adaptive-time-stepping},
whereas ``LLF'' refers to a similar time stepping strategy, but
with the viscous flux function fixed to be the local Lax-Friedrichs-type
flux function. In addition, ``PPL'' corresponds to only the positivity-preserving
limiter, while ``EL'' corresponds to both the positivity-preserving
and entropy limiters. Based on~\citep{Joh20_2} and~\citep{Chi22},
a reference solution is computed using $p=2$, 2000 elements, artificial
viscosity, BR2, and EL. All other solutions are computed using $p=3$
and 200 elements. The thermodyamic relations can be found in~\citep[Section 8.3]{Chi22}.

Figure~\ref{fig:shock_tube_BR2} shows the mass fraction, pressure,
temperature, and entropy profiles obtained with BR2. Except for the
reference solution, artificial viscosity is not employed in order
to isolate the effects of the limiters. Note that the linear-scaling
limiters alone are not expected to eliminate small-scale spurious
oscillations~\citep{Zha10,Jia18,Lv15_2,Wu21_2}. The results are
very similar to those in the inviscid case~\citep{Chi22}. The species
profiles are well-captured using both types of limiting. The entropy
limiter dampens large-scale instabilities in the pressure, temperature,
and entropy distributions significantly better than the positivity-preserving
limiter. Furthermore, just as observed in~\citep{Chi22}, the instabilities
still present with the positivity-preserving limiter are substantially
larger than those usually present in monocomponent, calorically perfect
shock-tube solutions computed with the positivity-preserving limiter~\citep{Zha10,Zha12_2,Zha17},
and the relative advantage of applying the entropy limiter is much
greater. The addition of artificial viscosity would greatly suppress
the small-scale instabilities; for brevity, such results are not included
here, but they are very similar to those in~\citep{Chi22}. At the
same time, artificial viscosity alone (without the limiters) results
in negative concentrations and other instabilities, thus motivating
a combination of the two stabilization mechanisms. The corresponding
LLF results are given in Figure~\ref{fig:shock_tube_LLF}, which
are very similar to the BR2 results. However, the temperature overshoot
at the shock is noticeably smaller in the LLF case, indicating that
the Lax-Friedrichs-type viscous flux function can sometimes have better
stabilization properties than the BR2 scheme. Regardless, the results
in the following subsection suggest that the latter is still the preferred
viscous flux function, provided that the positivity property is satisfied.

\begin{figure}[H]
\subfloat[\label{fig:shock_tube_Y_BR2}Mass fractions.]{\includegraphics[width=0.45\columnwidth]{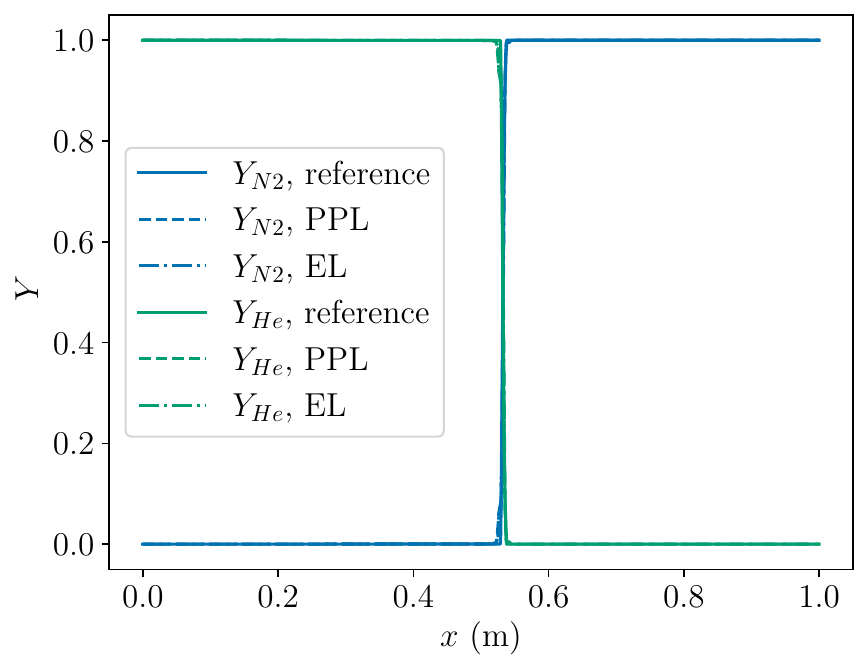}}\hfill{}\subfloat[\label{fig:shock_tube_P_BR2}Pressure.]{\includegraphics[width=0.45\columnwidth]{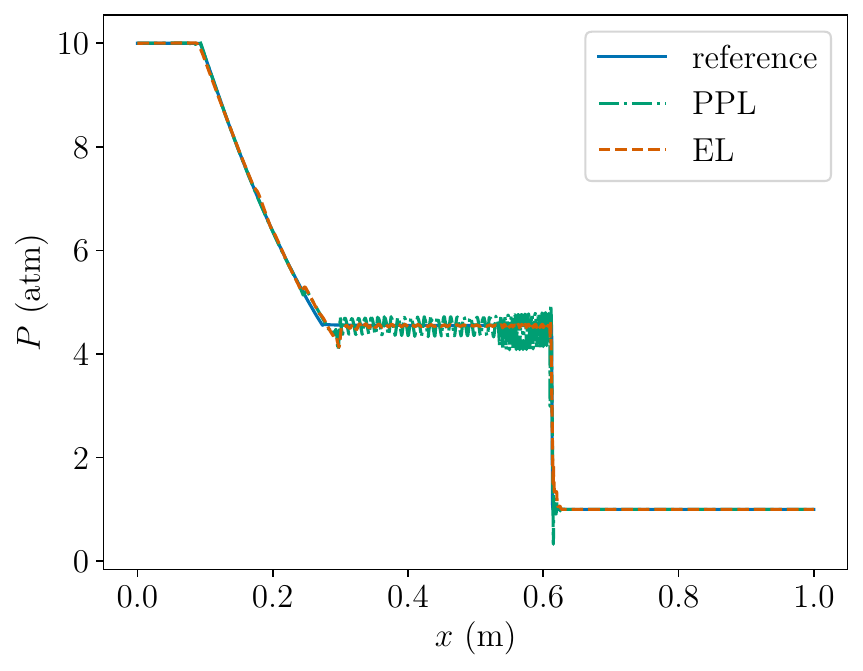}}\hfill{}\subfloat[\label{fig:shock_tube_T_BR2}Temperature.]{\includegraphics[width=0.48\columnwidth]{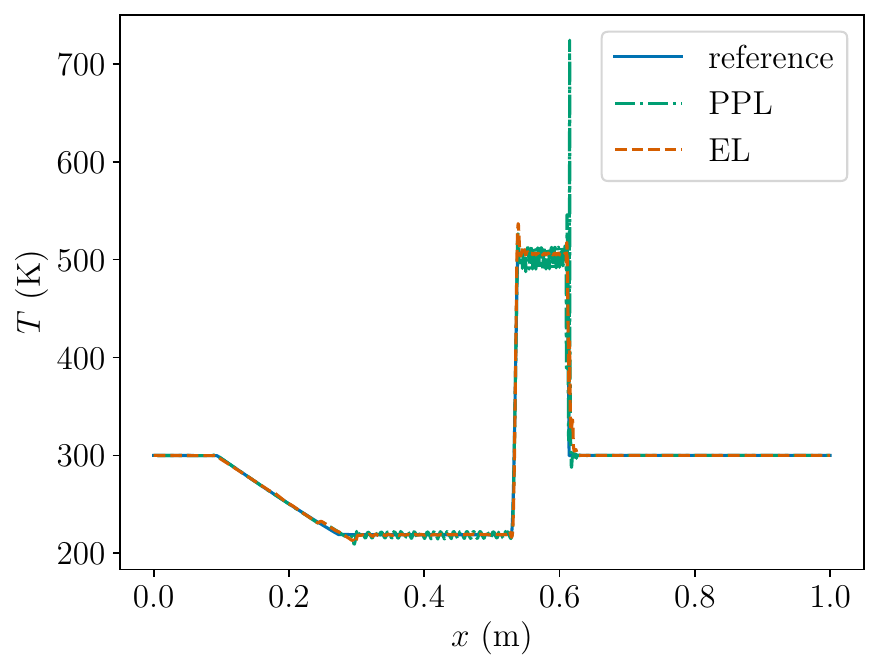}}\hfill{}\subfloat[\label{fig:shock_tube_s_BR2}Specific thermodynamic entropy.]{\includegraphics[width=0.45\columnwidth]{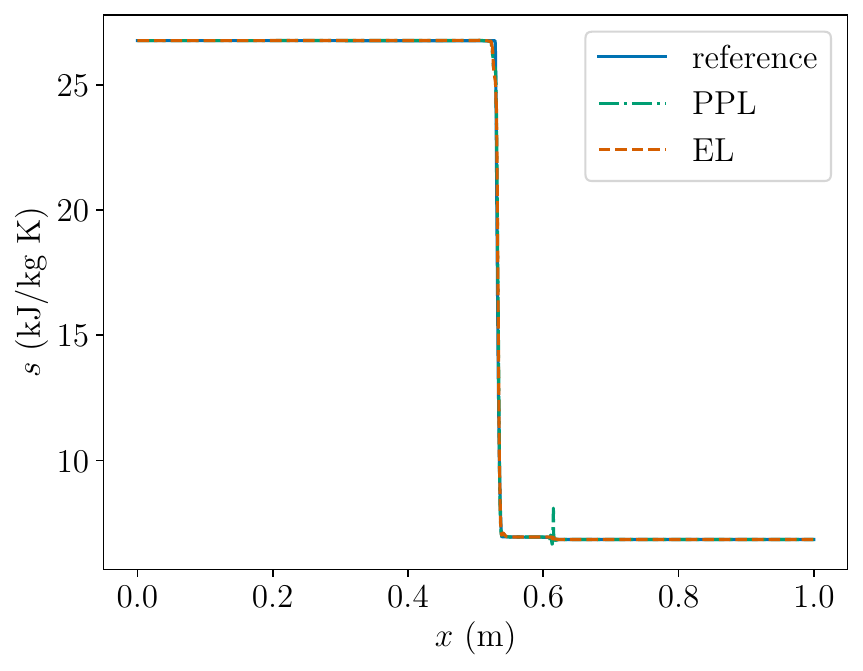}}

\caption{\label{fig:shock_tube_BR2}Results for $p=3$ solutions computed using
BR2 on 200 elements without artificial viscosity for the one-dimensional,
multicomponent shock-tube problem with initialization in Equation~(\ref{eq:shock-tube-IC-Houim}).
``PPL'' corresponds to the positivity-preserving limiter by itself,
and ``EL'' refers to both the positivity-preserving and entropy
limiters with the local entropy bound in Equation~(\ref{eq:local-entropy-bound}).}
\end{figure}

\begin{figure}[H]
\subfloat[\label{fig:shock_tube_Y_LLF}Mass fractions.]{\includegraphics[width=0.45\columnwidth]{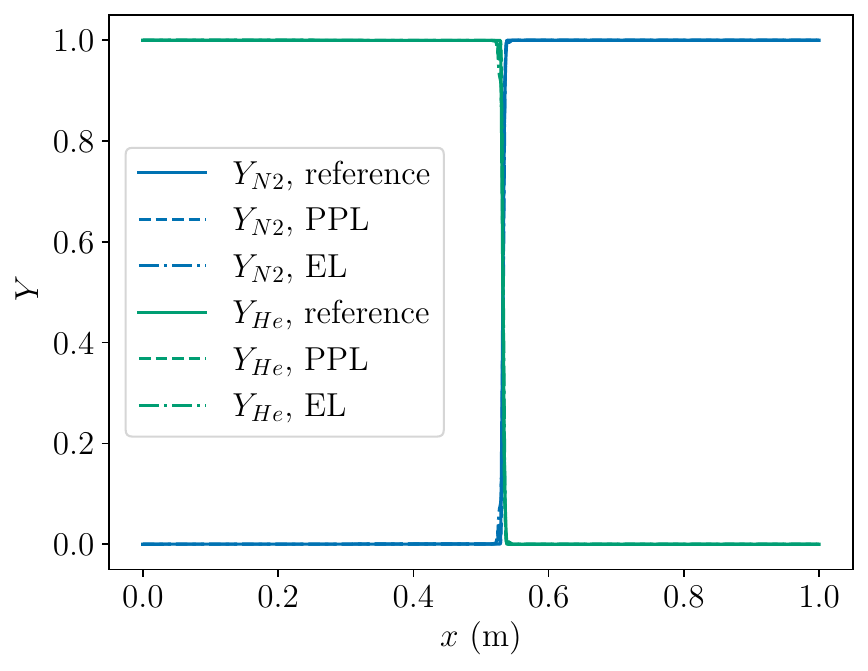}}\hfill{}\subfloat[\label{fig:shock_tube_P_LLF}Pressure.]{\includegraphics[width=0.45\columnwidth]{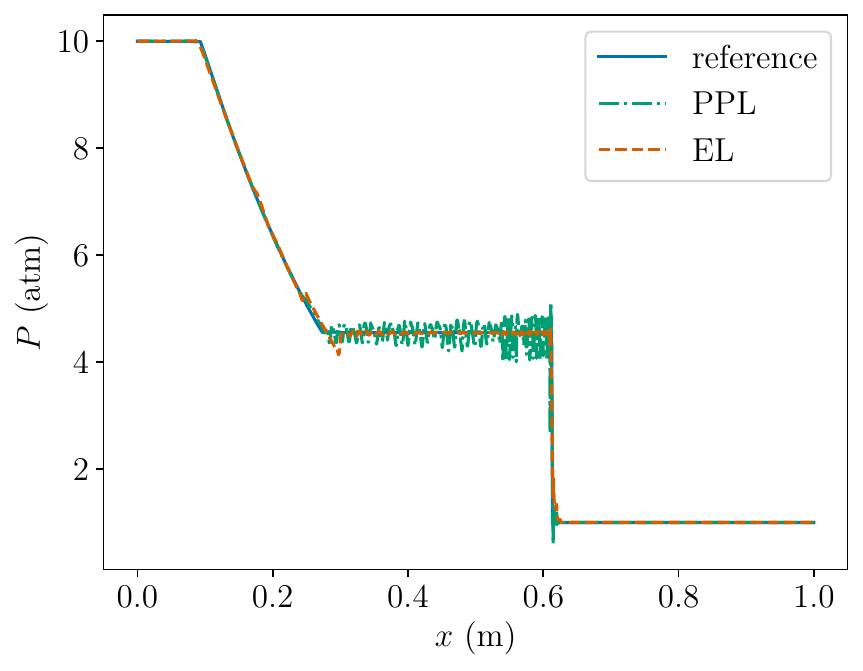}}\hfill{}\subfloat[\label{fig:shock_tube_T_LLF}Temperature.]{\includegraphics[width=0.48\columnwidth]{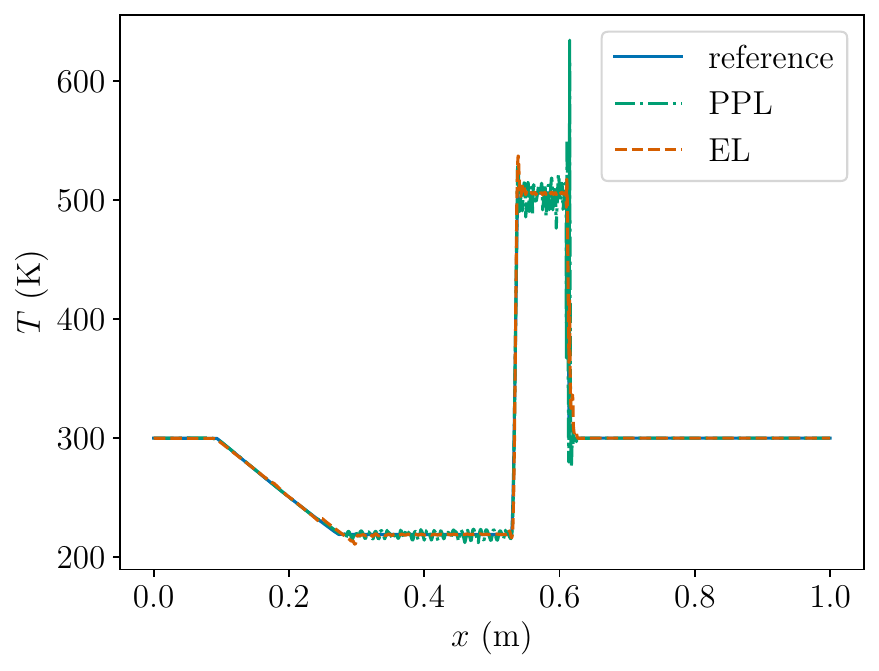}}\hfill{}\subfloat[\label{fig:shock_tube_s_LLF}Specific thermodynamic entropy.]{\includegraphics[width=0.45\columnwidth]{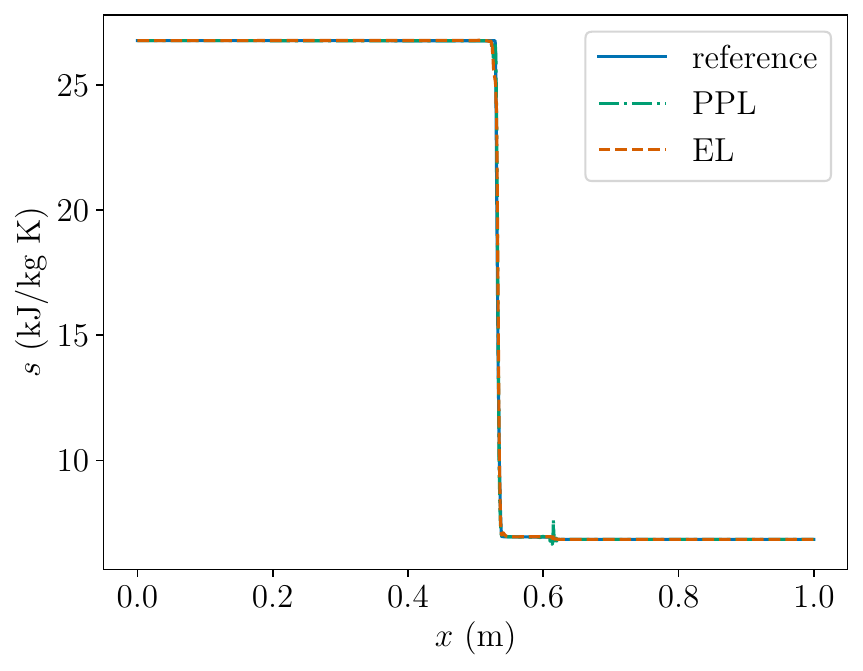}}

\caption{\label{fig:shock_tube_LLF}Results for $p=3$ solutions computed using
LLF on 200 elements without artificial viscosity for the one-dimensional,
multicomponent shock-tube problem with initialization in Equation~(\ref{eq:shock-tube-IC-Houim}).
``PPL'' corresponds to the positivity-preserving limiter by itself,
and ``EL'' refers to both the positivity-preserving and entropy
limiters with the local entropy bound in Equation~(\ref{eq:local-entropy-bound}).}
\end{figure}

\subsection{Two-dimensional detonation wave~\label{subsec:detonation-results}}

This test case involves a moving hydrogen-oxygen detonation wave diluted
in Argon with initial conditions
\begin{equation}
\begin{array}{cccc}
\qquad\qquad\qquad\left(v_{1},v_{2}\right) & = & \left(0,0\right)\text{ m/s},\\
X_{Ar}:X_{H_{2}O}:X_{OH}:X_{O_{2}}:X_{H_{2}} & = & \begin{cases}
8:2:0.1:0:0\\
7:0:0:1:2
\end{cases} & \begin{array}{c}
x_{1}<0.012\text{ m},x\in\mathcal{C}_{1},x\in\mathcal{C}_{2}\\
\mathrm{otherwise}
\end{array},\\
\qquad\qquad\qquad\qquad P & = & \begin{cases}
\expnumber{5.50}5 & \text{ Pa}\\
\expnumber{6.67}3 & \text{ Pa}
\end{cases} & \begin{array}{c}
x_{1}<0.012\text{ m},x\in\mathcal{C}_{1},x\in\mathcal{C}_{2}\\
\mathrm{otherwise}
\end{array},\\
\qquad\qquad\qquad\qquad T & = & \begin{cases}
3500 & \text{ K}\\
300\text{\hspace{1em}\hspace{1em}} & \text{ K}
\end{cases} & \begin{array}{c}
x_{1}<0.012\text{ m},x\in\mathcal{C}_{1},x\in\mathcal{C}_{2}\\
\mathrm{otherwise}
\end{array},
\end{array}\label{eq:2D-detonation-initialization}
\end{equation}
where
\begin{align*}
\mathcal{C}_{1} & =\left\{ x\left|\sqrt{\left(x_{1}-0.019\right)^{2}+\left(x_{2}-0.015\right)^{2}}<0.0035\text{ m}\right.\right\} ,\\
\mathcal{C}_{2} & =\left\{ x\left|\sqrt{\left(x_{1}-0.020\right)^{2}+\left(x_{2}-0.044\right)^{2}}<0.0035\text{ m}\right.\right\} ,
\end{align*}
which represent two high-pressure/high-temperature regions to perturb
the flow. The computational domain is $\text{\ensuremath{\Omega}}=\left(0,0.45\right)\mathrm{m}\times\left(0,0.06\right)\mathrm{m}$,
with adiabatic, no-slip walls at the left, right, bottom, and top
boundaries. The chemical mechanism is based on the Westbrook mechanism
and can be found in~\citep[Appendix D]{Chi22_2}.

Johnson and Kercher computed this flow without viscous effects with
$p=1$ and a very fine mesh with spacing $h=9\times10^{-5}$ m~\citep{Joh20_2}.
In~\citep{Chi22_2}, we simulated this flow (also without viscous
effects) using a series of triangular grids ranging from very coarse
to fine. Stability was maintained across all resolutions. The finer
cases predicted the correct diamond-like cellular structure, with
a cell length of $0.055\text{ m}$ and a cell height of $0.03\text{ m}$~\citep{Ora98,Lef98}.
In particular, there were two cells in the vertical direction. Here,
we recompute this flow with viscous effects and quadrilateral elements.
Specifically, we use Gmsh~\citep{Geu09} to first generate structured-type,
uniform grids with element sizes of $2h$, $8h$, and $64h$; the
cells are then clustered near the top and bottom walls, resulting
in smaller mesh spacing in the vertical direction at said walls. Since
the grids do not directly account for the circular perturbations in
Equation~(\ref{eq:2D-detonation-initialization}), the discontinuities
in the initial conditions are slightly smoothed using hyperbolic tangent
functions. For the remainder of this subsection, ``BR2'' refers
to the adaptive time stepping strategy exactly as described in Section~\ref{subsec:Adaptive-time-stepping},
whereas ``LLF'' refers to a similar time stepping strategy, but
with the viscous flux function fixed to be the local Lax-Friedrichs-type
flux function. SSPRK3 time integration is employed with a default
CFL of 0.4.

Figure~\ref{fig:2D-detonation-LLF} presents the distributions of
OH mole fraction and temperature obtained from $p=2$ solutions at
$t=200$~$\mu\mathrm{s}$ computed with LLF. Unsurprisingly, the
$64h$ solution is extremely smeared behind the shock. A large, nonphysical
temperature undershoot is observed near the top and bottom walls.
To more clearly illustrate this undershoot, Figure~\ref{fig:2D-detonation-T-zoom}
(top) zooms in on the temperature field at the bottom wall. The flow
is much better resolved in the $8h$ case according to Figure~\ref{fig:2D-detonation-LLF}.
The near-wall instabilities largely disappear, but spurious oscillations
can be observed in the mole-fraction field. Figure~\ref{fig:2D-detonation-BR2}
displays the corresponding distributions of OH mole fraction and temperature
obtained with BR2, along with a $2h$ solution. Figure~\ref{fig:2D-detonation-T-zoom}
(bottom) gives the near-wall temperature distribution for $64h$.
These $64h$ and $8h$ solutions are similar to the LLF solutions,
but with much smaller nonphysical instabilities. The detonation-front
locations are fairly close across all cases. In the $2h$ solution,
the flow topology, including transverse waves, vortices, and triple
points, is well-captured.

The gradient limiting procedure described in Section~\ref{subsec:zero-species-concentrations}
is applied at the first time step of both LLF simulations, as well
as the second time step of only the $8h$ LLF simulation. Without
it, at the first step of both simulations, $\Delta t$ would need
to be halved approximately thirty times, resulting in a time-step
size close to machine precision, for the solver to proceed. In contrast,
gradient limiting is not applied in the BR2 simulations. This observation,
combined with the smaller oscillations in the BR2 solutions, is the
reason why the BR2 flux function is chosen to be the ``default''
flux function in the adaptive time stepping strategy proposed in Section~\ref{subsec:Adaptive-time-stepping}.

\begin{figure}[H]
\subfloat[\label{fig:2D-detonation-X-OH-LLF}OH mole fraction]{\includegraphics[width=0.96\columnwidth]{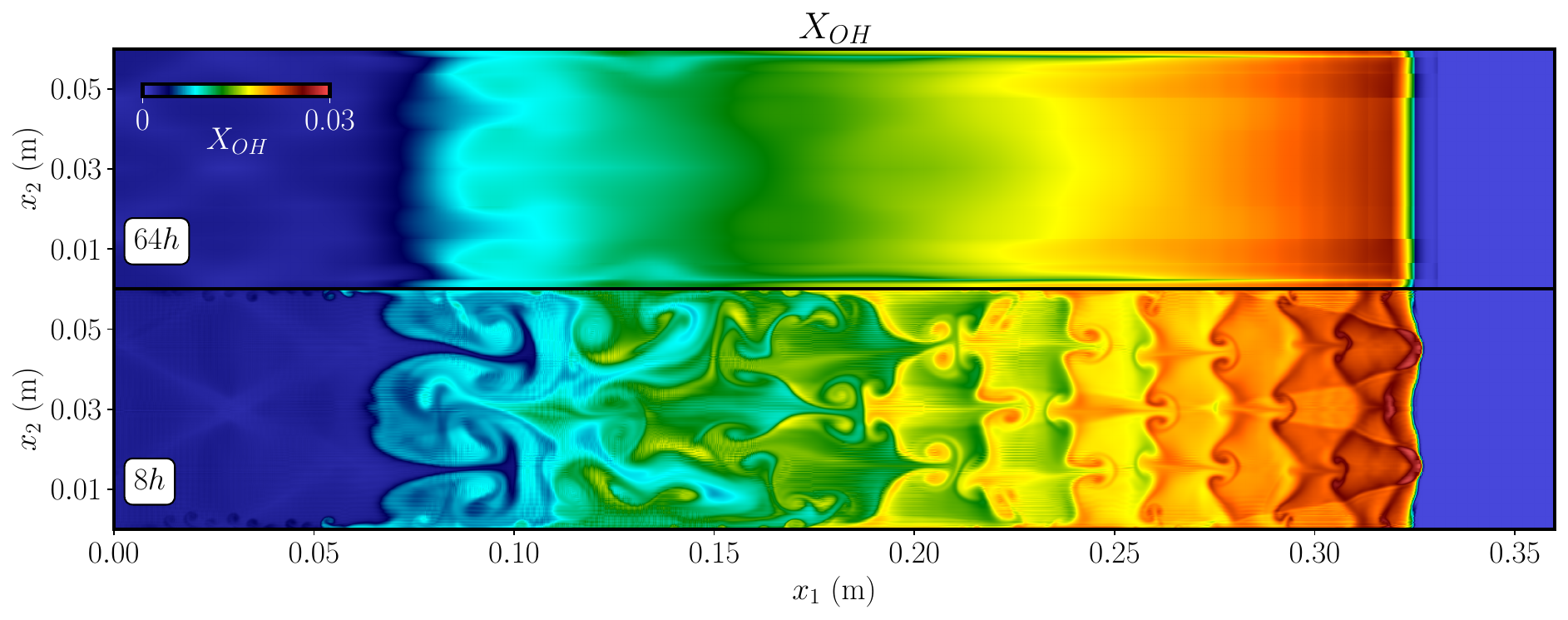}}\hfill{}\subfloat[\label{fig:2D-detonation-T-LLF}Temperature]{\includegraphics[width=0.96\columnwidth]{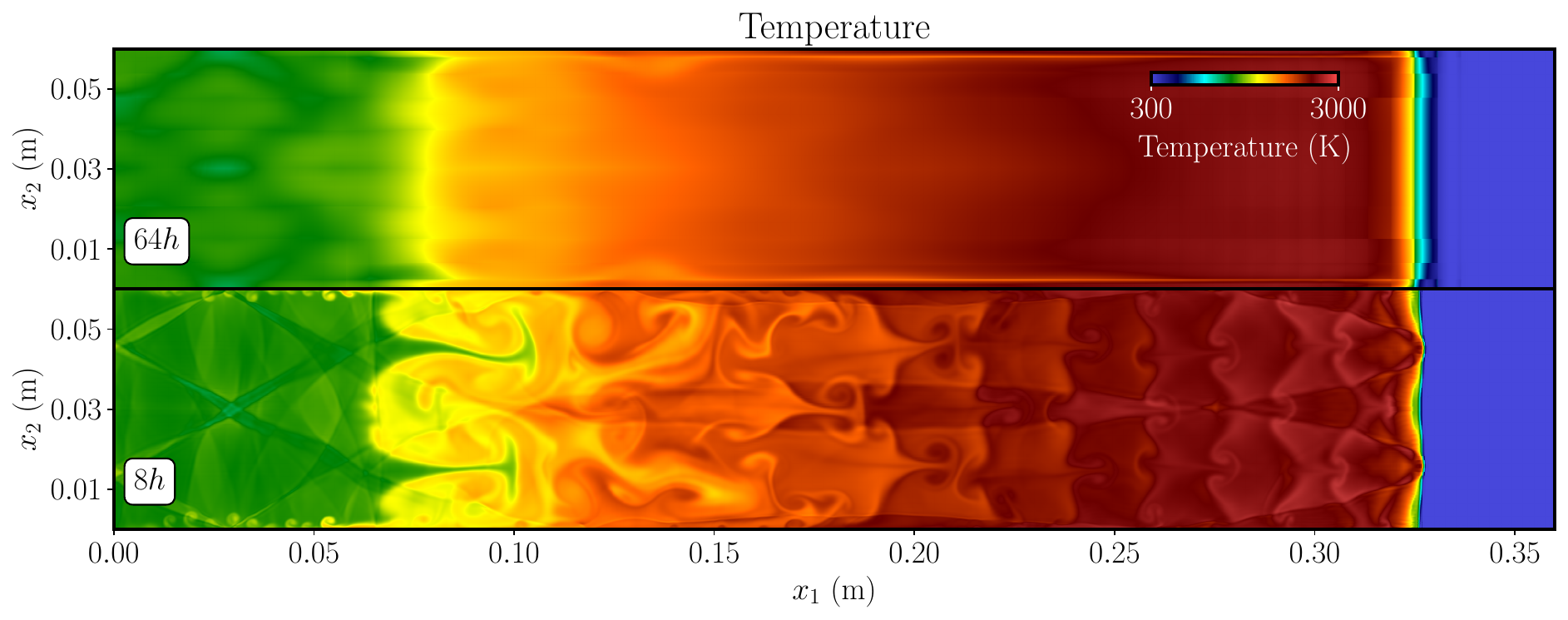}}

\caption{\label{fig:2D-detonation-LLF}$p=2$ solution to a two-dimensional
moving detonation wave at $t=200$~$\mu\mathrm{s}$ computed with
LLF. The initial conditions are given in Equation~(\ref{eq:2D-detonation-initialization}).}
\end{figure}

\begin{figure}[H]
\begin{centering}
\includegraphics[width=0.96\columnwidth]{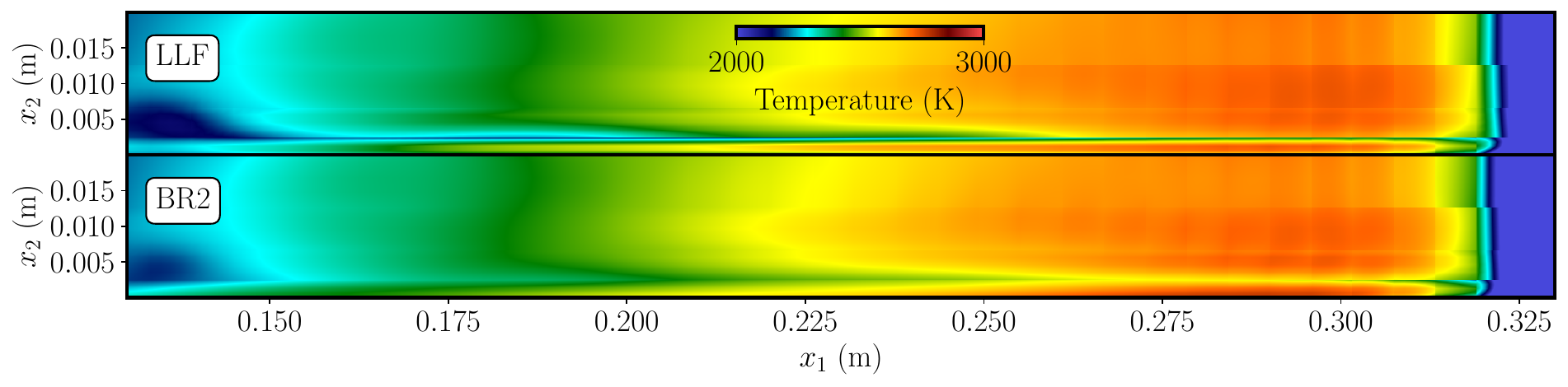}
\par\end{centering}
\caption{\label{fig:2D-detonation-T-zoom}Temperature distributions obtained
with the $64h$ mesh zoomed in on the bottom wall.}
\end{figure}
\begin{figure}[H]
\subfloat[\label{fig:2D-detonation-X-OH-BR2}OH mole fraction]{\includegraphics[width=0.96\columnwidth]{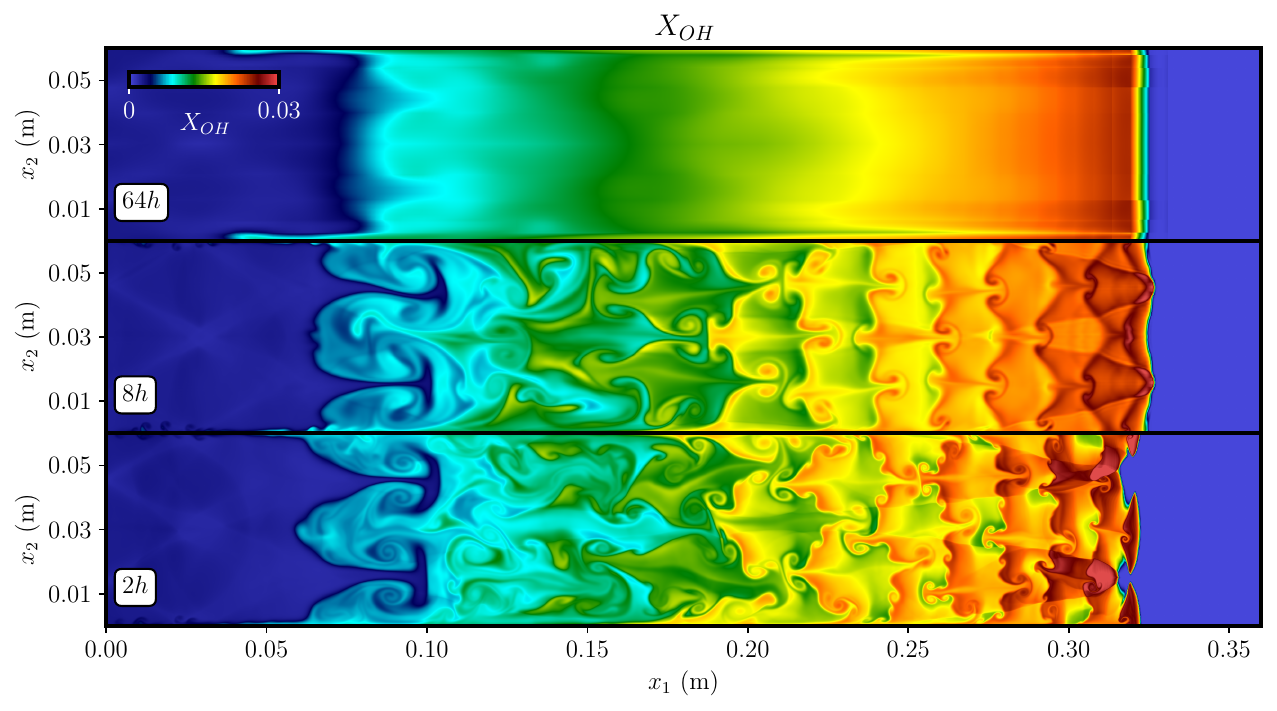}}\hfill{}\subfloat[\label{fig:2D-detonation-T-BR2}Temperature]{\includegraphics[width=0.96\columnwidth]{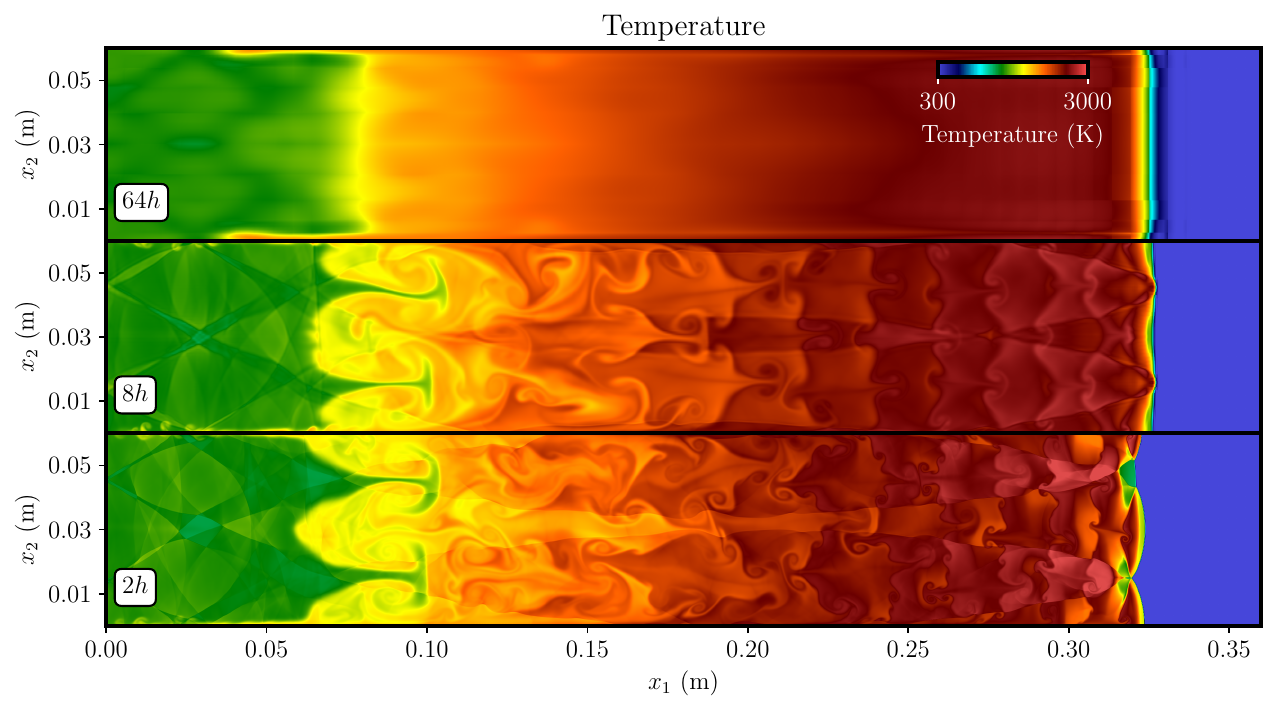}}

\caption{\label{fig:2D-detonation-BR2}$p=2$ solution to a two-dimensional
moving detonation wave at $t=200$~$\mu\mathrm{s}$ computed with
BR2. The initial conditions are given in Equation~(\ref{eq:2D-detonation-initialization}).}
\end{figure}

Figure~\ref{fig:2D-detonation-pmax} presents the maximum-pressure
history, $P^{*}$, where $P^{*,j+1}(x)=\max\left\{ P^{j+1}(x),P^{*,j}(x)\right\} $,
for the $p=2$, BR2 solutions. No cellular structure can be discerned
in the $64h$ case due to the extremely coarse mesh. The cells in
the $8h$ solution can be clearly identified, but begin to slightly
dissipate towards the right of the domain. Finally, those in the $2h$
solution remain sharp throughout.

\begin{figure}[H]
\begin{centering}
\includegraphics[width=0.96\columnwidth]{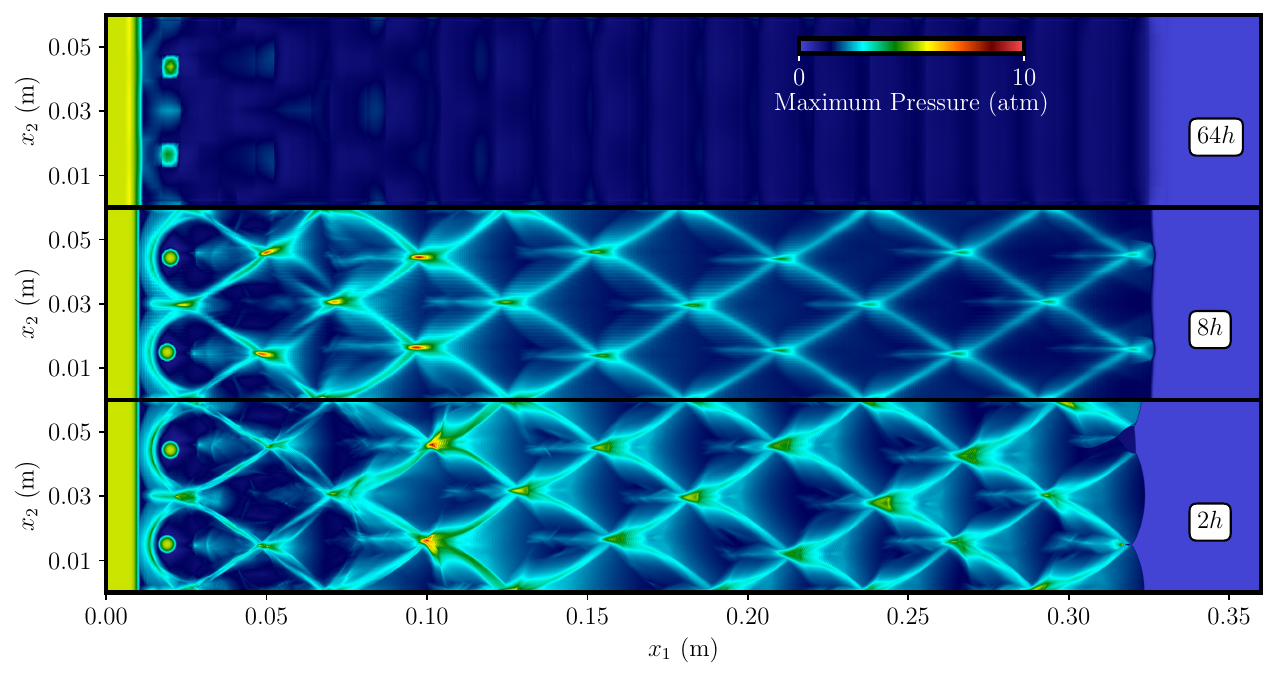}
\par\end{centering}
\caption{\label{fig:2D-detonation-pmax}Maximum-pressure history,  $P^{*}$,
where $P^{*,j+1}(x)=\max\left\{ P^{j+1}(x),P^{*,j}(x)\right\} $,
for a two-dimensional moving detonation wave at $t=200$~$\mu\mathrm{s}$
computed with $p=2$, BR2, and a sequence of meshes, where $h=9\times10^{-5}$
m. The initial conditions are given in Equation~(\ref{eq:2D-detonation-initialization}).}
\end{figure}

Finally, we recompute the BR2, $64h$ case with curved elements of
quadratic geometric order. Specifically, high-order geometric nodes
are first inserted into the straight-sided mesh, after which the midpoint
nodes at interior interfaces are perturbed. These perturbations are
performed only for $x>0.05\text{ m}$ to ensure the initial conditions
are the same. This low-resolution case is computed in order to guarantee
that the limiter is frequently activated. Figure~\ref{fig:2D-detonation-X-OH-32h-linear-curved}
displays the distributions of OH mole fraction for the linear and
curved meshes, which are superimposed. The solution obtained with
curved mesh is stable and extremely similar to that computed with
the linear mesh, demonstrating that the proposed formulation is indeed
compatible with curved elements.

\begin{figure}[H]
\begin{centering}
\includegraphics[width=0.96\columnwidth]{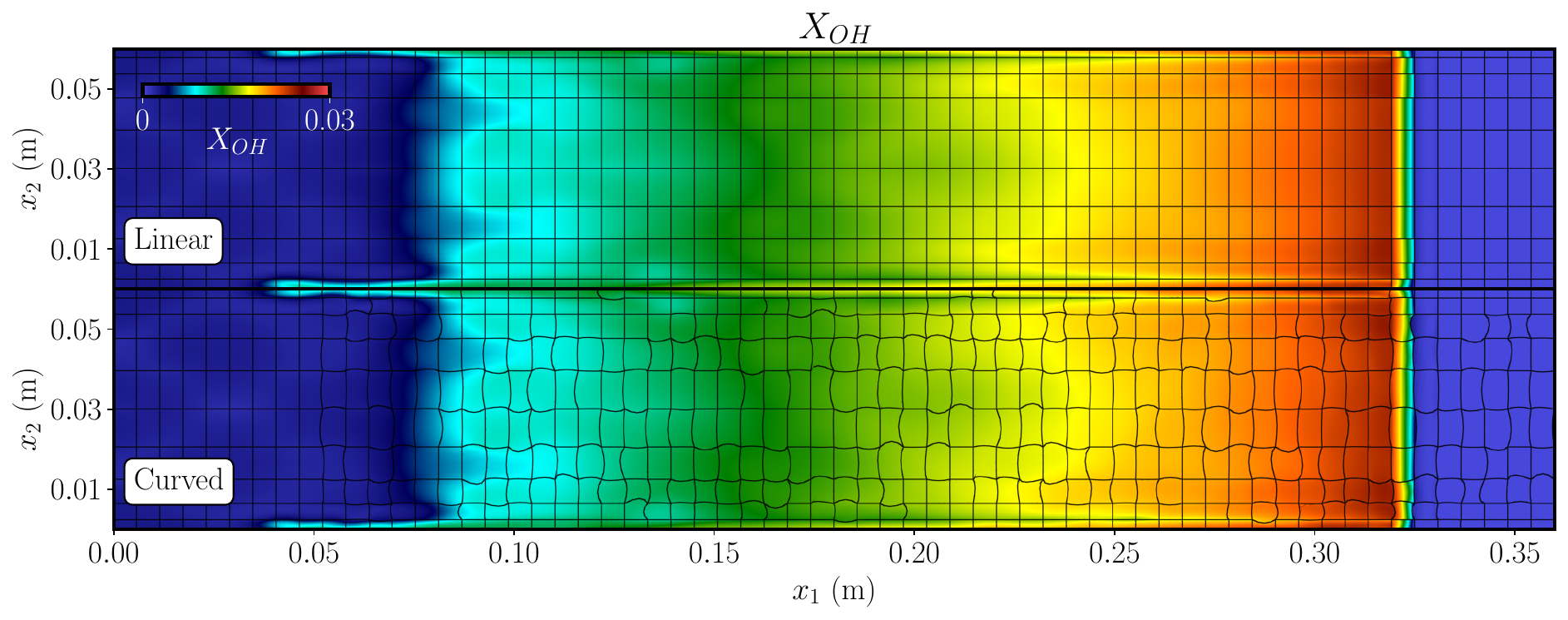}
\par\end{centering}
\caption{\label{fig:2D-detonation-X-OH-32h-linear-curved}OH mole-fraction
field for a two-dimensional moving detonation wave at $t=200$~$\mu\mathrm{s}$
computed with $p=2$, BR2, and $64h$, where $h=9\times10^{-5}$ m,
on linear and curved meshes. The curved mesh, which is of quadratic
order, is obtained by inserting high-order geometric nodes into the
linear mesh and perturbing said nodes. The initial conditions are
given in Equation~(\ref{eq:2D-detonation-initialization}).}
\end{figure}

\subsection{Three-dimensional shock/mixing-layer interaction}

In this section, we compute a three-dimensional chemically reacting
mixing layer that intersects with an oblique shock. This test case
was first presented in~\citep{Bur22}, which built on the configuration
introduced in~\citep{Fer14}. The mesh and flow parameters are slightly
different from those in~\citep{Bur22}.

Figure~\ref{fig:shear_layer_diagram} displays a two-dimensional
schematic of the flow configuration. Supersonic inflow is applied
at the left boundary, and extrapolation is applied at the right boundary.
Flow parameters for the incoming air and fuel are listed in Table~\ref{tab:Air_and_Fuel}.
Slip-wall conditions are applied the top and bottom walls since it
is not necessary to capture the boundary layers. We use the detailed
reaction mechanism described in~\citep[Appendix D]{Chi22_2}. A default
CFL of 0.5 with SSPRK3 time integration is employed.

\begin{figure}[H]
\begin{centering}
\includegraphics[width=0.8\columnwidth]{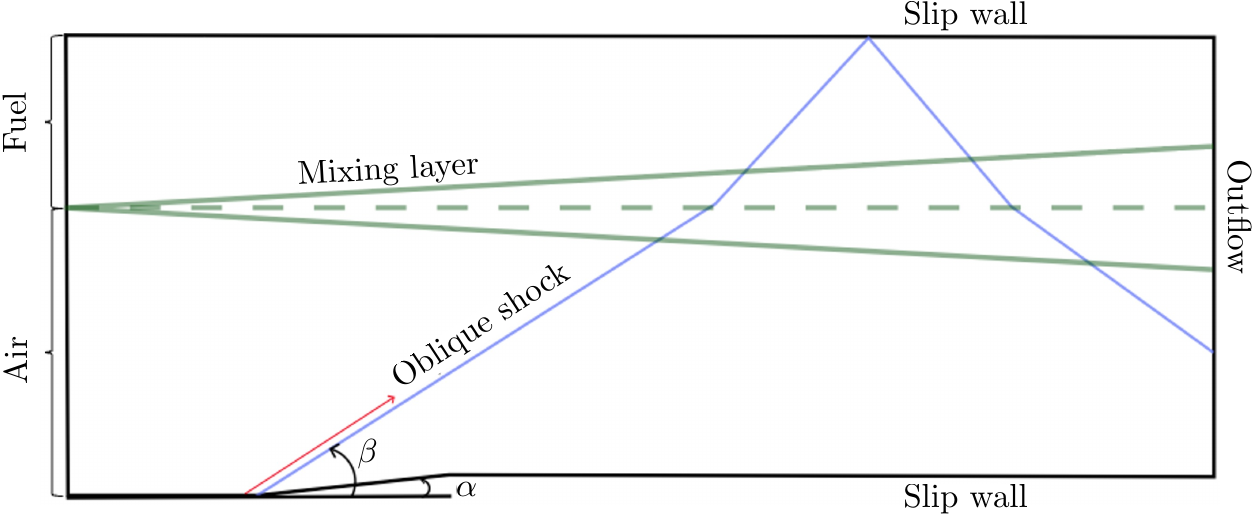}
\par\end{centering}
\caption{\label{fig:shear_layer_diagram}Schematic of the three-dimensional
shock/mixing-layer interaction test case. }
\end{figure}

\begin{table}[H]
\centering{}%
\begin{tabular}{|c|c|c|}
\hline
\textbf{Prescribed Quantity} & \textbf{Air Boundary} & \textbf{Fuel Boundary}\tabularnewline
\hline
\hline
Velocity, $v_{1}$ (m/s) & 1634 m/s & 973 m/s\tabularnewline
\hline
Temperature, $T$ (K) & 1475 & 545\tabularnewline
\hline
$Y_{O_{2}}$ & 0.278 & 0\tabularnewline
\hline
$Y_{N_{2}}$ & 0.552 & 0.95\tabularnewline
\hline
$Y_{H_{2}}$ & 0 & 0.05\tabularnewline
\hline
$Y_{H_{2}O}$ & 0.17 & 0\tabularnewline
\hline
$Y_{H}$ & $5.6\times10^{-7}$ & 0\tabularnewline
\hline
$Y_{O}$ & $1.55\times10^{-4}$ & 0\tabularnewline
\hline
$Y_{OH}$ & $1.83\times10^{-3}$ & 0\tabularnewline
\hline
$Y_{HO_{2}}$ & $5.1\times10^{-6}$ & 0\tabularnewline
\hline
$Y_{H_{2}O_{2}}$ & $2.5\times10^{-7}$ & 0\tabularnewline
\hline
\end{tabular}\caption{\label{tab:Air_and_Fuel}Inflow parameters for three-dimensional shock/mixing-layer
interaction. These values are taken from~\citep{Fer14}.}
\end{table}

To connect the fuel and air streams, we utilize a hyperbolic tangent
function for prescribing the species, temperature, and velocity with
a constant pressure specification,

\begin{eqnarray}
Y_{i}(x_{2},x_{3},t) & = & \frac{1}{2}\left(\left(Y_{i,F}+Y_{i,O}\right)+\left(Y_{i,F}-Y_{i,O}\right)\tanh\left(\frac{\left(2\left(x_{2}-h\left(x_{3},t\right)\right)\right)}{L\left(x_{3},t\right)}\right)\right)\nonumber \\
T(x_{2},x_{3},t) & = & \frac{1}{2}\left(\left(T_{F}+Y_{O}\right)+\left(T_{F}-T_{O}\right)\tanh\left(\frac{\left(2\left(x_{2}-h\left(x_{3},t\right)\right)\right)}{L\left(x_{3},t\right)}\right)\right)\nonumber \\
v_{1}(x_{2},x_{3},t) & = & \frac{1}{2}\left(\left(v_{1,F}+v_{1,O}\right)+\left(v_{1,F}-v_{1,O}\right)\tanh\left(\frac{\left(2\left(x_{2}-h\left(x_{3},t\right)\right)\right)}{L\left(x_{3},t\right)}\right)\right)\label{eq:shear_layer_boundary_condition}\\
P & = & 94232.25\text{ Pa},\nonumber
\end{eqnarray}
where $(\cdot)_{O}$ denotes air, $(\cdot)_{F}$ denotes fuel, $L$
is a length scale, and $h$ is the center of the hyperbolic tangent.
Equation~\ref{eq:shear_layer_boundary_condition} is also used to
initialize the solution. $L$ is given by

\begin{eqnarray}
L(x_{3},t) & = & L_{s}+\sum_{i=1}^{n_{t}}A_{i}\sin\left(\frac{m_{i}2\pi t}{t_{r}}\right)+l(x_{3},t)\nonumber \\
l\left(x_{3},t\right) & = & B\sin\left(\frac{24\pi t}{t_{r}}\right)\left[\sin\left(\frac{8\pi x_{3}}{z_{h}}\right)+\sin\left(\frac{32\pi x_{3}}{z_{h}}\right)\right]\text{,}\label{eq:length_scale}
\end{eqnarray}
where $L_{s}=0.05$ mm is a reference length scale, $z_{h}=1.44$
mm is the domain thickness in the $x_{3}$-direction, and $t_{r}=3.635\times10^{-5}$
s is a reference time scale. Furthermore, $n_{t}=4$, $\left(A_{1},A_{2},A_{3},A_{4}\right)=\left(0.0025,0.0125,0.00125,0.00625\right)$
mm, $\left(m_{1},m_{2},m_{3},m_{4}\right)=\left(1,3,11,13\right)$,
and $B=0.0125$ mm.  To induce additional variation in this three-dimensional
case, we prescribe $h$ as

\begin{eqnarray}
h(x_{3},t) & = & h_{s}+10B\sin\left(\frac{16\pi x_{3}}{z_{h}}\right),\label{eq:third_dimension_perturbation}
\end{eqnarray}
where $h_{s}=8.64$ mm is the ambient center of the hyperbolic tangent.

\begin{figure}[H]
\begin{centering}
\includegraphics[width=0.96\columnwidth]{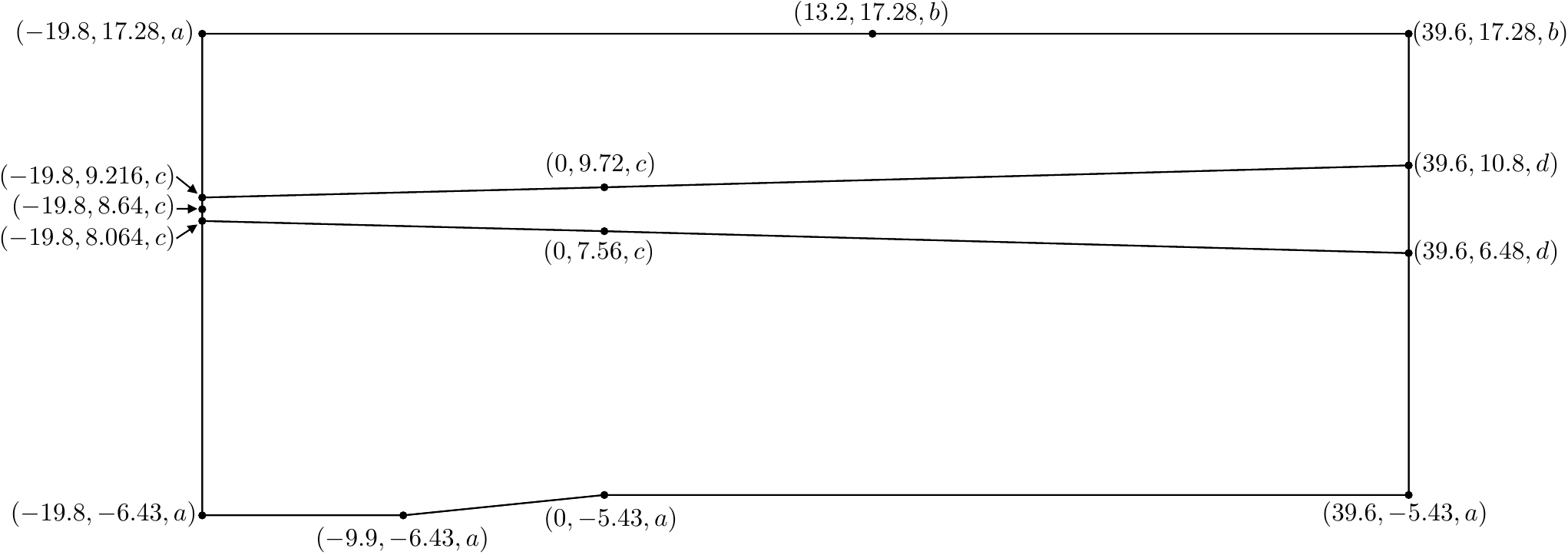}
\par\end{centering}
\caption{\label{fig:mesh_specs}Diagram of geometry with $x_{1}x_{2}$-point
locations (in mm) for mesh generation in Gmsh~\citep{Geu09}. The
third value in each tuple ($a$ through $d$) is the target mesh size
at the respective location. }
\end{figure}

Figure~\ref{fig:mesh_specs} shows the specifications in Gmsh~\citep{Geu09}
used to create an unstructured tetrahedral mesh. For each tuple, the
first two values are the $x_{1}$ and $x_{2}$ locations of the given
point. The third value represents the target mesh size. We select
$a=0.5$ mm, $b=0.2$ mm , $c=0.06$ mm, and $d=0.12$ mm. The mesh
is extruded in the $x_{3}$-direction from $x_{3}=0$ to $x_{3}=z_{h}$,
with periodicity applied at the resulting $x_{1}x_{2}$-planes.

Figure~\ref{fig:3D-splitter-plate} shows instantaneous isosurfaces
corresponding to $Y_{OH}=0.00017$, superimposed on a numerical Schlieren
result sampled along an $x_{1}x_{2}$-plane. The $Y_{OH}$ isosurfaces
ares colored by pressure to highlight the abrupt compression experienced
through the oblique shock. The right image provides a zoomed-in perspective
to emphasize the three dimensional flow features. Roll-up is observed
upstream of the oblique shock. The interaction between the shock and
the mixing layer causes the generation of smaller-scale compression
waves. These results demonstrate that the proposed formulation can
capture complex flow features in three dimensions.

\begin{figure}[H]
\begin{centering}
\includegraphics[width=0.96\columnwidth]{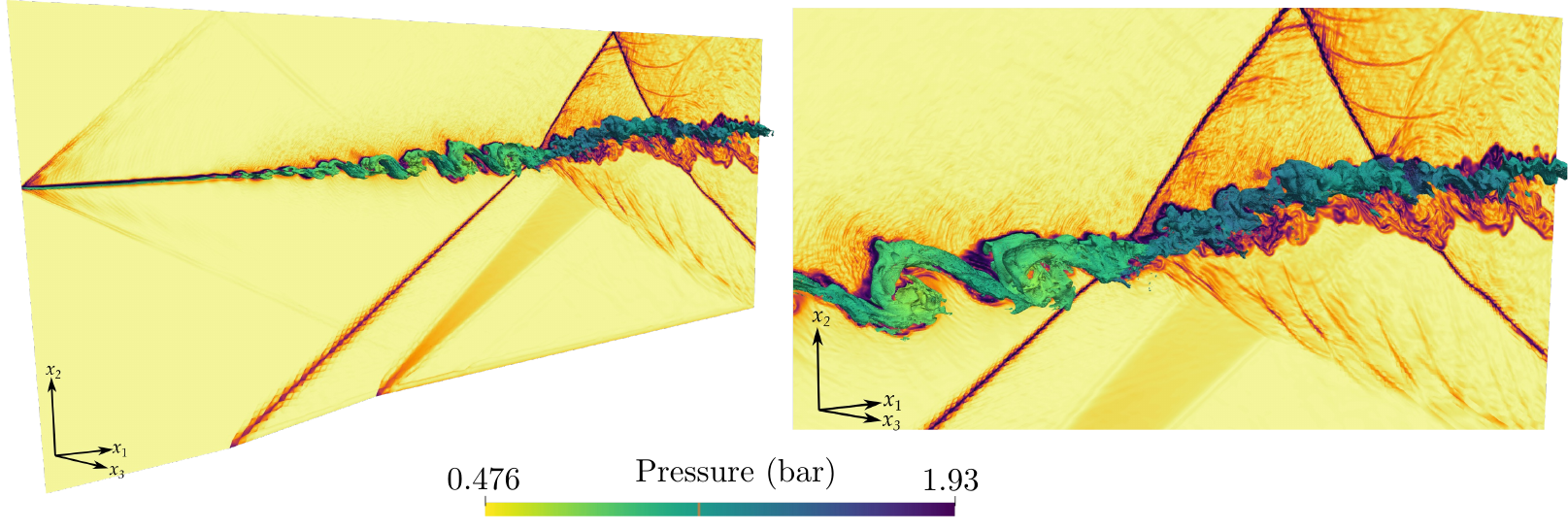}
\par\end{centering}
\caption{\label{fig:3D-splitter-plate}$Y_{OH}$ and numerical Schlieren results
for three-dimensional shock/mixing-layer interaction.}
\end{figure}

\section{Concluding remarks}

In this paper, we developed a fully conservative, positivity-preserving,
and entropy-bounded DG formulation for the chemically reacting, compressible
Navier-Stokes equations. The formulation builds on the fully conservative,
positivity-preserving, and entropy-bounded DG formulation for the
chemically reacting, compressible Euler equations that we previously
introduced~\citep{Chi22,Chi22_2}. A key ingredient is the positivity-preserving
Lax-Friedrichs-type viscous flux function devised by Zhang~\citep{Zha17}
for the monocomponent case, which we extended to multicomponent flows
with species diffusion in a manner that separates the inviscid and
viscous fluxes. This is in contrast with the work by Du and Yang~\citep{Du22},
who similarly extended the Lax-Friedrichs-type viscous flux function,
but treated the inviscid and viscous fluxes together. We discussed
in detail the consideration of boundary conditions and the techniques
by Johnson and Kercher~\citep{Joh20_2} to reduce spurious pressure
oscillations, introducing additional constraints on the time step
size. Entropy boundedness is enforced on only the convective contribution
since the minimum entropy principle only applies to the Euler equations~\citep{Tad86,Gue14}
and the viscous flux function is not fully compatible with the entropy
bound. Drawing from~\citep{Zha17}, we proposed an adaptive solution
procedure that favors large time step sizes and the BR2 viscous flux
function since the Lax-Friedrichs-type viscous flux function was found
to more likely lead to spurious oscillations. Small time-step sizes
and/or the Lax-Friedrichs-type viscous flux function are employed
only when necessary. However, it should be noted that the Lax-Friedrichs-type
viscous flux function guarantees a finite time-step size such that
the positivity property is maintained. Furthermore, we discussed potential
issues with the Lax-Friedrichs-type viscous flux function in the case
of zero (or near-zero) species concentrations and introduced a gradient-limiting
procedure as a remedy. The proposed methodology is compatible with
high-order polynomials and curved elements of arbitrary shape.

The DG methodology was applied to a series of test cases. The first
two comprised smooth, one-dimensional flows: advection-diffusion of
a thermal bubble and a premixed flame. In the former, optimal convergence
was demonstrated for both viscous flux functions. In the latter, we
obtained a much more accurate solution on a relatively coarse mesh
with the proposed methodology than with conventional species clipping.
Next, we computed viscous shock-tube flow and found that just as in
the inviscid setting, enforcement of entropy boundedness considerably
reduces the magnitude of large-scale instabilities that otherwise
appear if only the positivity property is enforced. Finally, we computed
two-dimensional, moving, viscous detonation waves and three-dimensional
shock/mixing-layer interaction, demonstrating that the proposed formulation
can accurately and robustly compute complex reacting flows with detailed
chemistry using high-order polynomial approximations. Future work
will entail the simulation of larger-scale viscous, chemically reacting
flows involving more complex geometries.

\section*{Acknowledgments}

This work is sponsored by the Office of Naval Research through the
Naval Research Laboratory 6.1 Computational Physics Task Area. 

\bibliographystyle{elsarticle-num}
\bibliography{citations}

\appendix

\section{Boundary conditions\label{sec:boundary-conditions}}

The boundary conditions here were originally described by Johnson
and Kercher~\citep{Joh20_2}, building on the discussion in~\citep{Har13}.

\subsection{Supersonic inflow}

Letting $y_{\infty}$ denote the fully prescribed state, we specify

\begin{align*}
y_{\partial}\left(y^{+},n^{+}\right)=y_{\infty} & \textup{ on }\epsilon\qquad\forall\epsilon\in\mathcal{E}_{\mathrm{in}},\\
\mathcal{F}_{\partial,k}^{\nu}\left(y_{\partial}\left(y^{+},n^{+}\right),\nabla y^{+}\right)=\mathcal{F}_{k}^{\nu}\left(y_{\infty},\nabla y^{+}\right) & \textup{ on }\epsilon\qquad\forall\epsilon\in\mathcal{E}_{\mathrm{in}},\\
\mathcal{F}_{\partial}^{c\dagger}\left(y^{+},n^{+}\right)=\mathcal{F}^{c}\left(y_{\infty}\right)\cdot n^{+} & \textup{ on }\epsilon\qquad\forall\epsilon\in\mathcal{E}_{\mathrm{in}}.
\end{align*}

\subsection{Supersonic outflow}

This boundary condition is given by

\begin{align*}
y_{\partial}\left(y^{+},n^{+}\right)=y^{+} & \textup{ on }\epsilon\qquad\forall\epsilon\in\mathcal{E}_{\mathrm{out}},\\
\mathcal{F}_{\partial,k}^{\nu}\left(y_{\partial}\left(y^{+},n^{+}\right),\nabla y^{+}\right)=\mathcal{F}_{k}^{\nu}\left(y^{+},\nabla y^{+}\right) & \textup{ on }\epsilon\qquad\forall\epsilon\in\mathcal{E}_{\mathrm{out}},\\
\mathcal{F}_{\partial}^{c\dagger}\left(y^{+},n^{+}\right)=\mathcal{F}^{c}\left(y^{+}\right)\cdot n^{+} & \textup{ on }\epsilon\qquad\forall\epsilon\in\mathcal{E}_{\mathrm{out}},
\end{align*}
where the boundary state is simply extrapolated from the interior.

\subsection{Slip wall}

We define the boundary velocity, $v_{\partial}=\left(v_{\partial,1},\ldots,v_{\partial,d}\right)$,
as
\[
v_{\partial}\left(y^{+},n^{+}\right)=\left(v_{1}^{+}-\left(\sum_{k=1}^{d}v_{k}^{+}n_{k}^{+}\right)n_{1}^{+},\cdots,v_{d}^{+}-\left(\sum_{k=1}^{d}v_{k}^{+}n_{k}^{+}\right)n_{d}^{+}\right).
\]
We then have

\begin{align*}
y_{\partial}\left(y^{+},n^{+}\right)=\left(\rho^{+}v_{\partial,1},\ldots,\rho^{+}v_{\partial,d},\left(\rho e_{t}\right)^{+},C_{i}^{+},\ldots,C_{n_{s}}^{+}\right)^{T} & \textup{ on }\epsilon\qquad\forall\epsilon\in\mathcal{E}_{\mathrm{slip}},\\
\mathcal{F}_{\partial,k}^{\nu}\left(y_{\partial}\left(y^{+},n^{+}\right),\nabla y^{+}\right)=\mathcal{F}_{k}^{\nu}\left(y_{\partial}\left(y^{+},n^{+}\right),\nabla y^{+}\right) & \textup{ on }\epsilon\qquad\forall\epsilon\in\mathcal{E}_{\mathrm{slip}},\\
\mathcal{F}_{\partial}^{c\dagger}\left(y^{+},n^{+}\right)=\mathcal{F}^{c\dagger}\left(y^{+},2y_{\partial}\left(y^{+},n^{+}\right)-y^{+},n^{+}\right) & \textup{ on }\epsilon\qquad\forall\epsilon\in\mathcal{E}_{\mathrm{slip}}.
\end{align*}
The viscous boundary flux is obtained from the boundary state and
the interior gradient. 

\subsection{Adiabatic wall}

Here, let $v_{\partial}$ be a prescribed boundary velocity. This
boundary condition is then given by

\begin{align*}
y_{\partial}\left(y^{+},n^{+}\right)=\left(\rho^{+}v_{\partial,1},\ldots,\rho^{+}v_{\partial,d},\left(\rho e_{t}\right)^{+},C_{i}^{+},\ldots,C_{n_{s}}^{+}\right)^{T} & \textup{ on }\epsilon\qquad\forall\epsilon\in\mathcal{E}_{\mathrm{adi}},\\
\mathcal{F}_{\partial,k}^{\nu}\left(y_{\partial}\left(y^{+},n^{+}\right),\nabla y^{+}\right)=\begin{pmatrix}\frac{\partial\tau_{1k}}{\partial\left(\nabla y\right)}\left(y_{\partial}\right):\nabla y^{+}\\
\vdots\\
\frac{\partial\tau_{dk}}{\partial\left(\nabla y\right)}\left(y_{\partial}\right):\nabla y^{+}\\
{\textstyle \sum_{j=1}^{d}}v_{\partial,j}\frac{\partial\tau_{kj}}{\partial\left(\nabla y\right)}\left(y_{\partial}\right):\nabla y^{+}\\
0\\
\vdots\\
0
\end{pmatrix} & \textup{ on }\epsilon\qquad\forall\epsilon\in\mathcal{E}_{\mathrm{adi}},\\
\mathcal{F}_{\partial}^{c\dagger}\left(y^{+},n^{+}\right)=\mathcal{F}^{c}\left(y_{\partial}\left(y^{+},n^{+}\right)\right)\cdot n^{+} & \textup{ on }\epsilon\qquad\forall\epsilon\in\mathcal{E}_{\mathrm{adi}},
\end{align*}
where the species diffusion velocities and heat flux have been set
to zero. The two boundary conditions below are not directly used in
this study, but we include them for completeness.

\subsection{Isothermal wall}

Again, let $v_{\partial}$ be a prescribed boundary velocity. We specify

\begin{align*}
y_{\partial}\left(y^{+},n^{+}\right)=\left(\rho_{\partial}v_{\partial,1},\ldots,\rho_{\partial}v_{\partial,d},\rho_{\partial}u_{\partial}+0.5\cdot{\textstyle \sum_{j=1}^{d}}\rho_{\partial}v_{\partial,k}v_{\partial,k},C_{\partial,i},\ldots,C_{\partial,n_{s}}\right)^{T} & \textup{ on }\epsilon\qquad\forall\epsilon\in\mathcal{E}_{\mathrm{iso}},\\
\mathcal{F}_{\partial,k}^{\nu}\left(y_{\partial}\left(y^{+},n^{+}\right),\nabla y^{+}\right)=\begin{pmatrix}\frac{\partial\tau_{1k}}{\partial\left(\nabla y\right)}\left(y_{\partial}\right):\nabla y^{+}\\
\vdots\\
\frac{\partial\tau_{dk}}{\partial\left(\nabla y\right)}\left(y_{\partial}\right):\nabla y^{+}\\
\left[{\textstyle \sum_{j=1}^{d}}v_{\partial,j}\frac{\partial\tau_{kj}}{\partial\left(\nabla y\right)}\left(y_{\partial}\right)-\frac{\partial q_{k}}{\partial\left(\nabla y\right)}\left(y_{\partial}\right)\right]:\nabla y^{+}\\
0\\
\vdots\\
0
\end{pmatrix} & \textup{ on }\epsilon\qquad\forall\epsilon\in\mathcal{E}_{\mathrm{iso}},\\
\mathcal{F}_{\partial}^{c\dagger}\left(y^{+},n^{+}\right)=\mathcal{F}^{c}\left(y_{\partial}\left(y^{+},n^{+}\right)\right)\cdot n^{+} & \textup{ on }\epsilon\qquad\forall\epsilon\in\mathcal{E}_{\mathrm{iso}},
\end{align*}
where $T_{\partial}$ is the prescribed boundary temperature, $C_{\partial,i}$
is the boundary concentration of the $i$th species,
\[
C_{\partial,i}\left(y^{+}\right)=\frac{T^{+}}{T_{\partial}}C_{i}^{+}.
\]
$\rho_{\partial}$ is the boundary density, 
\[
\rho_{\partial}=\sum_{i=1}^{n_{s}}W_{i}C_{\partial,i},
\]
and $u_{\partial}$ is the boundary internal energy, which is evaluated
at $T_{\partial}$.

\subsection{Characteristic}

Let $y^{*}\left(y^{+},y_{\infty},n^{+}\right)$ be the characteristic
boundary value. Its derivation for non-reflecting inflow and outflow
boundary conditions can be found in~\citep{Joh20_2}. This boundary
condition is then given as

\begin{align*}
y_{\partial}\left(y^{+},n^{+}\right)=y^{*}\left(y^{+},y_{\infty},n^{+}\right) & \textup{ on }\epsilon\qquad\forall\epsilon\in\mathcal{E}_{\mathrm{cha}},\\
\mathcal{F}_{\partial,k}^{\nu}\left(y_{\partial}\left(y^{+},n^{+}\right),\nabla y^{+}\right)=\mathcal{F}_{k}^{\nu}\left(y_{\partial}\left(y^{+},n^{+}\right),\nabla y^{+}\right) & \textup{ on }\epsilon\qquad\forall\epsilon\in\mathcal{E}_{\mathrm{cha}},\\
\mathcal{F}_{\partial}^{c\dagger}\left(y^{+},n^{+}\right)=\mathcal{F}^{c\dagger}\left(y^{+},y_{\partial}\left(y^{+},n^{+}\right),n^{+}\right) & \textup{ on }\epsilon\qquad\forall\epsilon\in\mathcal{E}_{\mathrm{cha}}.
\end{align*}

\section{Supporting lemma~\label{sec:supporting-lemma}}

In the following, let $\Delta y$ denote the quantity
\[
\Delta y=\left(\Delta(\rho v),\Delta(\rho e_{t}),\Delta C_{1},\ldots,\Delta C_{n_{s}}\right)^{T}.
\]
The lemma below is related to that in~\citep[Appendix B]{Chi22_2}.
\begin{lem}
\label{lem:alpha-constraints}Assume that $y=\left(\rho v,\rho e_{t},C\right)^{T}$
is in $\mathcal{G}$ and that $C_{i}>0,\:\forall i$. Then $\check{y}=y-\alpha^{-1}\Delta y$,
where $\alpha>0$, is also in $\mathcal{G}$ under the following conditions:

\begin{equation}
\alpha>\alpha^{*}\left(y,\Delta y\right)=\left.\max\left\{ \max_{i=1,\ldots,n_{s}}\frac{\Delta C_{i}}{C_{i}},\alpha_{T},0\right\} \right|_{\left(y,\Delta y\right)},\label{eq:alpha-constraint}
\end{equation}
where
\begin{equation}
\alpha_{T}=\begin{cases}
\frac{-\mathsf{b}+\sqrt{\mathsf{b}^{2}-4\rho^{2}u\mathsf{g}}}{2\rho^{2}u}, & \mathsf{b}^{2}-4\rho^{2}u\mathsf{g}\geq0\\
0, & \mathrm{otherwise}
\end{cases},\label{eq:alpha_T}
\end{equation}
$\mathsf{b}=-\rho e_{t}\mathsf{M}-\rho\Delta(\rho e_{t})+\rho v\cdot\Delta(\rho v)+2\rho u_{0}\mathsf{M}$,
$\mathsf{g}=\mathsf{M}\Delta(\rho e_{t})-\frac{1}{2}\left|\Delta(\rho v)\right|^{2}-u_{0}\mathsf{M}^{2}$,
and $\mathsf{M}=\sum_{i=1}^{n_{s}}W_{i}\Delta C_{i}$.
\end{lem}

\begin{proof}
The proof is similar to that for Lemma~\ref{lem:beta-constraints}.
$\check{y}=y-\alpha^{-1}\Delta y$ can be expanded as
\[
\begin{split}\check{y}= & \left(\check{\rho v},\check{\rho e_{t}},\check{C_{1}},\ldots,\check{C}_{n_{s}}\right)^{T}\\
= & \left(\rho v-\alpha\Delta(\rho v),\rho e_{t}-\alpha\Delta(\rho e_{t}),C_{1}-\alpha\Delta C_{1},\ldots,C_{n_{s}}-\alpha\Delta C_{n_{s}}\right)^{T}.
\end{split}
\]
$\check{C_{i}}=C_{i}-\alpha^{-1}\Delta C_{i}>0,\;\forall i$ under
the condition 
\[
\alpha>\max\left\{ \max_{i=1,\ldots,n_{s}}\frac{\Delta C_{i}}{C_{i}},0\right\} .
\]
$Z\left(\check{y}\right)$, where $Z$ is defined as in Equation~(\ref{eq:Z-definition}),
can be expressed as
\begin{align*}
Z\left(y-\alpha^{-1}\Delta y\right)= & \sum_{i=1}^{n_{s}}W_{i}\left(C_{i}-\alpha^{-1}\Delta C_{i}\right)\left(\rho e_{t}-\alpha^{-1}\Delta(\rho e_{t})\right)\\
 & -\frac{1}{2}\left|\rho v-\alpha^{-1}\Delta(\rho v)\right|^{2}-\left[\sum_{i=1}^{n_{s}}W_{i}\left(C_{i}-\alpha^{-1}\Delta C_{i}\right)\right]^{2}u_{0},
\end{align*}
which, after multiplying both sides by $\alpha^{2}$, can be rewritten
as 
\begin{align}
\alpha^{2}Z\left(y-\alpha^{-1}\Delta y\right)= & \rho^{2}u\alpha^{2}-\mathsf{b}\alpha+\mathsf{g}.\label{eq:alpha-quadratic-form}
\end{align}
Setting the RHS of Equation~(\ref{eq:alpha-quadratic-form}) equal
to zero yields a quadratic equation with $\alpha$ as the unknown.
Since $\rho^{2}u$ is positive, the quadratic equation is convex.
As such, if $\mathsf{b}^{2}-4\rho^{2}u\mathsf{g}<0$, then no real
roots exist, and $Z\left(y-\alpha^{-1}\Delta y\right)>0$ for all
$\alpha\neq0$; otherwise, at least one real root exists, in which
case a sufficient condition to ensure $Z\left(y-\alpha^{-1}\Delta y\right)>0$
is $\alpha>\alpha_{0}$, where $\alpha_{0}$ is given by
\[
\alpha_{0}=\max\left\{ \frac{-\mathsf{b}+\sqrt{\mathsf{b}^{2}-4\rho^{2}u\mathsf{g}}}{2\rho^{2}u},0\right\} .
\]
\end{proof}

\end{document}